\documentclass[10pt,oneside, draft]{amsart}
\usepackage[a4paper, textwidth=14cm]{geometry}
\usepackage[T1]{fontenc}
\usepackage[english]{babel}
\usepackage{xcolor} 
\usepackage{tikz} 
\usepackage[draft=false]{hyperref} 
\usepackage{mathtools}
\usepackage{amsmath,amsthm,amssymb}
\usepackage[all,cmtip]{xy}

\usepackage[capitalise]{cleveref}  
\newcommand{\clevertheorem}[3]{%
	\newtheorem{#1}[thm]{#2}
	\crefname{#1}{#2}{#3}
}
\numberwithin{equation}{section} 
\numberwithin{figure}{section} 

\theoremstyle{plain} 
\newtheorem{thm}{Theorem}[section]
\crefname{thm}{Theorem}{Theorems}
\newtheorem*{thm*}{Theorem}
\clevertheorem{proposition}{Proposition}{Propositions}
\clevertheorem{prop}{Proposition}{Propositions}
\newtheorem*{prop*}{Proposition}
\clevertheorem{theorem}{Theorem}{Theorems}
\clevertheorem{lemma}{Lemma}{Lemmas}

\clevertheorem{corollary}{Corollary}{Corollaries}
\clevertheorem{cor}{Corollary}{Corollaries}
\clevertheorem{conj}{Conjecture}{Conjectures}
\clevertheorem{questn}{Question}{Questions}

\theoremstyle{definition} 
\clevertheorem{definition}{Definition}{Definitions}
\clevertheorem{assumption}{Assumption}{Assumptions}
\clevertheorem{defn}{Definition}{Definitions}
\clevertheorem{notn}{Notation}{Notations}
\clevertheorem{conv}{Convention}{Conventions}

\clevertheorem{remark}{Remark}{Remarks}
\clevertheorem{rmk}{Remark}{Remarks}
\clevertheorem{warn}{Warning}{Warnings}
\clevertheorem{example}{Example}{Examples}
\clevertheorem{summ}{Summary}{Summaries}

\usepackage{amsmath,amsthm,amssymb,amsfonts,extarrows}
\usepackage{enumerate,amscd,amsxtra,MnSymbol}
\usepackage{mathrsfs}
\usepackage{color}
\usepackage[cmtip,all]{xy}
\usepackage{eucal}
\usepackage{bbold}
\usepackage{upgreek}
\usepackage{paralist}
\usepackage{tikz-cd}
\usepackage{dsfont}
\usepackage{bbm}
\usepackage[bbgreekl]{mathbbol}

\DeclareMathSymbol\bbDelta \mathord{bbold}{"01}
\DeclareMathSymbol\bDelta \mathord{bbold}{"01}

\newtheorem{remark*}{Remark}

\newtheorem{construction}[thm]{Construction}

\newtheorem{notation}[thm]{Notation}

\newcommand{\bE}{{\mathbb E}}

\renewcommand{\P}{{\mathbb P}}
\newcommand{\bQ}{{\mathbb Q}}

\newcommand{\bN}{{\mathbb N}}

\newcommand{\mB}{{\mathcal B}}
\newcommand{\mC}{{\mathcal C}}
\newcommand{\mD}{{\mathcal D}}
\newcommand{\mE}{{\mathcal E}}

\newcommand{\mK}{{\mathcal K}}
\newcommand{\mL}{{\mathcal L}}
\newcommand{\mM}{{\mathcal M}}

\newcommand{\mO}{{\mathcal O}}
\newcommand{\mP}{{\mathcal P}}
\newcommand{\mQ}{{\mathcal Q}}

\newcommand{\mS}{{\mathcal S}}

\newcommand{\mU}{{\mathcal U}}
\newcommand{\mV}{{\mathcal V}}

\newcommand{\A}{{\mathrm A}}
\newcommand{\B}{{\mathrm B}}
\newcommand{\C}{{\mathrm C}}
\newcommand{\D}{{\mathrm D}}

\newcommand{\F}{{\mathrm F}}
\newcommand{\G}{{\mathrm G}}
\newcommand{\rH}{{\mathrm H}}
\newcommand{\I}{{\mathrm I}}

\newcommand{\K}{{\mathrm K}}
\renewcommand{\L}{{\mathrm L}}
\newcommand{\M}{{\mathrm M}}

\renewcommand{\P}{{\mathrm P}}
\newcommand{\Q}{{\mathrm Q}}
\newcommand{\R}{{\mathrm R}}
\newcommand{\rS}{{\mathrm S}}
\newcommand{\T}{{\mathrm T}}

\newcommand{\V}{{\mathrm V}}
\newcommand{\W}{{\mathrm W}}
\newcommand{\X}{{\mathrm X}}
\newcommand{\Y}{{\mathrm Y}}
\newcommand{\Z}{{\mathrm Z}}

\newcommand{\rc}{{c}}
\newcommand{\rd}{{d}}

\newcommand{\bj}{{j}}
\newcommand{\bi}{{i}}
\newcommand{\m}{{m}}
\newcommand{\bk}{{k}}

\newcommand{\n}{{n}}
\newcommand{\br}{{r}}

\newcommand{\op}{\mathrm{op}}
\newcommand{\Ch}{\mathsf{Ch}}
\newcommand{\s}{s}
\newcommand{\colim}{\mathrm{colim}}
\newcommand{\Mod}{{\mathrm{Mod}}}
\newcommand{\LMod}{{\mathrm{LMod}}}
\newcommand{\RMod}{{\mathrm{RMod}}}
\newcommand{\BMod}{{\mathrm{BMod}}}

\newcommand{\Ass}{  {\mathrm {   Ass  } }   }

\newcommand{\Env}{{\mathrm{Env}}}  
\newcommand{\ot}{\otimes}

\newcommand{\co}{\mathrm{co}}

\newcommand{\id}{\mathrm{id}}
\newcommand{\Cat}{\mathrm{Cat}}

\newcommand{\Set}{\mathrm{Set}}

\renewcommand{\Pr}{\mathrm{Pr}}

\newcommand{\Coalg}{\mathrm{Coalg}}
\newcommand{\Alg}{\mathrm{Alg}}

\newcommand{\Comm}{\mathrm{Comm}}

\newcommand{\Mon}{\mathrm{Mon}}
\newcommand{\Fun}{\mathrm{Fun}}
\newcommand{\Op}{{\mathrm{Op}}}

\newcommand{\tu}{{\mathbb 1}}

\newcommand{\Ind}{{\mathrm{Ind}}}

\newcommand{\bZ}{{\mathbb{Z}}}

\newcommand{\cofib}{\mathrm{cofib}}

\newcommand{\fib}{\mathrm{fib}}

\newcommand{\Sp}{{\mathrm{Sp}}} 
\newcommand{\Vect}{{\mathrm{Vect}}} 

\newcommand{\map}{{\mathrm{map}}} 

\newcommand{\Fin}{{\mathrm{Fin}}}

\newcommand{\Sym}{{\mathrm{Sym}}}

\newcommand{\f}{{\mathrm{f}}}

\newcommand{\Ho}{{\mathrm{Ho}}}

\newcommand{\Grp}{{\mathrm{Grp}}}

\newcommand{\coComm}{{\mathrm{coComm}}}  
\newcommand{\sSeq}{{\mathrm{sSeq}}}  

\newcommand{\coOp}{{\mathrm{coOp}}}  
  
\newcommand{\nun}{{\mathrm{nun}}}  
\newcommand{\Barc}{{\mathrm{Barc}}}  
\newcommand{\Cobar}{{\mathrm{Cobar}}}  
\newcommand{\Bialg}{{\mathrm{Bialg}}}  
\newcommand{\pr}{{\mathrm{pr}}}  

\newcommand{\Lie}{{\mathrm{Lie}}}  
\newcommand{\Hopf}{{\mathrm{Hopf}}}  
\newcommand{\Prim}{{\mathrm{Prim}}}  
\newcommand{\triv}{{\mathrm{triv}}}  
\newcommand{\conil}{{\mathrm{conil}}}  
\newcommand{\LM}{{\mathrm{LM}}}  
\newcommand{\RM}{{\mathrm{RM}}}  
\newcommand{\BM}{{\mathrm{BM}}}  

\usepackage{pdfpages}

\author{Hadrian Heine}

\begin{document}
	
\title{A derived Milnor-Moore theorem}
	
\maketitle 

\begin{abstract}
	
For every stable presentably symmetric monoidal $\infty$-category
$\mC$ we use the Koszul duality between the spectral Lie operad and the cocommutative cooperad
to construct an enveloping Hopf algebra functor
$\mU: \Alg_\Lie(\mC) \to \Hopf(\mC)$ from Lie algebras in $\mC$ to
cocommutative Hopf algebras in $\mC$ left adjoint to a functor of derived primitive elements $\Prim$.
We study the unit of this adjunction in rational and chromatic homotopy theory:
we prove that if $\mC$ is a rational stable presentably symmetric monoidal $\infty$-category, the enveloping Hopf algebra functor $\mU: \Alg_\Lie(\mC) \to \Hopf(\mC)$ is fully faithful
reproving a result of Gaitsgory-Rozenblyum.
Let $n \geq 1 $ be a natural and $\Phi[-1]: \mS_{v_n} \to \Alg_\Lie(\Sp_{T_n})$ the shifted Bousfield-Kuhn functor from $v_n$-periodic homotopy types to spectral Lie algebras in $T_n$-local spectra.
We prove that for every $v_n$-periodic homotopy type $X$ the unit $\Phi(X)[-1] \to \Prim \mU(\Phi(X)[-1])$ identifies with the Goodwillie completion
$ \Phi \to \lim_{n \geq 0} P_n(\Phi)$ evaluated at the loop space of $X.$

\end{abstract}

\tableofcontents

\section{Introduction}

In the seminal paper \cite{10.2307/1970615} Milnor and Moore prove that
every graded Lie algebra over a field of characteristic zero
arises as the primitive elements of its universal enveloping Hopf algebra
\cite[Proposition 5.17.]{10.2307/1970615}.
From this result Milnor and Moore reveal a fundamental connection between Lie algebras and Hopf algebras over any field $K$ of characteristic zero.
Forming the universal enveloping Hopf algebra and forming the primitive elements 
are part of an adjunction 
\begin{equation}\label{adj}
\mU: \Lie_K \rightleftarrows \Hopf_K: \Prim 
\end{equation} between the categories of graded Lie algebras and graded Hopf algebras over $K$.
By the Milnor-Moore theorem the unit of this adjunction is invertible,
which implies that the enveloping Hopf algebra functor $\mU$ is fully faithful.
Moreover Milnor-Moore \cite[Theorem 5.18.]{10.2307/1970615} characterize the Hopf algebras in the essential image of the enveloping Hopf algebra $\mU$ as the primitively generated Hopf algebras, those graded Hopf algebras whose underlying graded algebra is generated by the subspace of primitive elements.


The Milnor-Moore theorem has significant importance for rational homotopy theory.
Quillen \cite{Mont} assigns to any simply connected space $X$ a rational differentially graded Lie algebra $\mL(X)$, the Lie model of the space, and a rational differentially graded cocommutative coalgebra $\mC(X)$, the coalgebra model of the space, where the Lie model computes the rational homotopy and the coalgebra model computes the rational homology.
More precisely, the coalgebra model endows the rational chains on a simply connected space with the structure of a differentially graded cocommutative coalgebra inducing on homology the graded coalgebra coming from the diagonal of the space, while the Lie model induces on homology the graded Lie structure on the rational homotopy groups of a loop space given by the Whitehead product.
In other words $$H_*(\mC(X)) \cong H_*(X;\bQ), H_*(\mL(X)) \cong \pi_{*}(\Omega(X)) \ot \bQ.$$


Quillen \cite{Mont} proves that the rational homotopy type of a space is determined by either its Lie model or coalgebra model:
he constructs a commutative triangle of equivalences:
\begin{equation}\label{tri}
\begin{xy}
\xymatrix{
& \ar[ld]_\mC
\ar[rd]^\mL
\Ho(\mS^{\geq2}_{\bQ})
\\
\Ho(\co\Alg(\bQ)^{\geq 2}) \ar[rr] && \Ho(\Lie(\bQ)^{\geq 1})
}
\end{xy} 
\end{equation}
between the homotopy categories of simply connected rational spaces,
simply connected rational differentially graded cocommutative coalgebras and connected rational differentially graded Lie algebras.

Quillen builds the Lie model of a simply connected space $X$
in three steps.
First he forms Kan's loop group, a simplicial group that endows the loop space $\Omega(\X)$ with a group structure. As a second step he applies rational chains to Kan's loop group to obtain a simplicial rational cocommutative Hopf algebra.
Finally, he takes the primitive elements of this simplicial Hopf algebra to obtain a 
simplicial rational Lie algebra, which corresponds to a rational differentially graded Lie algebra via Dold-Kan's correspondence.
Quillen's construction guarantees that the homology of the Lie model
identifies with the primitive elements of the rational homology of the loop space, in symbols $H_*(\mL(X)) \cong \Prim(H_*(\Omega(X);\bQ)).$
Since $X$ is simply connected, $H_*(\Omega(X);\bQ)$ is connected, which ensures that the graded Hopf algebra $H_*(\Omega(X);\bQ)$ is primitively generated. 
As a consequence the Milnor-Moore theorem identifies the enveloping Hopf algebra 
of $$ \pi_{*}(\Omega(X)) \ot \bQ \cong H_*(\mL(X)) \cong \Prim(H_*(\Omega(X);\bQ))$$
with $H_*(\Omega(X);\bQ)$
and so reveals a fundamental relationship between rational homotopy and rational homology of a loop space.



From modern perspective Quillen's equivalence (\ref{tri}) between the Lie model and the coalgebra model is an instance of Koszul duality. Koszul duality between Lie algebras and cocommutative coalgebras
or rather the homotopically refined notions of $L_\infty$-algebras and $\bE_\infty$-coalgebras arises from the operadic Koszul duality between the shifted Lie operad $\Lie(1)$ and the non-unital cocommutative cooperad \cite[Theorem 6.8.]{Koszul}. It takes the form of an adjunction 
\begin{equation}\label{ad1} \Alg_{\Lie(1)}(\bQ) \rightleftarrows \co\Alg_{\bE_\infty}(\bQ): \Prim
\end{equation}
between rational $\Lie(1)$-algebras and rational coaugmented $\bE_\infty$-coalgebras, where the right adjoint assigns a derived version of primitive elements and the left adjoint 
takes the shifted Chevalley-Eilenberg complex computing Lie homology.
By a theorem of Ching-Harper \cite[Theorem 1.2.]{CHING2019118} adjunction (\ref{ad1}) restricts to an equivalence between simply connected objects.
Since a $\Lie(1)$-algebra structure on a rational chain complex is a $L_\infty$-structure on the negative shift, adjunction  (\ref{ad1}) restricts to an equivalence.
\begin{equation}\label{ad16} \Alg_{\Lie}(\bQ)^{\geq 1} \simeq \co\Alg_{\bE_\infty}(\bQ)^{\geq 2}
\end{equation}
between connected rational $L_\infty$-algebras and simply connected rational coaugmented $\bE_\infty$-coalgebras, recovering Quillen's result via the machinery of Koszul duality.

The interpretation of equivalence (\ref{tri}) 
as Koszul duality 
suggests a construction of the derived universal enveloping Hopf algebra purely expressed in terms of Koszul duality:
derived Hopf algebras are group structures on $\bE_\infty$-coalgebras and 
adjunction (\ref{ad1}) descends to an adjunction on group objects
\begin{equation}\label{adul9}
\Grp(\Alg_{\Lie(1)}(\bQ)) \rightleftarrows \Hopf(\bQ):= \Grp(\co\Alg_{\bE_\infty}(\bQ)): \Prim.
\end{equation}
By Corollary \ref{qqqp} every group object in the $\infty$-category of $\Lie(1)$-algebras admits a unique delooping identifying group objects in $\Lie(1)$-algebras with 
$L_\infty$-algebras.
This way we can identify the left hand side of adjunction (\ref{adul9}) with $L_\infty$-algebras
and obtain an adjunction
\begin{equation}\label{ad2}
	\mU: \Alg_\Lie(\bQ) \rightleftarrows \Hopf(\bQ): \Prim.
\end{equation}
Inspired by the classical situation 
it seems reasonable to view the left adjoint $\mU$ of the functor of derived primitive elements as a model for the functor of derived universal enveloping Hopf algebra. 
This model was identified by Gaitsgory-Rozenblyum with the classical definition \cite[Theorem 6.1.2.]{MR3701353} and generalized by Brantner-Camarena-Heuts \cite{BrantnerCamarenaHeuts},
Knudsen \cite{2016arXiv160501391K} and Shi \cite{Shi} to stable homotopy theory to produce enveloping $\bE_\n$-algebras replacing loops by n-fold loops.

Motivated by these results we apply Koszul duality to construct and study a derived universal enveloping Hopf algebra for algebras over Ching's spectral Lie operad \cite{Ching} in any stable presentably symmetric monoidal $\infty$-category.
Lurie \cite[5.2.]{lurie.higheralgebra} constructs Koszul duality between augmented associative algebras and coaugmented coassociative coalgebras in any monoidal $\infty$-category that admits the necessary limits and colimits to perform the construction. Applied to the monoidal $\infty$-category of symmetric sequences Lurie's Koszul duality
gives rise to a Koszul duality between augmented non-unital $\infty$-operads and coaugmented non-counital $\infty$-cooperads \cite[§4.3.]{articles}.
Based on this operadic Koszul duality Amabel \cite{Amab}, Branter-Campos-Nuiten \cite{https://doi.org/10.48550/arxiv.2104.03870},
Ching-Harper \cite{CHING2019118} and Francis-Gaitsgory \cite{articlet} describe and study Koszul duality between algebras over any augmented non-unital $\infty$-operad $\mO$ and coalgebras over the Koszul dual $\infty$-cooperad 
$\mO^\vee$ in any stable presentably symmetric monoidal $\infty$-category $\mC.$
This Koszul duality takes the form of an adjunction
\begin{equation}\label{addil}
\mathrm{TQ}_\mO: \Alg_\mO(\mC) \rightleftarrows \co\Alg^{\mathrm{dp}, \hspace{0,2mm} \conil}_{\mO^\vee}(\mC) \end{equation}
between the $\infty$-category of $\mO$-algebras in $\mC$ and the $\infty$-category of conilpotent $\mO^\vee$-coalgebras that are equipped with divided powers.
The left adjoint takes topological Quillen homology.
Forgetting divided powers and conilpotence gives an adjunction
\begin{equation*}\label{addi}
\mathrm{TQ}_\mO: \Alg_\mO(\mC) \rightleftarrows \co\Alg_{\mO^\vee}(\mC) : \Prim_\mO \end{equation*}
between the $\infty$-category of $\mO$-algebras in $\mC$ and the $\infty$-category of $\mO^\vee$-coalgebras in $\mC$, where the right adjoint takes a derived variant of primitive elements.
Specializing to the Koszul duality \cite{Ching} between the shifted spectral Lie operad $\Lie(1)$ and the non-unital cocommutative cooperad the latter adjunction
specializes to an adjunction
\begin{equation}\label{adul}
\mathrm{TQ}_{\Lie(1)}: \Alg_{\Lie(1)}(\mC) \rightleftarrows \co\Alg_{\bE_\infty}(\mC): \Prim_{\Lie(1)} \end{equation}
between $\Lie(1)$-algebras and coaugmented $\bE_\infty$-coalgebras generalizing adjunction (\ref{ad1}) to stable homotopy theory.
By Corollary \ref{qqqp} every group object in the $\infty$-category of $\Lie(1)$-algebras in $\mC$ admits a unique delooping identifying group objects in $\Lie(1)$-algebras in $\mC$ with 
$L_\infty$-algebras in $\mC.$
So adjunction (\ref{adul}) induces on group objects an adjunction
\begin{equation}\label{ehbp}
\mU:= \mathrm{TQ}_{\Lie(1)}: \Alg_\Lie(\mC) \simeq \Grp(\Alg_{\Lie(1)}(\mC)) \rightleftarrows$$$$ \Hopf(\mC):= \Grp(\co\Alg_{\bE_\infty}(\mC)) : \Prim:=\Prim_{\Lie(1)}.
\end{equation}	
The latter adjunction extends adjunction (\ref{ad2}) to stable homotopy theory
and gives a model for the derived universal enveloping Hopf algebra.

In this article we apply Koszul duality to deduce a derived version of the Milnor-Moore theorem. 
Harper-Hess \cite{HarperHess} construct for every spectral $\infty$-operad $\mO$,
where $\mO_0=0, \mO_1=S,$ a tower of $(\mO, \mO)$-bimodules $$ ... \to \tau_{\n}(\mO) \to \tau_{\n-1}(\mO) \to ... \to \tau_1(\mO)=\triv $$ consisting of operadic truncations $\mO \to \tau_{\n}(\mO)$,
where $\tau_{\n}(\mO)_\ell:=\left\{
\begin{array}{ll}
\mO_\ell, & \ell \leq \n \\
0, &  \textrm{else}. \\
\end{array}
\right.$
The latter tower gives for every $\mO$-algebra $\X$ a tower of $\mO$-algebras
$$ ... \to \tau_{\n}(\mO) \circ_\mO \X \to \tau_{\n-1}(\mO)\circ_\mO \X  \to ... \to \triv \circ_\mO \X \simeq \mathrm{TQ}_\mO(\X),$$
where $\circ$ is the composition product, interpolating between $\mathrm{TQ}_\mO(\X)$ and the $\mathrm{TQ}_\mO$-completion
$\mathrm{TQ}_\mO(X)^\wedge := \lim_{\n \geq 1}  \tau_{\n}(\mO) \circ_\mO \X.$
We prove the following theorem:
\begin{theorem}\label{0}(Theorem \ref{map2})
Let $\mC$ be a stable presentably symmetric monoidal $\infty$-category such that all norm maps of objects in $\mC$ with an action of a symmetric group are equivalences.
Let $\X$ be a Lie algebra in $\mC$.
The unit $X \to \Prim \mU(X)$ identifies with the completion map
$$ X \to \mathrm{TQ}_\Lie(X)^\wedge.$$
	
\end{theorem}

Before we explain how we prove \cref{0}, we explain important corollaries.
We apply Theorem \ref{0} to rational and chromatic homotopy theory.
Rationally, we immediately obtain the following derived Milnor-Moore theorem that was proven by  
Gaitsgory-Rozenblyum \cite[Theorem 4.4.6.]{MR3701353} via an elaborate filtration of the primitive elements.

\begin{theorem}\label{uuum}(Theorem \ref{thg})
Let $\mC $ be a $\bQ$-linear stable presentably symmetric monoidal $\infty$-category. The enveloping Hopf algebra functor $$\mU: \Alg_{\Lie}(\mC) \to \Hopf(\mC)$$ is fully faithful.
	
\end{theorem}

Moreover we explain how to deduce the classical Milnor-Moore theorem (Theorem \ref{Mil}) from Theorem \ref{uuum}, which gives an alternative proof of the classical version.

We apply Theorem \ref{0} to chromatic homotopy theory.
Chromatic homotopy theory \cite{DevinatzHopkinsSmith}, \cite{HopkinsSmith}, \cite{Ravenel},  \cite{Ravenel2},  \cite{Ravenel3} is an extension of rational homotopy.
For a fixed prime $p$ and height $n \geq 1$ one considers $v_n$-periodic homotopy groups, an extension of rational homotopy groups that captures periodic phenomena in homotopy groups.
Extending the concept of rational equivalences chromatic homotopy theory studies
$v_n$-periodic equivalences, those maps of pointed $p$-local spaces inducing an equivalence on $v_n$-periodic homotopy groups.
Extending the concept of rational spaces chromatic homotopy theory studies 
$p$-local pointed spaces up to $v_n$-periodic equivalences, i.e. the Bousfield localization $\mS_{(p)} \rightleftarrows \mS_{v_n}$ of the category $ \mS_{(p)}$ of pointed $p$-local spaces with respect to the class of $v_n$-periodic equivalences. The latter exists by \cite[Theorem 2.2.]{2018arXiv180306325H}.
Similarly, one considers the stable homotopy theory of $v_n$-periodic homotopy types,
which by \cite[Remark 3.20.]{2018arXiv180306325H}  is also known as the homotopy theory of $T_n$-local spectra .
The $v_n$-periodic homotopy groups are elegantly described via the Bousfield-Kuhn functor \cite{Bousfield1}, \cite{Bousfield2}, \cite{Bousfield3}, which is a key player in chromatic homotopy theory \cite{Kuhn}, \cite{behrensrezk}, \cite{2018arXiv180306325H}.

The latter is a functor $\Phi: \mS_{(p)} \to \Sp_{T_n}$ from $p$-local pointed spaces to $T_n$-local spectra that assigns to a $p$-local space $X$ a $T_n$-local spectrum $\Phi(X)$ whose homotopy groups are the $v_n$-periodic homotopy groups of $X$, and such that for every $p$-local spectrum $Y$ the $T_n$-local spectrum $\Phi(\Omega^\infty(Y))$
is the $T_n$-localization of $Y$ (see \cite[Part I]{Shi2} for details).

Heuts \cite[Theorem 2.6.]{2018arXiv180306325H} extends Quillen's rational homotopy theory \cite{Mont} to higher chromatic heights. Heuts proves that every $v_n$-periodic homotopy type $X$ is determined by a shifted Lie structure on the $T_n$-local spectrum $\Phi(X) $ in analogy that every simply connected rational space is determined by a rational differently graded Lie algebra.
Precisely, Heuts proves that the Bousfield-Kuhn functor $\Phi: \mS_{v_n} \to \Sp_{T_n}$ lifts to an equivalence $ \mS_{v_n} \to \Alg_{\Lie(1)}(\Sp_{T_n})$.
Composing with the equivalence $\Alg_{\Lie(1)}(\Sp_{T_n}) \simeq \Alg_{\Lie}(\Sp_{T_n})$ that loops the underlying spectrum, one obtains an equivalence
$\Phi[-1] : \mS_{v_n} \to \Alg_{\Lie}(\Sp_{T_n})$.

We deduce the following chromatic variant of the Milnor-Moore theorem: 

\begin{theorem}\label{2} (Theorem \ref{20})
Let $n$ be a natural and $X$ a $v_n$-periodic homotopy type.
The unit $$\Phi(X)[-1] \to \Prim \mU(\Phi(X)[-1])$$ identifies with the Goodwillie completion
$ \Phi \to \lim_{n \geq 0} P_n(\Phi)$ evaluated at $\Omega(X)$.

\end{theorem}

In view of Theorem \ref{2} the question arises which $v_n$-periodic homotopy types have a $\Phi$-good loop space. This is in general a hard question. However 
it is expected that any $v_n$-periodic homotopy type has a $\Phi$-good loop space for any prime and chromatic height $n$ (see \cite{behrens2024unstable} last paragraph).
See \cite{KjaerUnstable} for examples.

\vspace{1mm}

We deduce Theorem \ref{0} from the following theorem:
\begin{theorem}\label{tttu}(Theorem \ref{map})
Let $\mC$ be a stable presentably symmetric monoidal $\infty$-category, $\mO$ a non-unital $\infty$-operad in $\mC$ such that $\mO_1 = \tu $ and $\X$ an $\mO$-algebra in $\mC$.
There is a commutative square:
$$\begin{xy}
\xymatrix{\X \ar[d] 
\ar[rr]
&& \Prim_\mO \T\Q_\mO (\X) \ar[d] 
\\
\lim_{\n \geq 1}(\tau_\n(\mO) \circ_{\mO} \X) \ar[rr] && \lim_{\n \geq 1}((\tau_\n(\mO) \circ_{\mO} \triv) *^{\mO^\vee} (\triv \circ_\mO \X)),
}
\end{xy} $$
where $*$ is the cocomposition product.
If all norm maps of objects in $\mC$ with an action of a symmetric group are equivalences,
the right vertical morphism and bottom horizontal morphism are equivalences.
\end{theorem}

To prove Theorem \ref{tttu} we analyze Koszul duality and the interaction between the composition product and its dual version, the cocomposition product.
Although Theorem \ref{tttu} requires to assume that all norm maps of objects with an action of a symmetric group are invertible, a condition satisfied in all our examples of interest, we recently learned that this condition is not logically necessary but only necessary to apply our techniques:
Heuts \cite[Theorem 2.1.]{Heuts} recently proved a generalization of Theorem \ref{tttu} without assuming this condition but using different techniques.

\subsection{Acknowledgements}

We thank Lukas Brantner, Gijs Heuts, Ben Knudsen, Markus Spitzweck and Manfred Stelzer for helpful discussions.
We are especially grateful to Gijs Heuts for communicating to us the proof of Lemma \ref{calc}.
This work are main parts of my thesis written during my PhD 2014-2018 at the university of Osnabr\"uck. I am grateful for this opportunity. 

\subsection{Notation and Terminologie} 


We fix a hierarchy of Grothendieck universes whose objects we call small, large, etc.
We call a space small, large, etc. if its set of path components and its homotopy groups are so for any choice of base point. We call an $\infty$-category small, large, etc if its maximal subspace and all its mapping spaces are small, large, respectively.

We write 
\begin{itemize}
\item $\Set$ for the category of small sets.
\item $\Delta$ for the category of finite, non-empty, partially ordered sets and order preserving maps, whose objects we denote by $[\n] = \{0 < ... < \n\}$ for $\n \geq 0$.
\item $\Fin$ for (a skeleton of) the category of finite sets, whose objects we denote by $\n = \{1, ... , \n\}$ for $\n \geq 0$.
\item $\Fin_*$ for (a skeleton of) the category of finite pointed sets, whose objects we denote by $\langle\n\rangle = \{*, 1, ..., \n\}$ for $\n \geq 0$.
\item $\mS$ for the $\infty$-category of small spaces.
\item $ \Cat_\infty$ for the $\infty$-category of small $\infty$-categories.
\item $\Cat_\infty^{\rc \rc} $ for the $\infty$-category of large $\infty$-categories with small colimits and small colimits preserving functors.
\end{itemize}

We often indicate $\infty$-categories of large objects by $\widehat{(-)}$, for example we write $\widehat{\mS}, \widehat{\Cat}_\infty$ for the $\infty$-categories of large spaces, $\infty$-categories.


\vspace{1mm}

For any $\infty$-category $\mC$ containing objects $\A, \B$ we write
\begin{itemize}
\item $\mC(\A,\B)$ for the space of maps $\A \to \B$ in $\mC$,
\item $\mC_{/\A}$ 
for the $\infty$-category of objects over $\A$,
\item $\mC^\simeq $ for the maximal subspace in $\mC$.
\item 
$\Fun(\mC,\mD)$ for the $\infty$-category of functors $\mC \to \mD$ to any $\infty$-category $\mD.$
\end{itemize}



We heavily use symmetric $\infty$-operads \cite[Definition 2.1.1.10.]{lurie.higheralgebra} and symmetric monoidal $\infty$-categories \cite[Definition 2.0.0.7.]{lurie.higheralgebra}.
For every $\infty$-operad $\mO^\ot \to \Fin_*$ we write $\mO$ for the fiber over $\langle 1 \rangle \in \Fin_*$ and call $\mO$ the underlying $\infty$-category.
More generally, we use $\mO$-operads and $\mO$-monoidal $\infty$-categories \cite[Definition 2.1.2.13.]{lurie.higheralgebra},
which for $\mO^\ot=\Fin_*$ give 
symmetric $\infty$-operads and symmetric monoidal $\infty$-categories, for $\mO^\ot=\Ass$
give non-symmetric $\infty$-operads and monoidal $\infty$-categories and for $\mO^\ot=\LM,\RM,\BM$ give $\LM,\RM,\BM$-operads and left, right, bitensored $\infty$-categories. 

\begin{notation}
For any $\mO$-operad $\mD^\ot \to \mO^\ot$
we write $\Alg_\mO(\mD)$ for the $\infty$-category of $\mO$-algebras in $\mD.$
\begin{itemize}
\item For $\mO^\ot=\Fin_*$ we set $\Alg_{\bE_\infty}(\mD):=\Alg_{\Fin_*}(\mD)$.
\item For $\mO^\ot=\Ass$ we set $\Alg(\mD):=\Alg_{\Ass}(\mD)$.
\item If $\mD^\ot \to \Fin_*$ is a symmetric monoidal $\infty$-category, let
$\co\Alg_{\bE_\infty}(\mD):=\Alg_{\bE_\infty}(\mD^\op)^\op$.
\end{itemize}
\end{notation}




\vspace{2mm}



\section{$\infty$-operads}

 
\subsection{The composition product}\label{Ooper}

In this section we define the composition product following \cite[Definition 3.3.]{https://doi.org/10.48550/arxiv.2104.03870}.
	
\begin{notation}Let $\Sigma \simeq  \coprod_{\n \geq 0} \B\Sigma_\n$ be the groupoid of finite sets and bijections.
\end{notation}
	
\begin{notation}
For every $\infty$-category $\mC$ let $$\sSeq(\mC):=\Fun(\Sigma, \mC)\simeq \prod_{\n \geq 0} \Fun(\B\Sigma_\n, \mC)$$
be the $\infty$-category of symmetric sequences in $\mC$.
\end{notation}

For every symmetric monoidal $\infty$-category $\mC$ compatible with small colimits
the $\infty$-category $\sSeq(\mC)$ carries a symmetric monoidal structure compatible with small colimits given by Day-convolution \cite[Proposition 4.8.1.10.]{lurie.higheralgebra}, where we use the symmetric monoidal structure on $\Sigma$ taking disjoint union.
For symmetric sequences $\X, \Y $ in $\mC$ we have
$$(\X \ot \Y)_\n:= \colim_{\bi \coprod \bj=\n} \X_\bi \ot \Y_\bj.$$
There is a symmetric monoidal embedding $\iota: \mC \hookrightarrow \sSeq(\mC) $ left adjoint to evaluation at the initial object that views an object of $\mC$ as a symmetric sequence concentrated in degree zero. 
\begin{remark}
By \cite[Proposition 4.8.1.17.]{lurie.higheralgebra} there is a symmetric monoidal equivalence $\sSeq(\mC)\simeq \mC \ot \sSeq(\mS)$ and $\iota$ is induced by the similarly defined embedding $\iota: \mS \to \sSeq(\mS).$
\end{remark}
\begin{notation}
Let $\triv$ be the symmetric sequence in $\mC$ concentrated in degree 1 with value the tensor unit of $\mC. $ 	
\end{notation}

\begin{lemma}Let $\mC, \mD$ be symmetric monoidal $\infty$-categories compatible with small colimits and $\mC \to \mD$ a symmetric monoidal functor preserving small colimits. The functor 
$$\Psi: \Fun^{\ot, \L}_{\mC/}(\sSeq(\mC), \mD) \to \mD $$ evaluating at $\triv$ is an equivalence, where the left hand side is the $\infty$-category of small colimits preserving symmetric monoidal functors under $\mC.$

\end{lemma}

\begin{proof}
Observe that $\Sigma$ is the free symmetric monoidal $\infty$-category generated by a point. Thus $\sSeq(\mS) \simeq \Fun(\Sigma^\op, \mS) $ endowed with Day convolution is the free symmetric monoidal $\infty$-category compatible with small colimits generated by a point  \cite[Proposition 4.8.1.10.]{lurie.higheralgebra}.
Tensoring with $\mC$ we find that $ \mC \ot \sSeq(\mS) \simeq \sSeq(\mC)$ is the free symmetric monoidal $\infty$-category compatible with small colimits under $\mC$ generated by a point. So the composition 	
$$\Psi: \Fun^{\ot, \L}_{\mC/}(\sSeq(\mC), \mD) \simeq  \Fun^{\ot, \L}(\sSeq(\mS), \mD) \simeq \Fun^{\ot}(\Sigma, \mD) \simeq \mD $$ evaluating at $\triv$ is an equivalence.
	
\end{proof}

\begin{definition}
The composition product monoidal structure on $\sSeq(\mC) $ is the monoidal structure opposite to composition via the equivalence
$$\Psi: \Fun^{\ot, \L}_{\mC/}(\sSeq(\mC), \sSeq(\mC)) \simeq \sSeq(\mC) $$ evaluating at $\triv$. For $\X, \Y \in \sSeq(\mC) $ we write $\X \circ \Y $ for the composition product.
\end{definition}

\vspace{0,1mm}

\begin{remark}\label{dghfgh}
By definition $\triv$ becomes the tensor unit of the composition product. For every $\X \in \sSeq(\mC)$ there is a canonical equivalence $\X \simeq \coprod_{\bk \geq 0} \iota(\X_\bk) \ot_{\Sigma_\bk} \triv^{\ot \bk}. $ 

So the composition product of $\X, \Y \in \sSeq(\mC) $ is $$\X \circ \Y 
\simeq (\Psi^{-1}(\Y) \circ \Psi^{-1}(\X))(\triv) \simeq \Psi^{-1}(\Y) (\X) \simeq $$$$\Psi^{-1}(\Y)(\coprod_{\bk \geq 0} \iota(\X_\bk) \ot_{\Sigma_\bk} \triv^{\ot \bk}) \simeq \coprod_{\bk \geq 0} \iota(\X_\bk) \ot_{\Sigma_\bk} \Y^{\ot \bk}. $$ 

Thus for every $\n \in \Sigma $ there is an equivalence
$$ (\X \circ \Y ) _\n \simeq \underset{\bk \geq 0}{\coprod} ( \underset{ {\n_1 \coprod ... \coprod \n_\bk = \n} }{\colim} \X_\bk \ot( \bigotimes_{1 \leq \bj \leq \bk } \Y_{\n_\bj }))_{\Sigma_\bk} .$$

Another way to express the latter is the following one:

$$ (\X \circ \Y ) _\n \simeq  \underset{(\f: \n \to \bk) \in (\Fin_{\n/})^\cong}{\colim} \X_\bk \ot \bigotimes_{\bj \in \bk} \Y_{\f^{-1}(\bj)}.$$

The latter expression generalizes the following way: for $\ell \geq 2$ let $\Fin^\ell_\n$ be the groupoid
whose objects are sequences of maps of finite sets $\n \xrightarrow{\f_{\ell-1}} \n_{\ell-1} \xrightarrow{\f_{\ell-2}} ... \xrightarrow{\f_1} \n_1$ of length $\ell-1$ and the evident isomorphisms.	
For any $\mO_1, ..., \mO_\ell \in \sSeq(\mC)$ there is an equivalence 
$$(\mO_1 \circ ... \circ \mO_\ell)_\n \simeq \colim_{\f \in \Fin^\ell_\n}((\mO_1)_{\n_1} \ot \bigotimes_{\bi \in \n_1}(\mO_2)_{ \f_1^{-1}(\bi)} \ot ... \ot \bigotimes_{\bi \in \n_{\ell-1}}(\mO_\ell)_{\f_{\ell-1}^{-1}(\bi)})$$ 
(see \cite[Definition 2.12.]{2005math.....10490C}).

\end{remark}

\begin{notation}
Let $\Sigma_{ \geq 1}:= \coprod_{\n \geq 1} \B\Sigma_\n$ and for every $\infty$-category $\mC$ let $$\sSeq(\mC)_{\geq 1} := \Fun(\Sigma_{\geq 1}, \mC) \simeq \prod_{\n \geq 1} \Fun(\B\Sigma_\n, \mC).$$
\end{notation}

\begin{construction}\label{leftact}
There is an embedding $\sSeq(\mC)_{\geq 1} \subset \sSeq(\mC) $ left adjoint to restriction along the embedding $\Sigma_{\geq 1} \subset \Sigma$ that inserts the initial object in degree zero. If the symmetric monoidal structure on $\mC$ is compatible with the initial object, the composition product on $\sSeq(\mC)$ restricts to $\sSeq(\mC)_{\geq 1}.$
For every $\X \in \sSeq(\mC)$ and $\Y \in \mC$ the composition product
$$\X \circ \iota(\Y) \simeq \coprod_{\bk \geq 0} \iota(\X_\bk) \ot_{\Sigma_\bk} \iota(\Y)^{\ot \bk} \simeq \iota(\coprod_{\bk \geq 0} \X_\bk \ot_{\Sigma_\bk} \Y^{\ot \bk}) $$ belongs to the essential image of $\iota:\mC \to \sSeq(\mC).$
Thus the composition product gives rise to a left action of
$\sSeq(\mC)$ on $\mC.$
\end{construction}

\begin{remark}\label{remf}
	
Let $\phi: \mC \to \mD$ be a small colimits preserving symmetric monoidal functor between symmetric monoidal $\infty$-categories compatible with small colimits. There is a symmetric monoidal equivalence $$\mD \ot_\mC \sSeq(\mC)\simeq \mD \ot_\mC \mC \ot \sSeq(\mS)\simeq \mD \ot \sSeq(\mS)\simeq \sSeq(\mD).$$

Since the functor $\mD \ot_\mC (-): \Alg_{\bE_\infty}(\Cat_\infty^{\rc\rc})_{\mC/} \to \Alg_{\bE_\infty}(\Cat_\infty^{\rc\rc})_{\mD/}$ is lax $\Cat_\infty^{\rc\rc}$-linear, it yields a monoidal functor on composition products:
$$ \xi: \sSeq(\mC) \simeq \Fun^{\ot, \L}_{\mC/}(\sSeq(\mC), \sSeq(\mC)) \to \sSeq(\mD) \simeq \Fun^{\ot, \L}_{\mC/}(\sSeq(\mD), \sSeq(\mD)).$$

\end{remark}


\vspace{1mm}

We will often consider the composition product in case that $\mC$ has small colimits but the tensor product of $\mC$ does not preserve small colimits
component-wise.
In this case the composition product does not define a monoidal $\infty$-category but a weaker structure based on the notion of non-symmetric $\infty$-operad \cite[Definition 4.1.3.2.]{lurie.higheralgebra}:


\begin{definition}
	
A lax monoidal $\infty$-category is a non-symmetric $\infty$-operad
$\mV^\ot \to \Ass$ such that $\mV^\ot \to \Ass$ is a locally cocartesian fibration and the full suboperad spanned by the tensor unit is a monoidal $\infty$-category.	

Let $\mM^\ot \to \BM^\ot$ be a $\BM$-operad that exhibits an $\infty$-category $\mD$ as weakly bitensored over lax monoidal $\infty$-categories $\mC, \mE$.
We say that $\mM^\ot \to \BM^\ot$ exhibits $\mD$ as lax bitensored over $\mC, \mE$ 
if the functor $\mM^\ot \to \BM^\ot$ is a locally cocartesian fibration and the restricted weak biaction of the monoidal subcategories of $\mC, \mE$ spanned by the tensor unit is a bitensored $\infty$-category.	

\end{definition}

\begin{remark}
	
Let $\mV^\ot \to \Ass$ be a non-symmetric $\infty$-operad and locally cocartesian fibration.
Then for every $\V_1, ..., \V_\n$ there is a tensor product
$\V_1 \ot ... \ot \V_\n \in \mV$.

Moreover for every $\bk \geq0$, $\n_1, ..., \n_\bk \geq 0$
and $\V^\bi_1, ..., \V^\bi_{\n_\bi} \in \mV$ for $1 \leq \bi \leq \bk$
there is a canonical morphism in $\mV:$
$$\theta:  \V^1_1 \ot ... \ot \V^1_{\n_1} \ot ... \ot \V^\bk_1 \ot ... \ot \V^\bk_{\n_\bk} \to (\V^1_1 \ot ... \ot \V^1_{\n_1}) \ot ... \ot (\V^\bk_1 \ot ... \ot \V^\bk_{\n_\bk}).$$
The functor $\mV^\ot \to \Ass$ is a monoidal $\infty$-category if and only if $\theta$ is an equivalence for all choices of inputs.
The functor $\mV^\ot \to \Ass$ exhibits $\mV$ as a lax monoidal
$\infty$-category if and only if $\theta$ is an equivalence
for $\V^\bi_1 =... = \V^\bi_{\n_\bi} =\tu$ for $1 \leq \bi \leq \bk$.
In particular, for $\n_1=...=\n_\bk=0$ we find that
the canonical morphism $\tu \to \tu^{\ot \bk}$ in $\mV$ is an equivalence.

\end{remark}

\begin{proposition}\label{weak}
Let $\mC$ be a symmetric monoidal $\infty$-category compatible with the initial object that admits small colimits.

\begin{enumerate}
\item There is a lax 
monoidal $\infty$-category $\sSeq(\mC)^\ot \to \Ass $ whose $\infty$-category of colors is $\sSeq(\mC)$, which agrees with the composition product monoidal structure if $\mC$ is compatible with small colimits.

\item The tensor unit is $\triv$ and for every $\ell \geq 2$, any $\mO_1, ..., \mO_\ell \in \sSeq(\mC)$ and $\n \in \Sigma$ there is an equivalence 
$$(\mO_1 \circ ... \circ \mO_\ell)_\n \simeq \colim_{\f \in \Fin^\ell_\n}((\mO_1)_{\n_1} \ot \bigotimes_{\bi \in \n_1}(\mO_2)_{ \f_1^{-1}(\bi)} \ot ... \ot \bigotimes_{\bi \in \n_{\ell-1}}(\mO_\ell)_{\f_{\ell-1}^{-1}(\bi)}). $$ 

\item Let $\phi: \mC \to \mD$ be a symmetric monoidal functor between symmetric monoidal $\infty$-categories having small colimits.
There is a map of non-symmetric $\infty$-operads $\sSeq(\mC)^\ot \to \sSeq(\mD)^\ot $, which is a monoidal functor if $\phi$ preserves small colimits.

\item The underlying weak biaction of $\sSeq(\mC)$ on $\sSeq(\mC)$
exhibits $\sSeq(\mC)$ as lax bitensored over $\sSeq(\mC)$
and restricts to a lax biaction of $\sSeq(\mC)$ on $\mC$,
where the lax right action is trivial.

\end{enumerate}

\end{proposition}

\begin{proof}
	
By Proposition \ref{embe} there is a right adjoint symmetric monoidal embedding $\mC \subset \mC'$ into a symmetric monoidal $\infty$-category compatible with small colimits that preserves the initial object.
So we can form the monoidal $\infty$-category $\sSeq(\mC')^\ot \to \Ass$
and define $\sSeq(\mC)^\ot \subset \sSeq(\mC')^\ot$ as the full non-symmetric suboperad spanned by $\sSeq(\mC).$
Since the embedding $\mC \subset \mC'$ admits a left adjoint $\L$,
the embedding $\sSeq(\mC) \subset \sSeq(\mC')$ admits a left adjoint $\L_*$.
Thus by \cite[Lemma 3.13.]{heine2024local} the restriction $\sSeq(\mC)^\ot \subset \sSeq(\mC')^\ot \to \Ass$ is a locally cocartesian fibration, where the composition product of $\mO_1 ,... , \mO_\ell$ is $\L_*(\mO_1 \circ ... \circ \mO_\ell)$.
So by Remark \ref{dghfgh} we obtain a canonical equivalence
$$\L_*(\mO_1 \circ ... \circ \mO_\ell)_\n \simeq \L((\mO_1 \circ ... \circ \mO_\ell)_\n) \simeq $$$$\colim_{\f \in \Fin^\ell_\n}\L((\mO_1)_{\n_1} \ot \bigotimes_{\bi \in \n_1}(\mO_2)_{ \f_1^{-1}(\bi)} \ot ... \ot \bigotimes_{\bi \in \n_{\ell-1}}(\mO_\ell)_{\f_{\ell-1}^{-1}(\bi)})
$$$$\simeq \colim_{\f \in \Fin^\ell_\n}((\mO_1)_{\n_1} \ot \bigotimes_{\bi \in \n_1}(\mO_2)_{ \f_1^{-1}(\bi)} \ot ... \ot \bigotimes_{\bi \in \n_{\ell-1}}(\mO_\ell)_{\f_{\ell-1}^{-1}(\bi)}).$$
Moreover the embedding $\sSeq(\mC) \subset \sSeq(\mC')$ preserves the tensor unit because the embedding $\mC \subset \mC'$ preserves the initial object.
This proves 1., 2. and 4. 

3.: The symmetric monoidal functor $\phi$ induces a small colimits preserving symmetric monoidal functor $\mC' \to \mD'$ between symmetric monoidal $\infty$-categories compatible with small colimits.
By Remark \ref{remf} there is an induced monoidal functor $\sSeq(\mC')^\ot \to \sSeq(\mD')^\ot $, which restricts to a map of non-symmetric $\infty$-operads $\sSeq(\mC)^\ot \to \sSeq(\mD)^\ot $.
The description of the composition product of 2. implies 3.
	
\end{proof}


\begin{definition}
An oplax monoidal structure on an $\infty$-category is a lax monoidal structure on the opposite $\infty$-category. 	
	
\end{definition}

\begin{definition}Let $\mC$ be a symmetric monoidal $\infty$-category compatible with the final object that admits small limits.
By Proposition \ref{weak} applied to $\mC^\op$ there is a lax monoidal structure on the $\infty$-category $\sSeq(\mC^\op)$, i.e.
an oplax monoidal structure on $ \sSeq(\mC^\op)^\op \simeq \sSeq(\mC),$
which we call the cocomposition oplax monoidal structure.
We write $\X * \Y$ for the tensor product of $\X,\Y \in \sSeq(\mC)$ provided by the cocomposition oplax monoidal structure and call $\X * \Y$ the cocomposition product of $\X,\Y \in \sSeq(\mC)$.
So $\X * \Y$ is the composition product of $\X,\Y$ applied to $\mC^\op.$

\end{definition}

\subsection{$\infty$-operads and $\infty$-cooperads}

Next we use the composition product to define $\infty$-operads and $\infty$-cooperads in any symmetric monoidal $\infty$-category that admits small colimits.

\begin{definition}Let $\mC$ be a symmetric monoidal $\infty$-category having small colimits.
An $\infty$-operad in $\mC$ is an associative algebra in $\sSeq(\mC)$ with respect to the composition product.
We call an $\infty$-operad $\mO$ in $\mC$ non-unital if $\mO_0$ is initial in $\mC.$

\end{definition}

\begin{definition}
Let $\mC$ be a symmetric monoidal $\infty$-category $\mC$ having small limits. A (non-counital) $\infty$-cooperad in $\mC$ is a (non-unital) $\infty$-operad in $ \mC^\op $.
\end{definition}

\begin{notation}Let $\mC$ be a symmetric monoidal $\infty$-category $\mC$ having small colimits and $\mO$ an $\infty$-operad in $\mC$.
We write $\mO_1=\tu$ if the unit $\triv \to \mO$ induces an equivalence
$\tu = \triv_1 \simeq \mO_1$, and similar for $\infty$-cooperads.	
	
\end{notation}

\begin{notation}Let $\mC$ be a symmetric monoidal $\infty$-category $\mC$ having small colimits.

Let $$\Op(\mC) = \Alg(\sSeq(\mC))$$
be the $\infty$-category of $\infty$-operads in $\mC$.

Let $$\Op(\mC)^\nun \subset \Op(\mC)$$
be the $\infty$-category of non-unital $\infty$-operads in $\mC$.

Let $$\widetilde{\Op}(\mC) \subset \Op(\mC)^\nun$$
be the $\infty$-category of non-unital $\infty$-operads $\mO$ in $\mC$
such that $\mO_1 =\tu.$

For every $\infty$-operad $\mO \in \Op(\mC)$ we set $$\Alg_\mO(\mC):= \LMod_\mO(\mC),$$
where we use the left action of $\sSeq(\mC)$ on $\mC.$

\end{notation}

We fix the following dual notation:

\begin{notation}Let $\mC$ be a symmetric monoidal $\infty$-category $\mC$ having small limits.

Let $$\mathrm{coOp}(\mC):= \Op(\mC^\op)^\op$$
be the $\infty$-category of $\infty$-cooperads in $\mC$.

Let $$\mathrm{coOp}(\mC)^\nun \subset \mathrm{coOp}(\mC)$$
be the $\infty$-category of non-counital $\infty$-cooperads in $\mC$.

Let $$\widetilde{\mathrm{coOp}}(\mC) \subset \mathrm{coOp}(\mC)^\nun$$
be the $\infty$-category of non-counital $\infty$-cooperads $\mQ$ in $\mC$
such that $\mQ_1 =\tu.$

For every $\infty$-cooperad $\mQ \in \mathrm{coOp}(\mC)$ we set $$ \co\Alg_\mQ(\mC):= \Alg_\mQ(\mC^\op)^\op.$$
\end{notation}

\begin{remark}\label{monad}Let $\mC$ be a symmetric monoidal $\infty$-category compatible with small colimits.
Every $\infty$-operad in $\mC$ has an associated monad on $\mC:$	
the left action of the composition product on $\sSeq(\mC)$ on $\mC$ 
is the pullback of the endomorphism left action of $\Fun(\mC,\mC)$ on $\mC$ along
a unique monoidal functor $\sSeq(\mC) \to \Fun(\mC,\mC)$
sending an $\infty$-operad $\mO$ in $\mC$ to a monad $\T_\mO= \mO \circ (-)$ on $\mC.$
Hence $\Alg_\mO(\mC)$ identifies with $\LMod_{\T_\mO}(\mC)$,
the $\infty$-category of $\T_\mO$-algebras.	
Thus the monad associated to the free-forgetful adjunction $\mC \rightleftarrows \Alg_\mO(\mC)$ identifies with the monad associated to the free-forgetful adjunction $\mC \rightleftarrows \LMod_{\T_\mO}(\mC) $, which is $\T_\mO$
by \cite[Proposition 4.7.3.3. 2.]{lurie.higheralgebra}.

\end{remark}

\begin{proposition}\label{fghjjhbv}\label{dfghjvbn}
	
Let $\mC$ be a symmetric monoidal $\infty$-category compatible with the initial object that admits small colimits.

\begin{enumerate}
\item The $\infty$-category $\Op(\mC)$ admits an initial object lying over $\triv.$

\vspace{1mm}

\item If $\mC$ has a zero object, $\widetilde{\Op}(\mC)$ has a zero object
and the embedding $\widetilde{\Op}(\mC) \subset \Op(\mC)$ preserves the zero object.
\end{enumerate}

\end{proposition}

\begin{proof}
By Corollary \ref{cory} there is a symmetric monoidal embedding $\mC\subset \mC'$ into a symmetric monoidal $\infty$-category compatible with small colimits such that the embedding preserves initial and final objects.
The embedding $\mC \subset \mC'$ induces an embedding $\sSeq(\mC) \subset \sSeq(\mC')$ of non-symmetric $\infty$-operads that preserves the tensor unit.
Consequently, we can assume for 1. and 2. that the symmetric monoidal structure on $\mC$ is compatible with small colimits.

1.: Since the symmetric monoidal structure on $\mC$ is compatible with small colimits, the composition product on $\sSeq(\mC)$ defines a monoidal structure
whose tensor unit is $\triv.$
Hence by \cite[Proposition 3.2.1.8., Lemma 3.2.1.10.]{lurie.higheralgebra} the $\infty$-category $\Op(\mC)$ admits an initial object lying over $\triv.$


2.: Let $\Sigma_{ \geq 2}:= \coprod_{\n \geq 2} \B\Sigma_\n$ and $ \sSeq(\mC)_{\geq 2} := \Fun(\Sigma_{\geq 2}, \mC) \simeq \prod_{\n \geq 2} \Fun(\B\Sigma_\n, \mC). $
Let $$(\sSeq(\mC)_{\geq 1})_{\triv / }' \subset (\sSeq(\mC)_{\geq 1})_{\triv / } $$
be the full subcategory of symmetric sequences $\mO$ under $\triv$ concentrated in positive degrees such that the induced morphism $ \tu \simeq \triv_1\to \mO_1$ in $\mC$ is an equivalence.
There is an equivalence 
$(\sSeq(\mC)_{\geq 1})_{\triv / } \simeq \mC_{\tu / } \times \sSeq(\mC)_{\geq 2} $ that restricts to an equivalence
$ (\sSeq(\mC)_{\geq 1})_{\triv / }' \simeq \sSeq(\mC)_{\geq 2}. $ Under this equivalence $\triv$ corresponds to the initial object, which is the zero object since $\mC$ admits a zero object.
Thus the $\infty$-category $ (\sSeq(\mC)_{\geq 1})_{\triv / }'$ admits a zero object.

By 1. the $\infty$-category $\widetilde{\Op}(\mC)$ has an initial object, which is preserved by the forgetful functor
$ \widetilde{\Op}(\mC) \to   (\sSeq(\mC)_{\geq 1})_{\triv / }'$.
So it is enough to see that the latter functor reflects the final object.
For every $\X, \Y \in \sSeq(\mC)_{\geq 1}$ there is a natural equivalence $ (\X \circ \Y ) _1 \simeq \X_1 \ot \Y_1.$
So the monoidal structure on $(\sSeq(\mC)_{\geq 1})_{\triv / }$
induced by the composition product on $\sSeq(\mC)_{\geq 1}$
restricts to a monoidal structure on $(\sSeq(\mC)_{\geq 1})_{\triv / }'.$
Hence there is a canonical equivalence over $(\sSeq(\mC)_{\geq 1})_{\triv / }':$
$$ \Alg((\sSeq(\mC)_{\geq 1})_{\triv / }') \simeq (\sSeq(\mC)_{\geq 1})_{\triv / }' \times_{(\sSeq(\mC)_{\geq 1})_{\triv / }} \Alg((\sSeq(\mC)_{\geq 1})_{\triv / }) $$$$ \simeq \{ \id_\tu \} \times_{ \mC_{\tu / } } \Op^{\mathrm{}}(\mC)_{\triv / } \simeq\widetilde{\Op}(\mC).$$
So the result follows from the fact that final objects lift to algebras \cite{lurie.higheralgebra}.

\end{proof}

\begin{lemma}\label{leuma}
	
Let $\mC$ be a symmetric monoidal $\infty$-category compatible with the initial object such that $\mC$ has small colimits.
The forgetful functor $\Alg_\triv(\mC) \to \mC$ is an equivalence.

\end{lemma}

\begin{proof}
	
If $\mC$ has small colimits and the symmetric monoidal structure on $\mC$ is compatible with the initial object, by Proposition \ref{embe} there is a symmetric monoidal embedding $\mC \subset \mC'$ into a symmetric monoidal $\infty$-category compatible with small colimits such that the embedding $\mC \subset \mC'$ preserves the initial object.
Thus the induced embedding $\sSeq(\mC)^\ot \subset \sSeq(\mC')^\ot$ 
of non-symmetric $\infty$-operads preserves the tensor unit
$\triv.$ Since $ \sSeq(\mC')^\ot \to \Ass$ is a monoidal $\infty$-category,
the forgetful functor $\Alg_\triv(\mC') \to \mC'$ is an equivalence.
Because the embedding $\sSeq(\mC)^\ot \subset \sSeq(\mC')^\ot$ preserves the tensor unit, the forgetful functor $\Alg_\triv(\mC) \to \mC$ is the pullback of the latter forgetful functor. 
	
\end{proof}

\begin{lemma}\label{lek}
	
Let $\mC$ be a symmetric monoidal $\infty$-category that admits small colimits
and $\mO$ an $\infty$-operad in $\mC$.

\begin{enumerate}
\item If $\mC$ admits a final object, $\Alg_\mO(\mC)$ admits a final object and
the forgetful functor $\Alg_\mO(\mC)\to \mC$ preserves the final object.

\item If $\mC$ is compatible with the initial object, $\Alg_\mO(\mC)$ has an initial object that lies over $\mO_0.$

\item If $\mC$ is compatible with the initial object and $\mO_0$ is a final object of $\mC$, then $\Alg_\mO(\mC)$ has a zero object that lies over the final object.
\end{enumerate}

\end{lemma}

\begin{proof}
	
3. follows from 1. and 2. If $\mC$ is compatible with small colimits, the composition product on $\sSeq(\mC)$ defines a monoidal structure. So the forgetful functor
$\Alg_{\mO}(\mC) \to \mC$ admits a left adjoint sending $\X$ to
$\mO \circ \X \simeq \coprod_{\n \geq 0} (\mO_\n \ot \X^{\ot \n})_{\Sigma_\n}.$
Hence $\Alg_{\mO}(\mC)$ admits an initial object that lies over
$\mO_0.$
Moreover if $\mC$ has a final object, $\Alg_\mO(\mC)=\LMod_\mO(\mC)$ as a final object preserved by the forgetful functor to $\mC.$

In the general case we embed $\mC$ symmetric monoidally and small limits preserving into a symmetric monoidal $\infty$-category $\mC'$ compatible with small colimits to obtain an embedding $\Alg_{\mO}(\mC) \subset \Alg_{\mO}(\mC')$
whose essential image are the $\mO$-algebra structures on objects of $\mC.$
So the result follows.	
	
\end{proof}

\begin{remark}\label{imp}
	
Let $\mC $ be a symmetric monoidal $\infty$-category compatible with small colimits and $\mO$ an $\infty$-operad in $\mC.$

\begin{enumerate}
\item The $\infty$-category $\Alg_\mO(\mC)$ admits small sifted colimits, which are preserved by the forgetful functor to $\mC.$

\item The $\infty$-category $\Alg_\mO(\mC)$ admits those limits that $\mC$ admits and these limits are preserved by the forgetful functor to $\mC.$

\end{enumerate}
\end{remark}

\begin{proof}
	
1. follows from \cite[Corollary 4.2.3.3.]{lurie.higheralgebra} and
2. follows from \cite[Corollary 4.2.3.5.]{lurie.higheralgebra} since $\Alg_\mO(\mC)$ is the $\infty$-category of left $\mO$-modules with respect to the left action of the monoidal $\infty$-category $\sSeq(\mC)$ on $\mC$, where the left action functor $\mO \circ (-): \mC \to \mC$ preserves small sifted colimits.
	
\end{proof}

\begin{lemma}\label{gghjjgh}

Let $\mC$ be a preadditive symmetric monoidal $\infty$-category compatible with small colimits that has small limits.
The identity of $\sSeq(\mC)$ lifts to an oplax monoidal functor from
the composition product on $\sSeq(\mC)$ to the cocomposition product on $\sSeq(\mC).$
This oplax monoidal functor restricts to a monoidal equivalence
on $\sSeq(\mC)_{\geq 1}.$

\end{lemma}

\begin{proof}
Because $\mC$ admits small colimits and small limits, there are non-symmetric $\infty$-operads $\sSeq(\mC)^\ot \to \Ass, \sSeq(\mC^\op)^\ot \to \Ass$. Since the symmetric monoidal structure on $\mC$ is compatible with small colimits, $\sSeq(\mC)^\ot \to \Ass$ is a monoidal $\infty$-category.
Let $ (\sSeq(\mC)^\op)^\ot \to \Ass$ be the opposite monoidal $\infty$-category of $ \sSeq(\mC)^\ot \to \Ass$.
We like to see that there is a map of non-symmetric $\infty$-operads
$ (\sSeq(\mC)^\op)^\ot \to \sSeq(\mC^\op)^\ot$ inducing the identity on underlying $\infty$-categories that restricts to an equivalence
$(\sSeq(\mC)^\op_{\geq 1})^\ot\to \sSeq(\mC^\op)_{\geq 1}^\ot$.
By Corollary \ref{cory} we find that $\mC^\op$ embeds symmetric monoidally into a preadditive symmetric monoidal $\infty$-category
$\mC'$ compatible with small colimits such that the embedding $\mC^\op \subset \mC'$ preserves small limits.
The embedding  $\mC^\op \subset \mC'$ yields an embedding of non-symmetric $\infty$-operads $\sSeq(\mC^\op)^\ot \subset \sSeq(\mC')^\ot$.
There is a map of non-symmetric $\infty$-operads
$ (\sSeq(\mC)^\op)^\ot \to \sSeq(\mC^\op)^\ot$ inducing the identity on underlying $\infty$-categories that restricts to an equivalence
$(\sSeq(\mC)^\op_{\geq 1})^\ot\to \sSeq(\mC^\op)_{\geq 1}^\ot$
if there is a lax monoidal functor
$\iota: (\sSeq(\mC)^\op)^\ot \to \sSeq(\mC')^\ot$ that induces the embedding $\sSeq(\mC)^\op \simeq \sSeq(\mC^\op) \subset \sSeq(\mC')$ on underlying $\infty$-categories and restricts to a monoidal functor
$(\sSeq(\mC)^\op_{\geq 1})^\ot \to \sSeq(\mC')_{\geq 1}^\ot.$

For $\X, \Y \in \sSeq(\mC) $ the structure morphism $ \iota(\X) \ot \iota(\Y) \to \iota(\X \ot \Y) $ in $\sSeq(\mC')$ of $\iota$ induces in degree $\n \in \Sigma$ the canonical morphism 
$$ \alpha: \underset{\bk \in \bN}{\coprod} ( \underset{\underset{ \n_\bi \cap \n_\bj = \emptyset, \bi \neq \bj }{\n_1 \coprod ... \coprod \n_\bk = \n} }{\coprod} \X_\bk \ot( \bigotimes_{1 \leq \bj \leq \bk } \Y_{\n_\bj }))_{\Sigma_\bk}  \to 
\underset{\bk \in \bN}{\prod} ( \underset{\underset{ \n_\bi \cap \n_\bj = \emptyset, \bi \neq \bj }{\n_1 \coprod ... \coprod \n_\bk = \n} }{\prod} \X_\bk \ot( \bigotimes_{1 \leq \bj \leq \bk } \Y_{\n_\bj }))^{\Sigma_\bk} $$ 
induced by the norm map, where colimits and limits are taken in $\mC'.$
We will show that $\alpha$ is an equivalence if $ \Y_0 $ is initial. 

We show that for any $\bk, \n \in \Sigma$ and $\X, \Y \in \sSeq(\mC) $ such that $ \Y_0 $ is initial both maps
$$\phi: ( \underset{\underset{ \n_\bi \cap \n_\bj = \emptyset, \bi \neq \bj }{\n_1 \coprod ... \coprod \n_\bk = \n} }{\coprod} \X_\bk \ot( \bigotimes_{1 \leq \bj \leq \bk } \Y_{\n_\bj }))_{\Sigma_\bk}  \xrightarrow{ } ( \underset{\underset{ \n_\bi \cap \n_\bj = \emptyset, \bi \neq \bj }{\n_1 \coprod ... \coprod \n_\bk = \n} }{\coprod} \X_\bk \ot( \bigotimes_{1 \leq \bj \leq \bk } \Y_{\n_\bj }))^{\Sigma_\bk},  $$$$
\phi': (\underset{\underset{ \n_\bi \cap \n_\bj = \emptyset, \bi \neq \bj }{\n_1 \coprod ... \coprod \n_\bk = \n} }{\coprod} \X_\bk \ot( \bigotimes_{1 \leq \bj \leq \bk } \Y_{\n_\bj }))^{\Sigma_\bk}  \xrightarrow{ }
( \underset{\underset{ \n_\bi \cap \n_\bj = \emptyset, \bi \neq \bj }{\n_1 \coprod ... \coprod \n_\bk = \n} }{\prod} \X_\bk \ot( \bigotimes_{1 \leq \bj \leq \bk } \Y_{\n_\bj }))^{\Sigma_\bk} $$
are equivalences.

Let $\W$ be the set of $\bk$-tuples $(\n_1, ..., \n_\bk) $ of finite pairwise disjoint sets such that $ \n_1 \coprod ... \coprod \n_\bk = \n$. Let $\W' \subset \W$ be the subset of $\bk$-tuples $(\n_1, ..., \n_\bk) $ such that $ \n_1, ..., \n_\bk$ are not empty.
Since $\W'$ is finite, $\phi'$ is an equivalence by preadditivity. 
To prove that $\phi$ is an equivalence, it is enough to check that the canonical $\Sigma_\bk$-action on $ \underset{\underset{ \n_\bi \cap \n_\bj = \emptyset, \bi \neq \bj }{\n_1 \coprod ... \coprod \n_\bk = \n} }{\coprod}  \bigotimes_{1 \leq \bj \leq \bk } \Y_{\n_\bj } $ 
is free because the preadditivity of $\mC$ guarantees that free and cofree $\Sigma_\bk$-actions coincide.

The set $\W$ carries a canonical $\Sigma_\bk$-action such that the canonical embedding $ \W \subset \Sigma^{\times \bk}$ is $\Sigma_\bk$-equivariant, where $\Sigma^{\times \bk}$ carries the permutation action.
The $\Sigma_\bk$-action on $\W$ restricts to a free $\Sigma_\bk$-action on $\W'$
since the $\Sigma_\bk$-action on $\W'$ does not have fixed points.
Composing with the $\Sigma_\bk$-equivariant functor $ \Sigma^{\times \bk} \xrightarrow{\Y^{\times \bk}} \mC'^{\times \bk} \to \mC'$
the map $\W' \to \W$ is $\Sigma_\bk$-equivariant as functor over $\mC'$,
where $\W'$ is free as category over $\mC'$. Taking colimits
we obtain the following $\Sigma_\bk$-equivariant equivalence, where the left hand side is free:
$$ \underset{\underset{ \n_\bi \cap \n_\bj = \emptyset, \bi \neq \bj,  \n_\bi \neq \emptyset }{\n_1 \coprod ... \coprod \n_\bk = \n} }{\coprod}  \bigotimes_{1 \leq \bj \leq \bk } \Y_{\n_\bj } \simeq \underset{\underset{ \n_\bi \cap \n_\bj = \emptyset, \bi \neq \bj }{\n_1 \coprod ... \coprod \n_\bk = \n} }{\coprod}  \bigotimes_{1 \leq \bj \leq \bk } \Y_{\n_\bj }.$$

\end{proof}

\begin{corollary}Let $\mC$ be a preadditive symmetric monoidal $\infty$-category compatible with small colimits that admits small limits.
There is a canonical functor
$$ \co\Alg(\sSeq(\mC)) \simeq \mathrm{coOp}(\mC)$$
compatible with the forgetful functors to $\sSeq(\mC)$ 
that restricts to an equivalence
$$ \co\Alg(\sSeq(\mC)_{\geq 1}) \simeq \mathrm{coOp}(\mC)^\nun. $$

\end{corollary}

\begin{proof}By the last lemma we obtain a forgetful functor $$ \co\Alg(\sSeq(\mC) ) \to \coOp(\mC)=\Alg(\sSeq(\mC^\op))^\op $$ over $ \sSeq(\mC)$ that restricts to an equivalence $$ \co\Alg(\sSeq(\mC)_{\geq 1}) \simeq  \mathrm{CoOp}(\mC)^\nun=\Alg(\sSeq(\mC^\op)_{\geq 1})^\op.$$
\end{proof}

\begin{notation}Let $\mC$ be a preadditive symmetric monoidal $\infty$-category compatible with small colimits that admits small limits.
For any non-counital $\infty$-cooperad $\mQ \in \mathrm{CoOp}(\mC)^\nun \subset \co\Alg(\sSeq(\mC))$ we set $$\co\Alg^{\mathrm{dp}, \conil}_\mQ(\mC):= \co\LMod_\mQ(\mC), $$ where we form comodules in $\mC$ with respect to the left action of $\sSeq(\mC)$ on $\mC$.
\end{notation}

\begin{corollary}\label{forh}

Let $\mC$ be a preadditive symmetric monoidal $\infty$-category compatible with small colimits that admits small limits.
For any non-counital $\infty$-cooperad $\mQ \in \mathrm{CoOp}(\mC)^\nun \subset \co\Alg(\sSeq(\mC))$ there is a forgetful functor over $\mC:$
$$ \co\Alg^{\mathrm{dp}, \conil}_{\mQ}(\mC) \to \co\Alg_\mQ( \mC).$$

\end{corollary}

\begin{proof}By the last lemma we obtain a forgetful functor over $\mC:$
$$ \co\Alg^{\mathrm{dp}, \conil}_{\mQ}(\mC)=\mathrm{coLMod}_{\mQ}( \mC  ) \to \LMod_{\mQ}(\mC^\op)^\op = \co\Alg_\mQ( \mC),$$
where the left hand side are left comodules with respect to the left action
of $\sSeq(\mC)$ on $\mC$ and the right hand side are left modules with respect to the lax left action of $\sSeq(\mC^\op)$ on $\mC^\op.$

\end{proof}

\begin{lemma}\label{comon}
	
\begin{enumerate}
\item Let $\mC$ be a symmetric monoidal $\infty$-category compatible with small colimits and $\mO$ an $\infty$-operad in $\mC.$
The forgetful functor $ \Alg_{\mO}(\mC) \to \mC$ is monadic.	

\item Let $\mC$ be a symmetric monoidal $\infty$-category that has small colimits and $\mO$ an $\infty$-operad in $\mC.$
The forgetful functor $ \Alg_{\mO}(\mC) \to \mC$ is monadic if it has a left adjoint.	
	
\item Let $\mC$ be a preadditive symmetric monoidal $\infty$-category compatible with small colimits that admits small limits and $\mQ$ a non-counital $\infty$-cooperad in $\mC.$
The forgetful functor $ \co\Alg^{\mathrm{dp}, \conil}_{\mQ}(\mC) \to \mC$ is comonadic.


\end{enumerate}	
	
\end{lemma}

\begin{proof}
1.: The $\infty$-category $ \Alg_\mO(\mC)$ is the
$\infty$-category of left modules with respect to a left action of a monoidal
$\infty$-category and so is monadic \cite[Example 4.7.3.9.]{lurie.higheralgebra}.

2.: We embed $\mC$ symmetric monoidally into $\mP(\mC)$ endowed with Day-convolution, which is compatible with small colimits. By 1. the forgetful functor $ \Alg_{\mO}(\mP(\mC)) \to \mP(\mC)$ is monadic.
The forgetful functor $ \Alg_{\mO}(\mC) \to \mC$ is the pullback of the
latter monadic forgetful functor. This implies that every simplicial object in $ \Alg_{\mO}(\mC)$ whose image in $\mC$ splits, admits a geometric realization that is preserved by the forgetful functor.
Hence the claim follows from the Barr-Beck theorem \cite[Theorem 4.7.3.5]{lurie.higheralgebra}.

3.: The $\infty$-category $ \co\Alg^{\mathrm{dp}, \conil}_{\mQ}(\mC)$ is the
$\infty$-category of left comodules with respect to a left action of a monoidal
$\infty$-category. 
Dualising \cite[Example 4.7.3.9.]{lurie.higheralgebra} the $\infty$-category of left comodules with respect to a left action of a monoidal $\infty$-category is comonadic.


\end{proof}

\begin{remark}\label{preo}
	
Let $\mC$ be a preadditive symmetric monoidal $\infty$-category compatible with small colimits and $\mQ$ a non-counital $\infty$-cooperad in $\mC.$
The $\infty$-category $ \co\Alg^{\mathrm{dp}, \conil}_{\mQ}(\mC)$ admits small colimits that are preserved by the forgetful functor to $\mC$ by \cite[Corollary 4.2.3.3.]{lurie.higheralgebra} because $ \co\Alg^{\mathrm{dp}, \conil}_{\mQ}(\mC)$ is the $\infty$-category of left comodules with respect to a left action of a monoidal $\infty$-category.

Let $\mC$ be a preadditive presentably symmetric monoidal $\infty$-category. The $\infty$-category $ \co\Alg^{\mathrm{dp}, \conil}_{\mQ}(\mC)$ is presentable
\cite[Proposition 5.51.]{Mon} since the functor $ \co\Alg^{\mathrm{dp}, \conil}_{\mQ}(\mC) \to \mC$ is comonadic by Lemma \ref{comon} (3) and the associated comonad $\mQ \circ (-): \mC \to \mC$ is accessible.

\end{remark}

\begin{lemma}\label{dfghjvbn}
Let $\mC$ be a symmetric monoidal $\infty$-category having small colimits and $\K \to \Op(\mC)$ a functor such that the composition $\K \to \Op(\mC)\to \sSeq(\mC)$ admits a limit.
Then the functor $\K \to \Op(\mC)$ has a limit, which is preserved by the forgetful functor
$\Op(\mC) \to \sSeq(\mC).$
And similarly, for $\Op(\mC)$ replaced by $ \Op(\mC)^\nun$.
\end{lemma}

\begin{proof}We can generally assume that $\K$ is small by changing to a larger universe if necessary.
By Proposition \ref{embe} there is a symmetric monoidal embedding $\mC\subset \mC'$ into a symmetric monoidal $\infty$-category compatible with small colimits that admits small limits such that the embedding preserves small limits.
The embedding $\mC \subset \mC'$ induces embeddings $\Op(\mC) \subset \Op(\mC'), \Op(\mC)^\nun \subset \Op(\mC')^\nun$.
Consequently, we can assume that the symmetric monoidal structure on $\mC$ is compatible with small colimits.
In this case the result follows from \cite[Corollary 3.2.2.4.]{lurie.higheralgebra} since
$\Op(\mC) \simeq \Alg(\sSeq(\mC)), \Op(\mC)^\nun \simeq \Alg(\sSeq(\mC)_{\geq1})$
are $\infty$-categories of associative algebras in a monoidal $\infty$-category.	

\end{proof}

\subsection{Shifting $\infty$-operads}\label{dfghfghffff}

In this subsection we construct a shift functor on the $\infty$-category of $\infty$-operads and $\infty$-cooperads in any stable symmetric monoidal $\infty$-category $\mC$ having small colimits.
In the following let $[-]$ be the (object-wise) shift functor of $\sSeq(\mC).$

\begin{lemma}
Let $\mC$ be a stable symmetric monoidal $\infty$-category compatible with small colimits and $\n \in \bZ.$
Then $\triv[\n]$ is tensor inverse to $\triv[-\n]$ in the composition product. 
\end{lemma}

\begin{proof}
For every $\n, \m \in \bZ$ and $\X, \Y \in \sSeq(\mC)$ such that $\X$ is concentrated in degree 1  there is a canonical equivalence $$ \X[\n] \circ \Y[\m] \simeq \coprod_{\bk \in \bN} (\X[\n])_\bk \ot_{\Sigma_\bk} ( \Y[\m]  )^{\ot \bk} \simeq \X_1[\n] \ot \Y[\m] \simeq (\X_1 \ot \Y)[\n+ \m] .$$
In particular, there is a canonical equivalence $\triv[\n] \circ \triv[\m] \simeq \triv[\n+\m].$
\end{proof}

\begin{construction}
Let $\mC$ be a stable symmetric monoidal $\infty$-category compatible with small colimits and $\n \in \bZ.$
Via the canonical equivalence
$\Fun^{\ot, \L}_{\mC/}(\sSeq(\mC), \sSeq(\mC)) \simeq \sSeq(\mC) $ evaluating at $\triv$ the tensor invertible object $\triv[\n] $ corresponds to a symmetric monoidal autoequivalence $\alpha_\n$ of $\sSeq(\mC) $ under $\mC$,
which gives rise to a monoidal autoequivalence of 
$\Fun^{\ot, \L}_{\mC/}(\sSeq(\mC), \sSeq(\mC)) \simeq \sSeq(\mC) $ conjugating with $\alpha_\n$ denoted by $(-)(\n)$.
\end{construction}

\begin{lemma}\label{fghjvbnml}
Let $\mC$ be a stable symmetric monoidal $\infty$-category compatible with small colimits.
For every $\n \in \bZ, \Y \in \sSeq(\mC) $ and $\br \in \Sigma$ there is a natural equivalence $$ \Y(\n)_\br \simeq \Y_\br [(1-\br) \n ].$$

\end{lemma}

\begin{proof}

The object $\Y \in \sSeq(\mC)$ uniquely lifts to a symmetric monoidal small colimits preserving endofunctor $\bar{\Y} $ of $\sSeq(\mC) $ under $\mC$.
There is a canonical equivalence $$ \Y(\n) \simeq (\alpha_{- \n} \circ \bar{\Y} \circ \alpha_\n) (\triv) \simeq \alpha_{- \n} ( \bar{\Y} (\triv[\n])) \simeq \alpha_{- \n} ( \bar{\Y} (\triv))[\n] $$
$$\simeq \alpha_{- \n} (\Y)[\n] \simeq (\Y \circ (\triv[-\n]))[\n] \simeq $$ 
$$(\coprod_{\bk \in \Sigma} \Y_\bk \ot_{\Sigma_\bk} \triv[-\n]^{\ot \bk})[\n] \simeq \coprod_{\bk \in \Sigma} \Y_\bk \ot_{\Sigma_\bk} \triv^{\ot \bk}[(1-\bk) \n ] $$
and thus $ \Y(\n)_\br \simeq \Y_\br \ot_{\Sigma_\br} (\Sigma_\br \times \tu) [(1-\br) \n ] \simeq \Y_\br [(1-\br) \n ]. $

\end{proof}
\begin{remark}
Lemma \ref{fghjvbnml} shows that for any $\n \in \bZ$ the autoequivalence $(-)(\n)$ of $ \sSeq(\mC) $ restricts to the $\n$-fold shift of $\mC$ along the embedding $\mC \subset \sSeq(\mC)$ viewing an object concentrated in degree zero.
Moreover $(-)(\n)$ restricts to an autoequivalence of $\sSeq(\mC)_{\geq 1}. $ 
\end{remark}
The monoidal autoequivalence $(-)(\n)$ of $ \sSeq(\mC) $ gives rise to induced autoequivalences of $\co\Alg(\sSeq(\mC)), \co\LMod(\mC).$
If $\mC$ is preadditive, for every non-counital $\infty$-cooperad $\mQ$ in $\mC$ we obtain a commutative square
$$
\begin{xy}
\xymatrix{
\co\Alg^{\mathrm{dp}, \conil}_{\mQ(\n)}( \mC)   \ar[d] 
\ar[rr]^\simeq 
&&  \co\Alg^{\mathrm{dp}, \conil}_\mQ( \mC)  \ar[d]
\\
\mC \ar[rr]^{(-)[-\n]}_\simeq  &&  \mC.
}
\end{xy} 
$$ 

\begin{construction}

If $\mC$ is a stable symmetric monoidal $\infty$-category having small colimits, Corollary \ref{cory} states that $\mC$ embeds symmetric monoidally and exact into a stable symmetric monoidal $\infty$-category $ \mC'$ compatible with small colimits.
Thus the non-symmetric $\infty$-operad $ \sSeq(\mC)^\ot$ embeds into the
monoidal $\infty$-category $ \sSeq(\mC')^\ot $ and for any $\n \in \bZ$ the monoidal autoequivalence $(-)(\n)$ of $ \sSeq(\mC') $ restricts to a 
monoidal autoequivalence $(-)(\n)$ of $ \sSeq(\mC) $.
The autoequivalence $(-)(\n)$ gives rise to autoequivalences of
$\Op(\mC)= \Alg( \sSeq(\mC))$ and $ \LMod(\mC).$
For every $\infty$-operad $\mO $ in $\mC$ we obtain a commutative square
$$
\begin{xy}
\xymatrix{
\Alg_{\mO(\n)}(\mC)  \ar[d] 
\ar[rr]^\simeq 
&& \Alg_\mO(\mC) \ar[d]
\\
\mC \ar[rr]^{(-)[-\n]}_\simeq  &&  \mC.
}
\end{xy} 
$$ 
\end{construction}

\subsection{Truncating $\infty$-operads}

In this subsection we define and study truncations of $\infty$-operads,
which we use in the next section to analyze Koszul duality. 

\begin{notation}
For every $\n \geq 1$ let $\Sigma_{\geq 1 }^{\leq \n} \subset \Sigma_{\geq 1 }$ be the full subcategory of sets with less or equal than $\n$ elements.
\end{notation}
The embedding $\iota: \Sigma_{\geq 1 }^{ \leq \n} \subset \Sigma_{\geq 1 }$ induces a localization
$$ \iota^\ast:  \sSeq(\mC)_{\geq 1}  \rightleftarrows \sSeq(\mC)_{\geq 1}^{\leq\n}  : \iota_\ast, $$
where the left adjoint is restriction and the fully faithful right adjoint extends beyond degree $\n$ by the zero object.

\begin{lemma}\label{fghfghkm}
	
Let $\mC$ be a symmetric monoidal $\infty$-category compatible with the zero object that admits small colimits.
For every $\n \geq 1 $ the localization
$ \iota^\ast: \sSeq(\mC)_{\geq 1} \rightleftarrows \sSeq(\mC)_{\geq 1}^{\leq\n} :\iota_\ast $
is compatible with the composition product on $ \sSeq(\mC)_{\geq 1}. $	
	
\end{lemma}

\begin{proof}

By Proposition \ref{embe} we find that $\mC$ embeds symmetric monoidally into a symmetric monoidal $\infty$-category $\mC'$ compatible with small colimits such that the embedding $\mC \subset \mC'$ preserves the zero object and admits a left adjoint $\L.$ 
We obtain an embedding $\sSeq(\mC)^\ot_{\geq 1} \subset \sSeq(\mC')^\ot_{\geq 1}$ of non-symmetric $\infty$-operads and an embedding $\sSeq(\mC)_{\geq 1}^{\leq\n}  \subset \sSeq(\mC')_{\geq 1}^{\leq\n}.$
Thus $\X_1 \circ ... \circ \X_\bk $ is the image of the corresponding composition product in $\sSeq(\mC')_{\geq 1}$ under the functor $\L_\ast: \sSeq(\mC')_{\geq 1} \to \sSeq(\mC)_{\geq 1}$ induced by $\L.$ 	
A morphism of $ \sSeq(\mC)_{\geq 1} $ is a local equivalence if and only if it induces an equivalence in degrees $\leq \n $ and similar for $\mC'.$
So the embedding $\sSeq(\mC)_{\geq 1} \subset \sSeq(\mC')_{\geq 1}$
and its left adjoint $\L_\ast: \sSeq(\mC')_{\geq 1} \to \sSeq(\mC)_{\geq 1}$ preserve local equivalences.
Hence we can reduce to the case that the symmetric monoidal structure 
on $\mC$ is compatible with small colimits.

Let $ \f: \X \to \Y $ be a morphism in $ \sSeq(\mC)_{\geq 1}$
such that $\f_\bj : \X_\bj \to \Y_\bj$ is an equivalence for every $\bj \leq \n$ 
and $\Z \in \sSeq(\mC)_{\geq 1}. $
We want to see that $(\f \circ \Z)_\s$ and $(\Z \circ \f)_\s$ are both equivalences for every $\s \leq \n. $ 
We compute
$$(\f \circ \Z)_\s = \underset{\bk \geq 0}{\coprod} ( \underset{ {\n_1 \coprod ... \coprod \n_\bk = \s} }{\coprod} \f_\bk \ot( \bigotimes_{1 \leq \bj \leq \bk } \Z_{\n_\bj }))_{\Sigma_\bk} = \underset{\s \geq \bk \geq 0}{\coprod} ( \underset{ {\n_1 \coprod ... \coprod \n_\bk = \s} }{\coprod} \f_\bk \ot( \bigotimes_{1 \leq \bj \leq \bk } \Z_{\n_\bj }))_{\Sigma_\bk}, $$
$$(\Z \circ \f)_\s = \underset{\bk \geq 0}{\coprod} ( \underset{ {\n_1 \coprod ... \coprod \n_\bk = \s} }{\coprod} \Z_\bk \ot( \bigotimes_{1 \leq \bj \leq \bk } \f_{\n_\bj }))_{\Sigma_\bk}. $$
	
\end{proof}

\begin{notation}Let $\mC$ be a symmetric monoidal $\infty$-category that admits small colimits and a zero object preserved by the tensor product in each component.
For every $\n \geq 1$ let $ \Op(\mC)^\nun_{\leq \n} \subset \Op(\mC)^\nun $ be the full subcategory of $\infty$-operads in $\mC$ whose underlying symmetric sequence vanishes beyond degree $\n,$ i.e. belongs to the full subcategory $ \sSeq(\mC)_{\geq 1}^{\leq \n}  \subset  \sSeq(\mC)_{\geq 1} .$
\end{notation}
\begin{corollary}\label{cqop}
Let $\mC$ be a symmetric monoidal $\infty$-category that admits small colimits and a zero object preserved by the tensor product in each component.
The embedding $ \Op(\mC)^\nun_{\leq \n} \subset \Op(\mC)^\nun$ admits a left adjoint 
$ (-)_{\leq \n } :\Op(\mC)^\nun \to \Op(\mC)^\nun_{\leq \n } $
that fits into a commutative square, where the bottom functor is restriction:
\begin{equation*}
\begin{xy}
\xymatrix{
\Op(\mC)^\nun \ar[d] 
\ar[r] 
& \Op(\mC)^\nun_{\leq \n} \ar[d] 
\\
\sSeq(\mC)_{\geq 1} \ar[r] &  \sSeq(\mC)_{\geq 1}^{ \leq \n}.
}
\end{xy} 
\end{equation*}  

\end{corollary}

\vspace{1mm}

\begin{remark}\label{remfs}\label{dghjjjkkhg}\label{leuj}
	
If $\mC$ is a presentably symmetric monoidal $\infty$-category, by the adjoint functor theorem the functor 
$ (-)_{\leq \n } :\Op(\mC)^\nun \to \Op(\mC)^\nun_{\leq \n } $ preserving small limits and sifted colimits admits a left adjoint denoted by $\f_\n$.
The composition $$\f_\n \circ (-)_{\leq \n }: \Op(\mC)^\nun \to \Op(\mC)^\nun_{\leq \n } \to \Op(\mC)^\nun $$ is left adjoint to the functor $\tau_\n: \Op(\mC)^\nun \xrightarrow{ (-)_{\leq \n } } \Op(\mC)^\nun_{\leq \n } \subset \Op(\mC)^\nun. $

For every non-unital $\infty$-operad $\mO$ in $\mC$ the induced diagram 
of $\infty$-operads in $\mC$:
$$\f_1(\mO_{\leq1}) \to ... \to \f_\n(\mO_{\leq \n}) \to \f_{\n+1}(\mO_{\leq \n+1}) \to ...\to \mO $$
is a colimit diagram (see \cite[Remark 4.23.]{2018arXiv180306325H} for details).

This implies that for any $\mO$-algebra $\A$ in $\mC$ the induced diagram
$$\mO \circ_{\f_1(\mO_{\leq 1})}\A \to ... \to \mO \circ_{\f_\n(\mO_{\leq \n})} \A \to \mO \circ_{\f_{\n+1}(\mO_{\leq \n+1})} \A \to ... \to \mO \circ_\mO \A \simeq \A $$ of $\mO$-algebras in $\mC$ is a colimit diagram:
for any $\mO$-algebra $\B$ in $\mC$ the induced map
$$ \Alg_\mO(\mC)(\A,\B) \to \Alg_\mO(\mC)(\colim_{\n \geq 1} \mO \circ_{\f_\n(\mO_{\leq \n})} \A,\B) \simeq \lim_{\n \geq 1}\Alg_\mO(\mC)(\mO \circ_{\f_\n(\mO_{\leq \n})} \A,\B)$$
$$ \simeq \lim_{\n \geq 1}\Alg_{\f_\n(\mO_{\leq \n})}(\mC)(\A,\B) \simeq (\lim_{\n \geq 1}\Alg_{\f_\n(\mO_{\leq \n})}(\mC))(\A,\B)$$
identifies with the induced map on mapping spaces of the equivalence
$$\Alg_\mO(\mC) \simeq \Alg_{\colim_{\n \geq 1} \f_\n(\mO_{\leq \n})}(\mC) \simeq \lim_{\n \geq 1}\Alg_{\f_\n(\mO_{\leq \n})}(\mC).$$
	
\end{remark}

\vspace{1mm}
\begin{lemma}\label{fghfghjjjl}
	
Let $\mC$ be a symmetric monoidal $\infty$-category compatible with the initial object that admits small colimits and let $\n \in \Sigma.$
Let $\X_1, ..., \X_\bk$ be objects of $ \sSeq(\mC)_{\geq 1} $ for $\bk \geq 2.$ 
If $\X_\bi \in  \sSeq(\mC)_{\geq \n} $  for some $ 1 \leq \bi \leq \bk, $
then $\X_1 \circ ... \circ \X_\bk $ belongs to $ \sSeq(\mC)_{\geq \n} $.

\end{lemma}

\begin{proof}
	
Proposition \ref{embe} states that $\mC$ embeds symmetric monoidally into a symmetric monoidal $\infty$-category $\mC'$ compatible with small colimits such that the embedding $\mC \subset \mC'$ preserves the initial object and admits a left adjoint $\L.$ 
We obtain an embedding $\sSeq(\mC)^\ot_{\geq 1} \subset \sSeq(\mC')^\ot_{\geq 1}$ of non-symmetric $\infty$-operads and an embedding $\sSeq(\mC)_{\geq 1}^{\leq\n}  \subset \sSeq(\mC')_{\geq 1}^{\leq\n}.$
Moreover $\X_1 \circ ... \circ \X_\bk $ is the image of the corresponding composition product in $ \sSeq(\mC')_{\geq 1} $ under the functor $\L_\ast:  \sSeq(\mC')_{\geq 1}  \to \sSeq(\mC)_{\geq 1} $ induced by $\L.$ 
As a left adjoint the functor $\L$ preserves the initial object so that the functor $\L_\ast$ restricts to a functor $  \sSeq(\mC')_{\geq \n}  \to  \sSeq(\mC)_{\geq \n} $. 
Consequently, we can reduce to the case that the symmetric monoidal structure 
on $\mC$ is compatible with small colimits.
Let $\X \in \sSeq(\mC)_{\geq \n} , \Y \in  \sSeq(\mC)_{\geq 1} .$
For every $\s \geq 0 $ we have
$$(\X \circ \Y)_\s = \underset{\bk \geq 0}{\coprod} ( \underset{ {\n_1 \coprod ... \coprod \n_\bk = \s} }{\coprod} \X_\bk \ot( \bigotimes_{1 \leq \bj \leq \bk } \Y_{\n_\bj }))_{\Sigma_\bk} =\underset{\s \geq \bk \geq 0}{\coprod} ( \underset{ {\n_1 \coprod ... \coprod \n_\bk = \s} }{\coprod} \X_\bk \ot( \bigotimes_{1 \leq \bj \leq \bk } \Y_{\n_\bj }))_{\Sigma_\bk} $$
so that $\X \circ \Y \in  \sSeq(\mC)_{\geq \n} .$
Let $\X \in  \sSeq(\mC)_{\geq 1} , \Y \in \sSeq(\mC)_{\geq \n} .$
For every $\s \geq 0 $ we have
$$(\X \circ \Y)_\s = \underset{\bk \geq 0}{\coprod} ( \underset{ {\n_1 \coprod ... \coprod \n_\bk = \s} }{\coprod} \X_\bk \ot( \bigotimes_{1 \leq \bj \leq \bk } \Y_{\n_\bj }))_{\Sigma_\bk}= \underset{\bk \geq 1}{\coprod} ( \underset{ {\n_1 \coprod ... \coprod \n_\bk = \s} }{\coprod} \X_\bk \ot( \bigotimes_{1 \leq \bj \leq \bk } \Y_{\n_\bj }))_{\Sigma_\bk}. $$
Thus $\X \circ \Y \in \sSeq(\mC)_{\geq \n} .$	
	
\end{proof}

\begin{notation}
For every $\n \geq 1$ let $\Sigma_{ \geq \n} \subset \Sigma_{\geq 1 }$ be the full subcategory of sets having at least $\n$ elements.
	
\end{notation}
\begin{construction}\label{dfhnghdf}
	
Let $\mC$ be a symmetric monoidal $\infty$-category compatible with the initial object that admits small colimits.
The embedding $\kappa: \Sigma_{ \geq \n} \subset \Sigma_{\geq 1 }$ induces a colocalization
$$ \kappa_!: \sSeq(\mC)_{\geq \n} \rightleftarrows \sSeq(\mC)_{\geq 1} : \kappa^\ast, $$
where the right adjoint is restriction and the fully faithful left adjoint extends under degree $\n$ by the initial object.
By Lemma \ref{fghfghjjjl} the embedding $ \kappa_!: \sSeq(\mC)_{\geq \n} \subset \sSeq(\mC)_{\geq 1} $ is $\sSeq(\mC)_{\geq 1}$-linear. 
Thus for every non-unital $\infty$-operad $\mO$ in $\mC$ we obtain a colocalization
$$ \RMod_{ \mO }( \sSeq(\mC)_{\geq \n}) \rightleftarrows \RMod_{\mO} ( \sSeq(\mC)_{\geq 1}) : \tau_{\geq \n}. $$

\end{construction}

For later reference we add the following lemma:

\begin{lemma}\label{dfghfghjfghj}

Let $\mC$ be a presentably symmetric monoidal $\infty$-category and $\mO$ a non-unital $\infty$-operad in $\mC$ such that $\mO_1= \tu$.
Let $\n \geq 1$ and $\X$ a right $\mO$-module in $\mC$ concentrated in degree $\n$.
The space of right $\mO$-module structures on $\X$ is contractible.
\end{lemma}

\begin{proof}	
	

As $\sSeq(\mC)_{\geq1} $ is presentable, the small colimits preserving functor $(-) \circ \X: \sSeq(\mC)_{\geq1}  \to \sSeq(\mC)_{\geq1} $ admits a right adjoint
$[\X,-].$
Thus the space of right $\mO$-module structures on $\X$ is equivalent to the space of maps of $\infty$-operads $\mO \to [\X,\X]$ in $\mC.$

For every $\bk \in \Sigma_{\geq1} $ the functor $\sSeq(\mC)_{\geq1} \to \mC$ evaluating at $\bk$ admits a left adjoint $\F_\bk$ sending $\Z$ to the symmetric sequence concentrated in degree $\bk$ with value $\Sigma_\bk \times \Z.$
Hence by adjointness we get the following:
If $\bk > 1$, we find that $$\mC(\Z, [\X,\X]_\bk) \simeq \sSeq(\mC)_{\geq1} (\F_\bk(\Z)\circ \X,\X)$$ 
is contractible so that $[\X,\X]_\bk$ is the final object.
Otherwise, we find that 
$$\mC(\Z, [\X,\X]_1) \simeq \sSeq(\mC)_{\geq1} (\F_1(\Z) \circ \X,\X) \simeq \mC^{\Sigma_\n}(\Z \ot \X_\n,\X_\n) \simeq \mC(\Z,[\X_\n,\X_\n]^{\Sigma_n}),$$
where $[\X_\n,\X_\n]$ is the internal hom in $\mC$,
and so $[\X,\X]_1 \simeq [\X_\n,\X_\n]^{\Sigma_n}$.
Using the localization $$ (-)_{\leq 1 } :\Op(\mC)^\nun  \to \Op(\mC)^\nun_{\leq 1 } =\Alg(\mC) $$
the space of maps of $\infty$-operads $\mO \to [\X,\X]$ in $\mC$
identifies with the space of maps $\mO_1 \to [\X_\n,\X_\n]^{\Sigma_n}$
of associative algebras in $\mC$.
But by assumption $\mO_1 \simeq \tu$ so that $\mO_1$ is the initial associative algebra in $\mC.$

\end{proof}

\subsection{Hopf $\infty$-operads}

Let $\mC$ be a symmetric monoidal $\infty$-category.
By \cite[Proposition 3.2.4.3.]{lurie.higheralgebra} the $\infty$-category $\Alg_{\bE_\infty}(\mC)$
carries a canonical symmetric monoidal structure such that the forgetful functor
$\Alg_{\bE_\infty}(\mC) \to \mC$ is symmetric monoidal.
So dually, the $\infty$-category $$\co\Alg_{\bE_\infty}(\mC):= \Alg_{\bE_\infty}(\mC^\op)^\op$$
carries a symmetric monoidal structure such that the forgetful functor
$\co\Alg_{\bE_\infty}(\mC) \to \mC$ is symmetric monoidal. 

If $\mC$ admits small colimits, $\co\Alg_{\bE_\infty}(\mC)$ admits small colimits
and the forgetful functor $\co\Alg_{\bE_\infty}(\mC) \to \mC$ preserves small colimits \cite[Corollary 3.2.2.5.]{lurie.higheralgebra}.
So we can make the following definition:

\begin{definition}Let $\mC$ be a symmetric monoidal $\infty$-category $\mC$ having small colimits.

A Hopf $\infty$-operad in $\mC$ is an $\infty$-operad in $\co\Alg_{\bE_\infty}(\mC)$.
\end{definition}

\begin{notation}
We set $$ \Op_\Hopf(\mC):= \Op(\co\Alg_{\bE_\infty}(\mC)), \ \Op_\Hopf^\nun(\mC):= \Op(\co\Alg_{\bE_\infty}(\mC))^\nun.$$
	
\end{notation}

The symmetric monoidal forgetful functor $ \co\Alg_{\bE_\infty}(\mC) \to \mC $ yields forgetful functors $$\Op_\Hopf(\mC) \to \Op(\mC), \ \Op^\nun_\Hopf(\mC) \to \Op(\mC)^\nun.$$


\begin{remark}
Since $\co\Alg_{\bE_\infty}(\mC)$ admits a final object lying over the tensor unit of $\mC,$
by Lemma \ref{dfghjvbn} the $\infty$-categories $ \Op_\Hopf(\mC), \ \Op^\nun_\Hopf(\mC)$ admit a final object lying over the constant symmetric sequence (concentrated in positive degrees) with value the tensor unit.
\end{remark}

\vspace{1mm}

\begin{definition}
Let $\mC$ be a symmetric monoidal $\infty$-category $\mC$ having small colimits.

\vspace{1mm}	
The commutative $\infty$-operad $\Comm_\mC$ in $\mC$
is the final object of $ \Op_\Hopf(\mC)$.

\vspace{1mm}
The non-unital commutative $\infty$-operad $\Comm^\nun_\mC$ in $\mC$
is the final object of $ \Op^{\nun}_\Hopf(\mC)$.

\end{definition}

\begin{definition}

Let $\mC$ be a symmetric monoidal $\infty$-category $\mC$ having small limits.

\vspace{1mm}
The cocommutative $\infty$-cooperad $\coComm_\mC$ in $\mC$
is $\Comm_{\mC^\op}$.

\vspace{1mm}

The non-counital cocommutative $\infty$-cooperad $\coComm^\nun_\mC$ in $\mC$
is $\Comm^\nun_{\mC^\op}$.

\end{definition}



\begin{lemma}

Let $\mC$ be a symmetric monoidal $\infty$-category $\mC$ that admits small colimits and $\mO$ a Hopf $\infty$-operad in $\mC$ such that the unique map
$\mO_0 \to \tu$ of cocommutative coalgebras in $\mC$ is an equivalence.
The fiber product $$\mO^\nun:=\Comm_\mC^\nun \times_{\Comm_\mC} \mO$$
in $\Hopf(\mC)$ exists and is preserved by the forgetful functor $\Hopf(\mC) \to \Op(\mC) $.

\end{lemma}

\begin{proof}
The fiber product $$\Comm_\mC^\nun \times_{\Comm_\mC} \mO$$
in $\Op(\mC)$ exists and is preserved by the forgetful functor $\Op(\mC) \to \sSeq(\mC)$
by Lemma \ref{dfghjvbn}.
So it is enough to see that the fiber product $\Comm_\mC^\nun \times_{\Comm_\mC} \mO$
in $\Hopf(\mC)$ exists and is preserved by the forgetful functor $\Hopf(\mC) \to \sSeq(\mC) $.
Since every fiber product in $\sSeq(\co\Alg_{\bE_\infty}(\mC))$ of maps in $\Hopf(\mC)$ lifts to a fiber product in $\Hopf(\mC)$ by Lemma \ref{dfghjvbn}, it remains to prove that the fiber product $\Comm_\mC^\nun \times_{\Comm_\mC} \mO$
in $\sSeq(\co\Alg_{\bE_\infty}(\mC))$ exists and is preserved by the forgetful functor $\sSeq(\co\Alg_{\bE_\infty}(\mC)) \to \sSeq(\mC)$.
For that it is enough to show that for every $\n \geq 0$ the fiber product $(\Comm_\mC^\nun)_\n \times_{(\Comm_\mC)_\n} \mO_\n$
in $\co\Alg_{\bE_\infty}(\mC)$ exists and is preserved by the forgetful functor $\co\Alg_{\bE_\infty}(\mC) \to \mC$.
This holds by \cite[Proposition 3.1.3.3.]{lurie.higheralgebra} if and only if for every $\ell \geq 0$ the morphism
$$ ((\Comm_\mC^\nun)_\n \times_{(\Comm_\mC)_\n} \mO_\n)^{\ot\ell} \to ((\Comm_\mC^\nun)_\n)^{\ot\ell} \times_{((\Comm_\mC)_\n)^{\ot\ell}} (\mO_\n)^{\ot\ell} $$
is an equivalence.
The latter map trivially is an equivalence since $\Comm_\mC^\nun \to \Comm_\mC$ induces an equivalence at level
$\n$ for all $\n >0$ and $ \mO \to \Comm_\mC$ induces an equivalence at level
$0$.
\end{proof}

There is the following relationship between $\mO $ and $ \mO^\nun$-algebras:

\begin{proposition}\label{fghfghjok}
Let $\mC$ be a preadditive symmetric monoidal $\infty$-category that has small colimits.
\begin{enumerate}
\item There is a canonical embedding
$\Alg_{\mO^\nun}(\mC) \hookrightarrow \Alg_{\mO}(\mC)_{/\tu}$
sending $\X $ to $\X \oplus\tu.$

\item If $\mC$ admits fibers, 
the embedding $\Alg_{\mO^\nun}(\mC) \hookrightarrow \Alg_{\mO}(\mC)_{/\tu} $
admits a right adjoint that takes the fiber of the augmentation. 

\item The embedding $ \Alg_{\mO^\nun}(\mC) \hookrightarrow \Alg_{\mO}(\mC)_{/\tu} $ is an equivalence if $\mC$ is stable or idempotent complete
and additive.

\end{enumerate}

\end{proposition}

\begin{proof}

The forgetful functor
$\Alg_{\mO}(\mC) \to \Alg_{\mO^\nun}(\mC) $ admits a left adjoint $\F:=\mO \circ_{\mO^{\mathrm{nu}}} (-).$ 
For every $\mO^{\mathrm{nu}}$-algebra $\X$ in $\mC$ the unit
$ \X \to \F(\X) $ of the adjunction and the unit $\tu \to \F(\X) $ 
of the $\mO$-algebra give rise to a morphism $\theta_\X: \X \oplus \tu \to \F(\X)$ in $\mC,$ which determines a natural transformation $\theta: (-)\oplus \tu \to \F$ of functors $\Alg_{\mO^\nun}(\mC) \to \mC$.
For every $\Y \in \mC$ the morphism $$\theta_{\mO^\nun \circ \Y}:
(\mO^\nun \circ \Y) \oplus \tu \to \F(\mO^\nun \circ \Y) \simeq \mO \circ \Y$$
is the canonical equivalence. Because source and target of $\theta$ preserve sifted colimits and every $\mO^\nun$-algebra is a geometric realization of free $\mO^\nun$-algebras, $\theta$ is an equivalence.

Since $\mC$ is compatible with the initial object, $\mO^\nun \circ (-): \mC \to \mC$ preserves the zero object and so $ \Alg_{\mO^\nun}(\mC)$ admits a zero object that lies over the zero object of $ \mC.$
So $\F$ lifts to a functor $\bar{\F}:\Alg_{\mO^\nun}(\mC) \to \Alg_{\mO}(\mC)_{/ \tu}$ that factors through the full subcategory $\Alg_{\mO}(\mC)_{/ \tu}' \subset \Alg_{\mO}(\mC)_{/ \tu} $ of augmented algebras whose augmentation admits a fiber in $\mC.$

The functor $\bar{\F}$ is left adjoint to the functor
$$ \Gamma:\Alg_{\mO}(\mC)_{/ \tu}' \to \Alg_{\mO^\nun}(\mC)$$ that takes the fiber of the augmentation in $\Alg_{\mO^\nun}(\mC).$
For any $\X \in \Alg_{\mO^\nun}(\mC)$ the unit
$$ \X \to \Gamma(\bar{\F}(\X)) \simeq 0 \times_{\F(0)} \F(\X) \simeq 0 \times_{(0 \oplus \tu)} (\X \oplus \tu) $$ is the canonical equivalence so that $\F$ is fully faithful. 
If $\mC$ is stable, the right adjoint exists and is conservative since
equivalences can be tested on the fibers. So in this case $\bar{\F}$ is an equivalence.
By Corollary \ref{cory} there is a symmetric monoidal embedding
$\mC \subset \mD $ into a  stable symmetric monoidal $\infty$-category $\mD$ such that $\mC$ is closed in $\mD$ under finite products and retracts. We obtain a commutative square
\begin{align*} 
\begin{xy}
\xymatrix{
\Alg_{\mO^\nun}(\mC)  \ar[d] 
\ar[r]^{\bar{\F}}  
& \Alg_{\mO}(\mC)_{/ \tu}  \ar[d]
\\
\Alg_{\mO^\nun}(\mD) \ar[r]^{\bar{\F}} & \Alg_{\mO}(\mD)_{/ \tu},
}
\end{xy} \end{align*}
where the vertical functors are fully faithful and the bottom functor is an equivalence.
Since $\mC \subset \mD$ is closed under retracts, $\A \in \Alg_{\mO^\nun}(\mD) $ belongs to
$ \Alg_{\mO^\nun}(\mC) $ if its image $ \bar{\F}(\A) \simeq \A \oplus \tu $ belongs to $ \Alg_{\Comm}(\mC)_{/ \tu} $. So the top functor is an equivalence.

\end{proof} 

\begin{corollary}\label{fghjok}
Let $\mC$ be a preadditive symmetric monoidal $\infty$-category that admits small limits.

\begin{enumerate}
\item There is a canonical embedding
$\co\Alg_{\coComm^\mathrm{nu}}(\mC) \hookrightarrow \co\Alg_{\coComm}(\mC)_{\tu/}$
sending $\X $ to $\X \oplus\tu.$

\item If $\mC$ admits cofibers, 
the embedding $\co\Alg_{\coComm^\mathrm{nu}}(\mC) \hookrightarrow \co\Alg_{\coComm}(\mC)_{\tu/}$
admits a left adjoint that takes the cofiber of the coaugmentation. 

\item The embedding $\co\Alg_{\coComm^\mathrm{nu}}(\mC) \hookrightarrow \co\Alg_{\coComm}(\mC)_{\tu/}$ is an equivalence if $\mC$ is stable or idempotent complete and additive.

\end{enumerate}

\end{corollary}

The next proposition is \cite[Theorem 8.9.]{Mon}:

\begin{proposition}\label{Hoopf}
	
Let $\mC$ be a symmetric monoidal $\infty$-category compatible with small colimits
and $\mO$ a Hopf $\infty$-operad in $\mC$.
Then $\Alg_\mO(\mC) $ carries a canonical symmetric monoidal structure such that the forgetful functor
$\Alg_\mO(\mC) \to \mC$ is symmetric monoidal.
	
\end{proposition}

We obtain the following relationship between $\Comm$-algebras and $\bE_\infty$-algebras:

\begin{proposition}\label{compara}

Let $\mC$ be a symmetric monoidal $\infty$-category that admits small colimits.
There is a canonical equivalence over $\mC:$
$$\Alg_{\Comm}(\mC) \simeq \Alg_{\bE_\infty}(\mC).$$ 

\end{proposition}

\begin{proof}

We may assume that the symmetric monoidal structure on $\mC$ is compatible with small colimits. Otherwise we embed $\mC$ symmetric monoidally into $\mP(\mC)$ endowed with Day-convolution and the canonical equivalence
$$\Alg_{\Comm}(\mP(\mC)) \simeq \Alg_{\bE_\infty}(\mP(\mC)) $$ over $\mP(\mC)$
restricts to an equivalence $\Alg_{\Comm}(\mC) \simeq \Alg_{\bE_\infty}(\mC)$ over $\mC.$

If the symmetric monoidal structure on $\mC$ is compatible with small colimits,
the forgetful functor $ \Alg_{\Comm}(\mC) \to \mC$ is monadic and admits a left adjoint $$\Sym:= \Comm \circ (-): \mC \to \Alg_{\Comm} (\mC), \ \X \mapsto \Comm \circ \X \simeq \coprod_{\n \geq 0 } (\X^{\ot \n})_{\Sigma \n}.$$ 
By Proposition \ref{Hoopf} the forgetful functor $ \Alg_{\Comm}(\mC) \to \mC$ refines to a symmetric monoidal functor since $\Comm$ is a Hopf $\infty$-operad in $\mC$.
We start with proving that the symmetric monoidal structure on $ \Alg_{\Comm} (\mC) $ is cocartesian.
The tensor unit $\tu$ of $ \Alg_{\Comm} (\mC) $ is initial since the unique morphism $\alpha: \Sym(\emptyset) \to \tu$ is an equivalence:
indeed, the morphism $ \alpha$ factors as $\Sym(\emptyset) \to \Sym(\tu) \xrightarrow{\mu} \tu,$ where the counit $\mu$ is induced by the multiplication of $\tu.$ So $\alpha$ is the 1-fold multiplication $\tu^{\ot 1} \to \tu $ of the $\Comm$-algebra $\tu$ and so the identity by unitality of the algebra.

Let $\beta: \X \coprod \Y \to \Sym(\X) \ot \Sym(\Y)$ be the morphism in $\mC$ that is the morphism $$\X \simeq \X \ot \tu \simeq \X \ot \Sym (\emptyset) \to \Sym( \X) \ot \Sym( \Y)$$ on the first summand and
the morphism $\Y \simeq \tu \ot \Y \simeq \Sym(\emptyset) \ot \Y \to \Sym( \X) \ot \Sym( \Y)$ on the second summand.
By Lemma \ref{nnlklj} and Remark \ref{hhjbbccx} the symmetric monoidal structure on $\Alg_{\Comm}(\mC)$ is cocartesian if for every $\X, \Y \in \mC$ the following map adjoint to $\beta$ is an equivalence:
$$ \alpha: \Sym (\X \coprod \Y) \xrightarrow{\Sym(\beta)} \Sym (\Sym(\X) \ot \Sym(\Y)) \xrightarrow{\mu} \Sym( \X) \ot \Sym( \Y), $$ where $\mu$ is the multiplication of $\Sym(\X) \ot \Sym(\Y) $.
The latter map identifies with the canonical equivalence.
So the symmetric monoidal structure on $\Alg_{\Comm}(\mC)$ is cocartesian. 

By \cite[Proposition 2.4.3.16.]{lurie.higheralgebra} this implies that the forgetful functor $ \Alg_{\bE_\infty}(\Alg_{\Comm}(\mC)) \to \Alg_{\Comm}(\mC)$ is an equivalence. So the symmetric monoidal forgetful functor $\Alg_{\Comm}(\mC) \to \mC $ lifts to a functor $$\rho: \Alg_{\Comm}(\mC) \simeq \Alg_{\bE_\infty}(\Alg_{\Comm} (\mC)) \to \Alg_{\bE_\infty}(\mC)$$
over $\mC$ between monadic functors.
Let $\rS$ be the left adjoint of the forgetful functor $\Alg_{\bE_\infty}(\mC) \to \mC $.
By \cite[Corollary 4.7.3.16.]{lurie.higheralgebra} the functor $\rho$ is an equivalence if for every $\X \in \mC$ the morphism $\psi: \rS(\X) \to \rS(\Sym(\X)) \to \Sym(\X)$ adjoint to the unit $\X \to \Sym(\X)$ is an equivalence.
The counit $\rS(\Sym(\X)) \simeq \coprod_{\n \geq 0}(\Sym(\X)^{\ot \n})_{\Sigma \n} \to \Sym(\X)$ gives in each summand the multiplication of $\Sym(\X).$
For every $\n \geq 0$ the map $$ (\X^{\ot \n})_{\Sigma \n} \to (\Sym(\X)^{\ot \n})_{\Sigma \n} \to \Sym(\X) \simeq \coprod_{\n \geq 0 } (\X^{\ot \n})_{\Sigma \n},$$
where the second map is the multiplication of $\Sym(\X)$, is the canonical map to the coproduct. Thus $\psi$ is the canonical equivalence.

\end{proof}

\begin{corollary}\label{comparat}

Let $\mC$ be a symmetric monoidal $\infty$-category that admits small limits.
There is a canonical equivalence over $\mC:$
$$\co\Alg_{\coComm}(\mC) \simeq \co\Alg_{\bE_\infty}(\mC).$$ 

\end{corollary}

To prove Proposition \ref{compara} we used the following lemma: 

\begin{lemma}\label{nnlklj}

Let $\mC$ be a symmetric monoidal category that admits finite coproducts.
Let $\G: \mD \to \mC$ be a symmetric monoidal and monadic functor with left adjoint $\F$ such that the tensor unit of $\mD$ is an initial object.
The symmetric monoidal structure on $\mD$ is cocartesian if for every 
$\A, \B \in \mD$ that belong to the essential image of $\F$ 	 
the canonical morphisms $\A \simeq \A \ot \tu \to \A \ot \B, \B \simeq \tu \ot \B \to \A \ot \B $ in $\mD $ exhibit $\A \ot \B$ as a coproduct of $\A$ and $\B$
in $\mD.$

\end{lemma}

\begin{remark}\label{hhjbbccx}

For every $\X, \Y \in \mC$ the canonical morphisms $$ \F(\X) \to \F(\X) \ot \F(\Y), \ \F(\Y) \to \F(\X) \ot \F(\Y) $$ in $\mD$ define a morphism
$\alpha:  \F(\X \coprod \Y) \simeq \F(\X) \coprod \F(\Y) \to \F(\X) \ot \F(\Y)$ in $\mD$.  

Moreover there is a canonical morphism 
$$\beta: \X \coprod \Y \to \G(\F(\X)) \ot \G(\F(\Y)) \simeq \G( \F(\X) \ot \F(\Y))$$ in $\mC$ that is the morphism $$\X \simeq \X \ot \tu \simeq \X \ot \G(\F (\emptyset)) \to  \G(\F(\X)) \ot \G(\F(\Y)) $$ on the first summand and
the morphism $$\Y \simeq \tu \ot \Y \simeq \G(\F(\emptyset)) \ot \Y \to  \G(\F(\X)) \ot \G(\F(\Y)) $$ on the second summand.

The map $\alpha:  \F(\X \coprod \Y) \to \F(\X) \ot \F(\Y)$
is adjoint to $\beta:  \X \coprod \Y \to \G( \F(\X) \ot \F(\Y)).$

\end{remark}

\begin{proof}[Proof of Lemma \ref{nnlklj}]

We write $\A \simeq | \bar{\A} |, \B \simeq | \bar{\B} | $ for some $\G$-split simplicial objects $\bar{\A}, \bar{\B}: \Delta^\op \to \mD$ taking values in the essential image of $\F.$
Let $\Fun(\Delta^\op, \mD)$ be endowed with the levelwise symmetric monoidal structure. The tensor unit of $\Fun(\Delta^\op, \mD)$ is an initial object
as the tensor unit of $\mD$ is.
The canonical morphisms $\bar{\A} \simeq\bar{\A}\ot \tu \to \bar{\A}\ot \bar{\B}, \ \bar{\B} \simeq \tu \ot \bar{\B}\to \bar{\A} \ot \bar{\B}$ in $\Fun(\Delta^\op, \mD) $ exhibit $\bar{\A}\ot \bar{\B}$ as a coproduct of $\bar{\A}$ and $\bar{\B}$
in $\Fun(\Delta^\op, \mD) $ as they do object-wise.
The functor $\ot: \mD \times \mD \to \mD$ sends the $\G \times \G$-split simplicial object $(\bar{\A}, \bar{\B})$ in $\mD \times \mD$ to the 
$\G $-split simplicial object $\bar{\A}\ot \bar{\B}: \Delta^\op \to \mD. $ 

Let $\delta: \mD \to \Fun(\Delta^\op, \mD)$ be the diagonal functor.
The morphism $\bar{\A}\ot \bar{\B} \to \delta(\A \ot \B)$ in $\Fun(\Delta^\op, \mD)$ exhibits 
$\A \ot \B$ as the geometric realization of the simplicial object $\bar{\A}\ot \bar{\B}. $ 
So for every $ \Z \in \mD$
the canonical map $\mD(\A \ot \B, \Z) \to \mD(\A, \Z) \times  \mD(\B, \Z)$
factors as $$\mD(\A \ot \B, \Z) \simeq \Fun(\Delta^\op, \mD)(\bar{\A}\ot \bar{\B}, \delta(\Z)) \simeq $$$$ \Fun(\Delta^\op, \mD)(\bar{\A}, \delta(\Z))  \times  \Fun(\Delta^\op, \mD)(\bar{\B}, \delta(\Z))  \simeq \mD(\A, \Z) \times  \mD(\B, \Z).$$

\end{proof}

\subsection{Trivial algebras}

\begin{construction}\label{Prim}

Let $\mC$ be a symmetric monoidal $\infty$-category such that $\mC$ admits small colimits.
Let $\mO \to \triv$ be an augmented $\infty$-operad in $\mC$.
The augmentation of $\mO$ gives rise to a forgetful functor assigning the trivial $\mO$-algebra:
$$\triv_\mO: \mC \simeq \Alg_{\triv}(\mC) \to \Alg_\mO(\mC).$$

Dually, let $\mC$ be a symmetric monoidal $\infty$-category such that $\mC$ admits small limits.
Let $\triv \to \mQ$ be a coaugmented $\infty$-cooperad in $\mC$.
The coaugmentation of $\mQ$ gives rise to a forgetful functor assigning the trivial $\mQ$-coalgebra:
$$\triv_\mQ: \mC \simeq \co\Alg_{\triv}(\mC) \to \co\Alg_\mQ(\mC).$$ 
\end{construction}

We will prove the following proposition:

\begin{proposition}\label{lemgf}
	
Let $\mC$ be a symmetric monoidal $\infty$-category such that $\mC$ has small colimits and the symmetric monoidal structure is compatible with the initial object. Let $\mO \to \triv$ be an augmented $\infty$-operad in $\mC$.
Then 
$\triv_\mO: \mC \to \Alg_\mO(\mC) $ has a left adjoint
$$ \triv \circ_\mO (-) : \Alg_\mO(\mC)  \to \mC $$ that factors as 
$$\Alg_\mO(\mC) \xrightarrow{\Barc_\mO(-) } \Fun(\Delta^\op, \mC)  \xrightarrow{\colim} \mC $$ such that for any $\X \in \Alg_\mO(\mC) $ and $[\n] \in \Delta$ the object $\Barc_\mO(\X)_\n $ is $\mO \circ ... \circ \mO \circ \X$.

\end{proposition}

We will often apply the dual version of Proposition \ref{lemgf}: 

\begin{corollary}\label{lemfg2}
Let $\mC$ be a symmetric monoidal $\infty$-category such that $\mC$ has small limits and the symmetric monoidal structure is compatible with the final object. Let $\triv \to \mQ$ be a coaugmented $\infty$-cooperad in $\mC$.
Then $\triv_\mQ: \mC \to \co\Alg_\mQ(\mC) $ 
has a right adjoint $$\triv *^\mQ (-) : \co\Alg_\mQ(\mC) \to \mC $$ that factors as 
$$\co\Alg_\mQ(\mC)  \xrightarrow{\Cobar_\mQ(-) } \Fun(\Delta, \mC) \xrightarrow{\lim} \mC$$
such that for any $\X \in \co\Alg_\mQ(\mC) $ and $[\n] \in \Delta$ the object $\Cobar_\mO(\X)_\n $ is $\mO * ... * \mO * \X$.
\end{corollary}

To prove Proposition \ref{lemgf} we fix the following notation:

\begin{notation}
Let $\mD$ be an $\infty$-category bitensored over monoidal $\infty$-categories $\mC,\mE$, $\A$ an algebra in $\mC$, $\B$ an algebra in $\mE$, $\M$ an $\A,\B$-bimodule, $\X$ a right $\A$-module and $\Y$ a left $\B$-module.
By \cite[Theorem 4.4.2.8.]{lurie.higheralgebra} there is a simplicial object $\Barc^\mD_{\A,\B}(\X, \M, \Y)$ in $ \mD$ such that 
$$ \Barc^\mD_{\A,\B}(\X,\M,\Y)_\n \simeq \X \ot \A^{\ot \n} \ot \M \ot \B^{\ot \n} \ot \Y .$$
\end{notation}
\begin{remark}
	
By \cite[Example 4.7.2.7.]{lurie.higheralgebra} the simplicial object $\Barc^\mD_{\A,\B}(\A, \M, \B)$ in $ \mD$
canonically splits, and lifts and extends to an augmented simplicial object in ${_\A}\BMod_\B(\mD)$ that exhibits $\M$ as
$\A \ot_\A \M \ot_\B \A$ and sends $[\n] \in \Delta $ to the free $\A, \B$-bimodule on $ \A^{\ot \n} \ot \M \ot \B^{\ot \n}$.
\end{remark}

Let $\mD$ be an $\infty$-category lax bitensored over non-symmetric $\infty$-operads $\mC,\mE$.
By 
\cite[Definition 2.96.]{heine2024bi}, \cite[Proposition 2.95.]{heine2024higher} there is an enveloping $\infty$-category $\B\Env(\mD)$ bitensored over monoidal $\infty$-categories $\Env(\mC), \Env(\mE)$, which comes equipped with embeddings $\mC \subset \Env(\mC), \mE \subset \Env(\mE)$ of non-symmetric $\infty$-operads and an embedding $\mD \subset \B\Env(\mD)$. 
Every object of $\B\Env(\mD)$ is of the form $\V_1 \ot ... \ot \V_\n \ot \X \ot \W_1 \ot... \ot \W_\bk$ for $\V_1, ..., \V_\n \in \mC, \X \in \mD, \W_1, ..., \W_\bk \in \mE$ and $\bk,\n \geq 0.$
By \cite[Lemma 3.103.]{https://doi.org/10.48550/arxiv.2009.02428} the embedding $\mD \subset \B\Env(\mD)$ admits a left adjoint $\L$ (not compatible with the biactions) that sends an object
$\V_1 \ot ... \ot \V_\n \ot \X \ot \W_1 \ot... \ot \W_\bk$
for $\V_1, ..., \V_\n \in \mC, \X \in \mD, \W_1, ..., \W_\bk \in \mE$
to $\V_1 \ot ... \ot \V_\n \ot \X \ot \W_1 \ot... \ot \W_\bk$, where for the latter object the biaction of $\mC,\mE$ on $\mD$ is used.

\begin{lemma}\label{inamb}
	
Let $\mD$ be an $\infty$-category bitensored over monoidal $\infty$-categories $\mC,\mE$, $\A$ an algebra in $\mC$, $\B$ an algebra in $\mE$, $\M$ an $\A,\B$-bimodule, $\X$ a right $\A$-module and $\Y$ a left $\B$-module.
Then $\Barc^{\mD}_{\A,\B}(\X,-,\Y)$ factors as
$${_\A}\BMod_\B(\mD) \subset {_\A}\BMod_\B(\B\Env(\mD)) \xrightarrow{\Barc^{\B\Env(\mD)}_{\A,\B}(\X,-,\Y)} \Fun(\Delta^\op, \B\Env(\mD))$$$$ \xrightarrow{\Fun(\Delta^\op,\L)}\Fun(\Delta^\op, \mD).$$
	
\end{lemma}

\begin{proof}
	
Since $\mD$ is bitensored over monoidal $\infty$-categories, by \cite[Lemma 3.103.]{https://doi.org/10.48550/arxiv.2009.02428} the left adjoint $\L$ of the embedding $\mD \subset \B\Env(\mD)$ preserves the biactions and so induces a localization ${_\A}\BMod_\B(\L):{_\A}\BMod_\B(\B\Env(\mD)) \rightleftarrows {_\A}\BMod_\B(\mD)$.
Thus $\Barc^{\mD}_{\A,\B}(\X,-,\Y) $ factors as
$${_\A}\BMod_\B(\mD) \subset {_\A}\BMod_\B(\B\Env(\mD)) \xrightarrow{{_\A}\BMod_\B(\L)} {_\A}\BMod_\B(\mD) \xrightarrow{\Barc^{\mD}_{\A,\B}(\X,-,\Y)} \Fun(\Delta^\op, \mD)$$
and ${_\A}\BMod_\B(\B\Env(\mD)) \xrightarrow{{_\A}\BMod_\B(\L)} {_\A}\BMod_\B(\mD) \xrightarrow{\Barc^{\mD}_{\A,\B}(\X,-,\Y)} \Fun(\Delta^\op, \mD)$ identifies with
$${_\A}\BMod_\B(\B\Env(\mD)) \xrightarrow{\Barc^{\B\Env(\mD)}_{\A,\B}(\X,-,\Y)} \Fun(\Delta^\op, \B\Env(\mD)) \xrightarrow{\Fun(\Delta^\op,\L)}\Fun(\Delta^\op, \mD).$$
	
\end{proof}

The latter lemma motivates the following notation:

\begin{notation}
Let $\mD$ be an $\infty$-category lax bitensored over non-symmetric $\infty$-operads $\mC,\mE$ and $\A$ an algebra in $\mC$, $\B$ an algebra in $\mE$. For any right $\A$-module $\X$ and left $\B$-module $\Y$ let $\Barc^\mD_{\A,\B}(\X,-,\Y)$ be the composition:
$${_\A}\BMod_\B(\mD) \subset {_\A}\BMod_\B(\B\Env(\mD)) \xrightarrow{\Barc^{\B\Env(\mD)}_{\A,\B}(\X,-,\Y)} \Fun(\Delta^\op, \B\Env(\mD)) $$$$\xrightarrow{\Fun(\Delta^\op,\L)}\Fun(\Delta^\op, \mD).$$

\end{notation}

\begin{remark}

For any $\M \in {_\A}\BMod_\B(\mD) $ and $[\n]\in \Delta$ we have $$ \Barc^\mD_{\A,\B}(\X,\M,\Y)_\n \simeq \X \ot \underbrace{\A \ot ... \ot \A}_{\n - \text{times}}\ot \M \ot \underbrace{\B \ot ... \ot \B}_{\n - \text{times}} \ot \Y .$$
\end{remark}
\begin{notation}
For every $\A,\B$-bimodule $\M$ let $\X \ot_\A \M \ot_\B \Y$ be the geometric realization of $\Barc^\mD_{\A,\B}(\X,\M,\Y)$ if the latter exists.
\end{notation}
Proposition \ref{lemgf} immediately follows from Lemma \ref{leuma} and the following lemma:

\begin{lemma}\label{hggvujhnh}

Let $\mD$ be an $\infty$-category lax bitensored over non-symmetric $\infty$-operads $\mC,\mE$, $\A$ an augmented algebra in $\mC$ and $\B$ an augmented algebra in $\mE$.  
If $\mD$ has geometric realizations, the functor $$\tu \ot_\A (-) \ot_\B \tu: {_\A}\BMod_\B(\mD) \to {_\tu}\BMod_\tu(\mD) \simeq \mD$$ is left adjoint to the forgetful functor $\triv: \mD \simeq {_\tu}\BMod_\tu(\mD) \to {_\A}\BMod_\B(\mD).$

\end{lemma}

\begin{proof}
	
By \cite[Remark 4.5.3.3.]{lurie.higheralgebra} there is an adjunction $$\tu \ot_\A (-) \ot_\B \tu:  {_\A}\BMod_\B(\mP\B\Env(\mD)) \to {_\tu}\BMod_\tu(\mP\B\Env(\mD)):\triv.$$

Since the biaction of $\mC, \mE$ on $\mD$ is lax, the forgetful functor
${_\tu}\BMod_\tu(\mD) \to \mD$ is an equivalence.
Let $\L$ be the left adjoint of the embedding $\mD \subset \B\Env(\mD)$.
When viewing $\B\Env(\mD)$ as bitensored over the full monoidal subcategories of $\Env(\mC), \Env(\mE)$ spanned by the tensor unit of $\mC, \mE$, respectively,
and $\mD$ as bitensored over the full subcategories of $\mC, \mE$ spanned by the tensor unit of $\mC, \mE$, respectively, the functor $\L: \B\Env(\mD) \to \mD$ preserves the biaction and so gives rise to a localization
$$ {_\tu}\BMod_\tu(\B\Env(\mD)) \to {_\tu}\BMod_\tu(\mD) \simeq \mD.$$

Thus for every $ \Y \in \mD$ and $\M \in {_\A}\BMod_\B(\mD)$ there is a canonical equivalence
$$ {_\A}\BMod_\B(\mD)(\M, \triv(\Y)) \simeq {_\A}\BMod_\B(\mP\B\Env(\mD))(\M, \triv(\Y)) \simeq $$$${_\tu}\BMod_\tu(\mP\B\Env(\mD))(\colim(\Barc_{\A,\B}^{\mP\B\Env(\mD)}(\tu,\M,\tu)), \Y) \simeq $$$$ \lim{_\tu}\BMod_\tu(\B\Env(\mD))(\Barc^{\mB\Env(\mD)}_{\A,\B}(\tu,\M,\tu), \Y) \simeq$$$$ \lim \mD(\L \circ \Barc^{\mB\Env(\mD)}_{\A,\B}(\tu,\M,\tu), \Y) \simeq \mD(\colim(\L \circ \Barc^{\B\Env(\mD)}_{\A,\B}(\tu,\M,\tu)), \Y). $$	 

\end{proof}	



Next we investigate the relationship between the free associative algebra and the primitives:

\begin{lemma}
Let $\mC$ be an additive presentably symmetric monoidal $\infty$-category. There is an adjunction 
$$\T: \mC \rightleftarrows \Bialg(\mC):\Prim $$
whose left adjoint lifts the free associative algebra.
	
\end{lemma}

\begin{proof}

The symmetric monoidal $\infty$-category $\co\Alg_{\bE_\infty}(\mC)$ is compatible with small colimits
and the forgetful functor $\co\Alg_{\bE_\infty}(\mC) \to \mC$ preserves small colimits.
This implies that the forgetful functor $$ \Bialg(\mC)= \Grp(\co\Alg_{\bE_\infty}(\mC))=\Mon(\co\Alg_{\bE_\infty}(\mC)) \to \co\Alg_{\bE_\infty}(\mC)_{\tu/}$$ admits a left adjoint $\widetilde{\mathrm{Free}}$ that covers the left adjoint of the forgetful functor $\Alg(\mC) \to \mC_{\tu/},$ the free associative algebra functor.
By Corollaries \ref{fghjok} and \ref{comparat} there is an equivalence $$\Coalg_{\coComm^{\mathrm{nu}}}(\mC)\simeq \co\Alg_{\coComm}(\mC)_{\tu/} \simeq \co\Alg_{\bE_\infty}(\mC)_{\tu/}$$ adding the tensor unit.
The tensor algebra $\T$ is the composition
$$ \mC \xrightarrow{\triv_{\coComm^\mathrm{nu}}} \co\Alg_{\coComm^{\mathrm{nu}}}(\mC)\simeq \co\Alg_{\bE_\infty}(\mC)_{\tu/} \xrightarrow{\widetilde{\mathrm{Free}} } \Bialg(\mC).$$

Then $\T$ lifts the free associative algebra $\mC \to \Alg(\mC)$ along the forgetful functor
$\Bialg(\mC) \to \Alg(\mC).$

The functor of primitives $\Prim$ is the composition $$ \Bialg(\mC) \xrightarrow{\mathrm{forget}} \co\Alg_{\bE_\infty}(\mC)_{\tu/} \simeq \co\Alg_{\coComm^{\mathrm{nu}}}(\mC) \xrightarrow{\triv *^{\coComm^\mathrm{nu}}(-)} \mC.$$

Hence there is an adjunction 
$\T: \mC \rightleftarrows \Bialg(\mC):\Prim.$

\end{proof}

\section{Koszul duality}\label{Koszull}


In this section we study Koszul duality between $\infty$-operads and $\infty$-cooperads
and their respective $\infty$-categories of algebras and coalgebras.
We construct this Koszul duality from Koszul-duality between augmented 
associative algebras and coaugmented coassociative coalgebras and their respective $\infty$-categories of modules and comodules following \cite[5.2.]{lurie.higheralgebra}, \cite{https://doi.org/10.48550/arxiv.2104.03870}.
The next two theorems are \cite[Theorem 3.26.]{https://doi.org/10.48550/arxiv.2104.03870}:


\begin{theorem}\label{thh}\label{ccojjkk}
Let $\mC$ be a monoidal $\infty$-category and $\mD$ an $\infty$-category
left tensored over $\mC$ such that $ \mC, \mD $ admit geometric realizations. 
There is a 
commutative square
\begin{equation}\label{fghjkk}
\begin{xy}
\xymatrix{
\LMod(\mD) \ar[d] 
\ar[rrr]^{\Barc_\mD} 
&&& \mathrm{coLMod}(\mD):=\LMod(\mD^\op)^\op \ar[d] 
\\
\Alg(\mC)_{/\tu}  \ar[rrr]^{\Barc_\mC}  &&& \co\Alg( \mC)_{\tu/}:= (\Alg(\mC^\op)_{/\tu})^\op.
}
\end{xy} 
\end{equation} 
The bottom functor lifts the functor 
$$ \Alg(\mC)_{/\tu} \to \mC, \ \A \mapsto \tu \ot_\A \tu.$$
The top functor preserves cocartesian lifts of the vertical functors.

Square (\ref{fghjkk}) induces on the fiber over every $\A \in \Alg(\mC)_{/\tu} $ 
a functor $$ \LMod_\A( \mD) \to \mathrm{coLMod}_{\mathrm{Bar}_\mC(\A)}(\mD) $$ lifting the functor 
$ \tu \ot_\A (-): \LMod_\A(\mD) \to \mD.$
\end{theorem}

\begin{theorem}\label{Baar}
Let $\mC$ be a monoidal $\infty$-category and $\mD$ an $\infty$-category left tensored over $\mC$ such that $ \mC, \mD $ admit geometric realizations and totalizations.
There are adjunctions $$ \Barc_\mC: \Alg(\mC)_{/\tu} \rightleftarrows \co\Alg( \mC)_{\tu / } : \Cobar_\mC:=\Barc^\op_{\mC^\op}, $$
$$\Barc_\mD: \BMod(\mD) \rightleftarrows \mathrm{coBMod}(\mD): \Cobar_\mD:=\Barc^\op_{\mD^\op} $$ and square (\ref{fghjkk}) is a map of adjunctions
and so induces for any augmented algebra $\A$ in $\mC$ and map $ \Barc_\mC(\A) \to \A' $ of coaugmented coalgebras in $\mC$ an adjunction
$$ \tu \ot_\A (-): {\LMod_\A(\mD)} \rightleftarrows {\mathrm{LMod}_{\A'}(\mD)}:  \tu \ot^{\A'} (-). $$

\end{theorem}

\begin{remark}\label{contin}
Let $\mC$ be a monoidal $\infty$-category and $\mD$ an $\infty$-category
left tensored over $\mC$ compatible with geometric realizations.
For every $\Y \in \mD$ the canonical morphism 
$$  \tu \ot_\A \mathrm{triv}_{\A} (\Y) \simeq \colim_{\n \in \Delta^\op}(\A^{\ot \n} \ot \Y) \to \Barc_\mC(\A) \ot \Y \simeq \colim_{\n \in \Delta^\op}(\A^{\ot \n}) \ot \Y $$ is an equivalence.
This implies that the functor $$ \LMod_\A(\mD) \to {\mathrm{coLMod}_{ \mathrm{Bar}_\mC(\A)}( \mD)} $$ exhibits
the $\infty$-category ${\mathrm{coLMod}_{ \mathrm{Bar}_\mC(\A)}( \mD)}  $ as the $\infty$-category of coalgebras 
over the comonad associated to the adjunction $$ \tu \ot_\A (-):  { \LMod_\A( \mD )} \rightleftarrows \mD : \mathrm{triv}_{\A}.$$ 

\end{remark}

\vspace{1mm}

To define operadic Koszul duality we apply Theorem \ref{Baar} 
to the $\infty$-category of symmetric sequences endowed with the composition product.
For any symmetric monoidal $\infty$-category $\mC$ compatible with small colimits
by Construction \ref{leftact} the monoidal $\infty$-category $\sSeq(\mC)$ endowed with the composition product acts on $\mC$ from the left.
Restricting this left action to the full monoidal subcategory $\sSeq(\mC)_{\geq 1} \subset \sSeq(\mC)$ we obtain a left action of $\sSeq(\mC)_{\geq 1}$ on $\mC$, to which we can apply
Theorem \ref{Baar} if $\mC$ admits totalizations, to obtain the following:

\vspace{1mm}

\begin{proposition}\label{dfghcfghkk}

Let $\mC$ be a preadditive symmetric monoidal $\infty$-category compatible with  small colimits that admits totalizations.

\begin{enumerate}
\item There is a Koszul duality adjunction
$$(-)^\vee: \Op(\mC)^\nun_{/\triv} \simeq \Alg(\sSeq(\mC)_{\geq 1})_{/\triv} \rightleftarrows \co\Alg(\sSeq(\mC)_{\geq 1})_{\triv / } \simeq \mathrm{coOp}(\mC)^\nun_{\triv / }: (-)^\vee. $$

\item There is an adjunction 
$ \LMod(\mC) \rightleftarrows \co\LMod(\mC) $ and a map of adjunctions 
\begin{equation*}
\begin{xy}
\xymatrix{
\LMod( \mC)  \ar[d] 
\ar[r]^{ } 
&  \co\LMod( \mC ) \ar[d] 
\\
\Op(\mC)^\nun_{/\triv}  \ar[r]^{ }  & \mathrm{coOp}(\mC)^\nun_{\triv / }.
}
\end{xy} 
\end{equation*} 

For any augmented non-unital $\infty$-operad $\mO $ in $\mC$ and morphism $ \mO^\vee \to \mQ $ of coaugmented non-counital $\infty$-cooperads in $\mC$ this square induces an adjunction
$$ \Alg_\mO(\mC)= \LMod_\mO( \mC) \rightleftarrows \co\Alg^{\mathrm{dp},\conil}_{\mQ}(\mC) = \co\LMod_{ \mQ }(\mC), $$ where the left adjoint lifts the functor 
$ \triv \circ_\mO (-): \Alg_\mO( \mC) \to \mC$ left adjoint to the trivial $\mO$-algebra functor and the right adjoint lifts the functor 
$ \triv \circ^{\mQ} (-): \co\Alg^{\mathrm{dp},\conil}_{\mQ}(\mC) \to \mC$ right adjoint to the trivial $\mQ$-coalgebra functor. 

\end{enumerate}

\end{proposition}

Next we prove an extension of Proposition \ref{dfghcfghkk}, which provides Koszul duality for a preadditive symmetric monoidal $\infty$-category $\mC$ that admits small colimits but whose
tensor product does not necessarily preserve small colimits component-wise.
We need this extension to construct Koszul duality for the opposite monoidal $\infty$-category $\mC^\op$, whose tensor product rarely commutes with small colimits component-wise.
We need Koszul duality for the opposite monoidal $\infty$-category $\mC^\op$ to state and prove Theorem \ref{map}.

Let $\mC$ be a symmetric monoidal $\infty$-category that admits small colimits (but whose
tensor product does not necessarily preserve small colimits component-wise).
Then the composition product on  $\sSeq(\mC)_{\geq 1}$ does not define a monoidal $\infty$-category but a lax monoidal $\infty$-category.
In this case we cannot form the Bar-Cobar adjunctions for $\sSeq(\mC)_{\geq 1}$ directly. By embedding $\mC$ symmetric monoidally into a preadditive symmetric monoidal $\infty$-category compatible with small colimits that admits totalizations we can construct the Bar-Cobar adjunctions for $\sSeq(\mC)_{\geq 1}$ in this more general case via the following Proposition:

\begin{proposition}\label{dfghcfghkkk}

Let $\mC$ be a preadditive symmetric monoidal $\infty$-category that admits small colimits and small limits.

\begin{enumerate}
\item There is an adjunction
$$ (-)^\vee: \Op(\mC)^\nun_{/\triv} \rightleftarrows \mathrm{coOp}(\mC)^\nun_{\triv / }: (-)^\vee. $$

\item There is an adjunction
$$ \RMod(\sSeq(\mC)_{\geq 1}) \rightleftarrows \mathrm{coRMod}(\sSeq(\mC)_{\geq 1}) $$ and a map of adjunctions 
\begin{equation*}
\begin{xy}
\xymatrix{
\RMod(\sSeq(\mC)_{\geq 1})  \ar[d] 
\ar[r]^{ } 
&  \mathrm{coRMod}(\sSeq(\mC)_{\geq 1})  \ar[d] 
\\
\Op(\mC)^\nun_{/\triv}  \ar[r]^{ }  &  \mathrm{coOp}(\mC)^\nun_{\triv / }.
}
\end{xy} 
\end{equation*} 
For any augmented non-unital $\infty$-operad $\mO $ in $\mC$ and morphism $ \mO^\vee \to \mQ $ of coaugmented non-counital $\infty$-cooperads in $\mC$ this square induces an adjunction
$$\RMod_\mO(\sSeq(\mC)_{\geq 1}) \rightleftarrows \co\RMod_{ \mQ }(\sSeq(\mC)_{\geq 1}), $$ where the left adjoint lifts the functor 
$ (- )\circ_\mO \triv: \RMod_\mO(\sSeq(\mC)_{\geq 1}) \to \mC$ left adjoint to the trivial right $\mO$-module functor and the right adjoint lifts the functor 
$ (-) \circ^{\mQ} \triv:  \co\RMod_{ \mQ }(\sSeq(\mC)_{\geq 1}) \to \mC$ right adjoint to the trivial right $\mQ$-comodule.

\vspace{1mm}
\item For any augmented non-unital $\infty$-operad $\mO $ in $\mC$ and morphism $ \mO^\vee \to \mQ $ of coaugmented non-counital $\infty$-cooperads in $\mC$ there is an adjunction
$$\triv \circ_\mO (-): \Alg_\mO(\mC) \rightleftarrows \co\Alg_{\mQ}(\mC): \triv *^{\mQ} (-), $$ where the left adjoint lifts the functor 
$ \triv \circ_\mO (-): \Alg_\mO(\mC) \to \mC$ left adjoint to the trivial $\mO$-algebra functor and the right adjoint lifts the functor 
$ \triv *^{\mQ} (-): \co\Alg_{\mQ}(\mC) \to \mC$ right adjoint to the trivial $\mQ$-coalgebra functor. 

\end{enumerate}

\end{proposition}

\begin{proof}
	
Let us first assume that $\mC$ is a presentably symmetric monoidal $\infty$-category. Then the composition product on $\sSeq(\mC)$ defines a monoidal structure. In this case the statements 1. and 2. follow from Proposition \ref{Baar}.

Moreover for any augmented non-unital $\infty$-operad $\mO $ in $\mC$ and morphism $ \mO^\vee \to \mQ $ of coaugmented non-counital $\infty$-cooperads in $\mC$ there is an adjunction
$$\triv \circ_\mO (-): \Alg_\mO(\mC) \rightleftarrows \co\Alg^{\mathrm{dp}, \conil}_{\mQ}(\mC): \triv \circ^{\mQ} (-), $$ where the left adjoint lifts the functor 
$ \triv \circ_\mO (-): \Alg_\mO(\mC) \to \mC$ left adjoint to the trivial $\mO$-algebra functor and the right adjoint lifts the functor 
$ \triv \circ^{\mQ} (-): \co\Alg^{\mathrm{dp}, \conil}_{\mQ}(\mC) \to \mC$ right adjoint to the trivial $\mQ$-coalgebra functor.
The forgetful functor $\co\Alg^{\mathrm{dp}, \conil}_{\mQ}(\mC) \to \co\Alg_{\mQ}(\mC) $
admits a right adjoint $\R.$
The functor
$$\Alg_\mO(\mC) \xrightarrow{\triv \circ_\mO (-)}\co\Alg^{\mathrm{dp}, \conil}_{\mQ}(\mC) \to \co\Alg_{\mQ}(\mC)$$
is left adjoint to the functor
$$\triv *^{\mQ} (-): \co\Alg_{\mQ}(\mC) \xrightarrow{\R}\co\Alg^{\mathrm{dp}, \conil}_{\mQ}(\mC) \xrightarrow{\triv \circ^\mQ (-)} \Alg_\mO(\mC).$$
The latter functor lifts the functor $\co\Alg_{\mQ}(\mC) \to \mC$ right adjoint to the composition $$\mC \xrightarrow{\triv^{\mathrm{dp},\conil}_\Q} \co\Alg^{\mathrm{dp}, \conil}_{\mQ}(\mC) \to \co\Alg_{\mQ}(\mC),$$ which identifies with the trivial
$\mQ$-coalgebra functor $\triv_\mQ.$
Hence the functor $$\triv *^{\mQ} (-): \co\Alg_{\mQ}(\mC)\to\Alg_\mO(\mC)$$
lifts the functor $\triv *^{\mQ} (-): \co\Alg_{\mQ}(\mC) \to \mC$ right adjoint to $\triv_\mQ.$

Now we treat the general case. By Corollary \ref{cory} there are symmetric monoidal embeddings $\mC \subset \mC' \subset \mC''$
with preadditive symmetric monoidal $\infty$-categories $\mC', \mC''$ compatible with small colimits such that the embedding $\mC \subset \mC'$ admits a left adjoint $\L$, the embedding $\mC' \subset \mC''$ preserves small colimits, the embedding $\mC \subset \mC''$ preserves small limits and $\mC''$ is presentable.
The embeddings $\mC \subset \mC', \mC' \subset \mC''$
induce embeddings 
$$\Op(\mC) \subset \Op^{\mathrm{}}(\mC') \subset \Op^{\mathrm{}}(\mC''), $$$$ \mathrm{coOp}(\mC) \subset \mathrm{coOp}(\mC') \subset \mathrm{coOp}(\mC''), $$
$$ \RMod_\mO(\sSeq(\mC)_{\geq 1}) \subset \RMod_\mO(\sSeq(\mC')_{\geq 1}) \subset \RMod_\mO(\sSeq(\mC'')_{\geq 1}), $$
$$\mathrm{co}\RMod_{\mQ}(\sSeq(\mC)_{\geq 1}) \subset \mathrm{co}\RMod_{\mQ}(\sSeq(\mC')_{\geq 1}) \subset \mathrm{co}\RMod_{\mQ}( \sSeq(\mC'')_{\geq 1}),$$
$$\Alg_\mO(\mC) \subset \Alg_\mO(\mC') \subset \Alg_\mO(\mC''),$$
$$\co\Alg_\mQ(\mC) \subset \co\Alg_\mQ(\mC') \subset \co\Alg_\mQ(\mC'').$$

\vspace{1mm}

The induced embedding $\sSeq(\mC)_{\geq 1} \subset\sSeq(\mC'')_{\geq 1}$ is monoidal with respect to cocomposition product since the embedding $\mC \subset \mC''$ preserves small limits.
The induced embedding $\sSeq(\mC')_{\geq 1} \subset\sSeq(\mC'')_{\geq 1}$ is monoidal with respect to the composition product since the embedding $\mC \subset \mC''$ preserves small colimits.
By Lemma \ref{gghjjgh} the cocomposition product on $\sSeq(\mC'')_{\geq 1}$ agrees with the composition product so that the cocomposition product on $\sSeq(\mC'')_{\geq 1}$ restricts to $\sSeq(\mC')_{\geq 1}$ and identifies with the composition product on $\sSeq(\mC')_{\geq 1}$.
Hence the right adjoint embedding $\sSeq(\mC)_{\geq 1} \subset\sSeq(\mC')_{\geq 1}$ is monoidal with respect to cocomposition product.
Thus the embeddings 
$$ \mathrm{coOp}(\mC)_{\triv / } \subset  \mathrm{coOp}(\mC')_{\triv / }, $$$$ \mathrm{coRMod}_{\mQ}(\sSeq(\mC)_{\geq 1}) \subset \mathrm{coRMod}_{\mQ}(\sSeq(\mC')_{\geq 1}),$$
$$\co\Alg_\mQ(\mC) \subset \co\Alg_\mQ(\mC') $$ admit left adjoints $\L',\L'', \L''' $ respectively, that forget to the functors $\sSeq(\L)_{\geq 1}, \L,$ respectively.

Since the symmetric monoidal structure on $\mC''$ is compatible with small colimits,
we obtain the adjunctions of 1., 2. and 3. for $\mC''.$ 
The right adjoints $$ (-)^\vee: \mathrm{coOp}(\mC'')_{\triv / } \to \Op(\mC'')^{\mathrm{}}_{/\triv}, \ \X \mapsto \triv \ast^\X \triv $$$$  (-) \ast^{\mQ} \triv : \mathrm{coRMod}_{\mQ}(\sSeq(\mC'')_{\geq 1}) \to \mathrm{RMod}_{\mO}(\sSeq(\mC'')_{\geq 1}), $$
$$\triv *^{\mQ} (-): \co\Alg_\mQ(\mC'') \to \Alg_\mO(\mC'') $$
restrict to functors
$ (-)^\vee: \mathrm{coOp}(\mC)_{\triv / }\to \Op(\mC)^{\mathrm{}}_{/\triv}, $
$$ (-) \ast^{\mQ} \triv  :\mathrm{coRMod}_{\mQ}(\sSeq(\mC)_{\geq 1}) \to \mathrm{RMod}_{\mO}(\sSeq(\mC)_{\geq 1}), $$
$$\triv *^{\mQ} (-): \co\Alg_\mQ(\mC) \to \Alg_\mO(\mC) $$
since $\mC$ is closed in $\mC''$ under small limits. The left adjoints 
$$ (-)^\vee: \Op(\mC'')^{\mathrm{}}_{/\triv} \to \mathrm{coOp}(\mC'')_{\triv / }, $$$$(-) \circ_{\mO} \triv : \mathrm{RMod}_{\mO}(\sSeq(\mC'')_{\geq 1}) \to \mathrm{coRMod}_{\mQ}(\sSeq(\mC'')_{\geq 1}), $$
$$\triv \circ_\mO (-): \Alg_\mO(\mC'') \rightleftarrows \co\Alg_{\mQ}(\mC'')$$  
restrict to functors
$$ (-)^\vee: \Op(\mC')^{\mathrm{}}_{/\triv} \to \mathrm{coOp}(\mC')_{\triv / },$$
$$ (-) \circ_{\mO} \triv : \mathrm{RMod}_{\mO}(\sSeq(\mC')_{\geq 1}) \to \mathrm{coRMod}_{\mQ}(\sSeq(\mC')_{\geq 1}), $$ 
$$\triv \circ_\mO (-): \Alg_\mO(\mC') \rightleftarrows \co\Alg_{\mQ}(\mC')$$  
since $\mC'$ is closed in $\mC''$ under small colimits.
Thus the composition $$ \Op(\mC)^{\mathrm{}}_{/\triv} \subset \Op(\mC')^{\mathrm{}}_{/\triv} \xrightarrow{(-)^\vee} \mathrm{CoOp}(\mC')_{\triv / } \xrightarrow{\L'}  \mathrm{coOp}(\mC)_{\triv / }$$ is left adjoint to the functor
$(-)^\vee: \mathrm{coOp}(\mC)_{\triv / } \to \Op(\mC)^{\mathrm{}}_{/\triv}, $ the composition $$ \RMod_{\mO}(\sSeq(\mC)_{\geq 1})  \subset \RMod_{\mO}(\sSeq(\mC')_{\geq 1}) \xrightarrow{(-) \circ_{\mO} \triv } \mathrm{coRMod}_{\mQ}(\sSeq(\mC')_{\geq 1})$$$$ \xrightarrow{\L''} \mathrm{coRMod}_{\mQ}(\sSeq(\mC)_{\geq 1}) $$ is left adjoint to the functor
$(-) \ast^{\mQ} \triv :  \mathrm{coRMod}_{\mQ}(\sSeq(\mC)_{\geq 1}) \to \RMod_{\mO}(\sSeq(\mC)_{\geq 1}) $
and the composition $$ \Alg_{\mO}(\mC)  \subset \Alg_{\mO}(\mC')  \xrightarrow{\triv \circ_{\mO} (-)} \co\Alg_{\mQ}(\mC') \xrightarrow{\L'''} \co\Alg_{\mQ}(\mC) $$ is left adjoint to the functor
$\triv \ast^{\mQ} (-) : \co\Alg_{\mQ}(\mC) \to \Alg_{\mO}(\mC).$

\end{proof}

The next Proposition is well-known to the experts and we give a proof for the reader's convenience since we could not find any reference.

\begin{proposition}\label{eqqqqq}

Let $\mC$ be a stable symmetric monoidal $\infty$-category compatible with small colimits that admits totalizations. 	

\begin{enumerate}
\item Koszul-duality 
$$ (-)^\vee: \Op(\mC)^\mathrm{} \rightleftarrows \mathrm{CoOp}(\mC): (-)^\vee $$
restricts to an equivalence
$$ \widetilde{\Op}(\mC) \simeq  \widetilde{\mathrm{CoOp}}(\mC).$$

\item For every non-unital $\infty$-operad $\mO$ in $\mC $ such that $\mO_1 = \tu$ Koszul-duality
$$ (-) \circ_\mO \triv : \RMod_\mO(\sSeq(\mC)_{\geq 1}) \rightleftarrows  \mathrm{co}\RMod_{\mO^\vee}(\sSeq(\mC)_{\geq 1}): (-) \circ^{\mO^\vee} \triv $$
is an equivalence.

\end{enumerate} 

\end{proposition}

\begin{proof}

The unit of the adjunction of 1. applied to an $\infty$-operad $\mO$ is equivalent to
the unit of the adjunction of 2. applied to $\mO$ considered as module over itself.
Moreover note that $(\mO^\vee)_1 = \tu $ if $ \mO_1 = \tu$, which follows from the proof of Lemma \ref{fghjjhbv}. 

Hence 1. follows from 2. and the following statement:
for every  $\infty$-cooperad $\mQ$ in $\mC $ such that $\mQ_1 = \tu $ the Koszul-duality adjunction 
\begin{equation}\label{fghghjk}
(-) \circ_{\mQ^\vee} \triv : \RMod_{\mQ^\vee}(\sSeq(\mC)_{\geq 1}) \rightleftarrows  \mathrm{co}\RMod_{\mQ}(\sSeq(\mC)_{\geq 1}): (-) \circ^{\mQ} \triv 
\end{equation}
induced by forgetting along the counit $(\mQ^\vee)^\vee \to \mQ$
is an equivalence.

So we need to see that the unit and counit of the adjunction of 2. and of the adjunction (\ref{fghghjk}) are equivalences.
We will show that the unit $\eta $ of the adjunction of 2. is an equivalence.
The case of the counit of the adjunction of 2. and the other cases are similar.

We will prove that for every $\bk \geq 1 $ the following statement ($\ast$) holds:
for every $\n  \geq 1 $ and every right $\mO$-module $\X$ that vanishes below degree $\n $, i.e.
that belongs to $\RMod_\mO(\sSeq(\mC)_{\geq \n}), $ the following morphism is an equivalence: $$(\eta_\X)_\bk : \X_\bk \to ((\X \circ_\mO \triv ) \circ^{\mO^\vee} \triv )_\bk.$$

By Lemma \ref{fghfghkm} the objects $ \X \circ_\mO \triv$ and so
$ (\X \circ_\mO \triv ) \circ^{\mO^\vee} \triv $ belong to $\sSeq(\mC)_{\geq \n}$. Hence $\eta_\bk$ is an equivalence if $\n > \bk$.
Thus it remains to show ($\ast$) for $ \n \leq \bk, $ 
which we prove by descending induction on $\n.$ 
By Lemma \ref{fghfghjjjl} there is an $\X' \in \RMod_\mO(\sSeq(\mC)_{\geq \n+1})$ and a morphism $\X' \to \X $ in $ \RMod_\mO(\sSeq(\mC)_{\geq \n})$ that induces an equivalence in degrees larger than $\n. $ By Lemma \ref{dfghfghjfghj} the cofiber $\X''$ of the morphism $\X' \to \X $ in the stable $\infty$-category $ \RMod_\mO(\sSeq(\mC)_{\geq 1})$ is the trivial right $\mO$-module concentrated in degree $\n$ with value $\X_\n $ since $\mO_1=\tu$.
There is a commutative square
\begin{equation*}
\begin{xy}
\xymatrix{
\X'  \ar[d]^{\eta_{\X'}}  
\ar[r] 
& 	\X \ar[d]^{\eta_\X} 	\ar[r]   	& 	\X'' \ar[d]^{\eta_{\X''}} 
\\
(\X' \circ_\mO \triv ) \ast^{\mO^\vee} \triv \ar[r]  & (\X \circ_\mO \triv ) \ast^{\mO^\vee} \triv \ar[r]  	& 	(\X'' \circ_\mO \triv ) \ast^{\mO^\vee} \triv }
\end{xy} 
\end{equation*}  
in $ \RMod_\mO(\sSeq(\mC)_{\geq 1})$, 
where bottom and top morphisms are  cofiber sequences.

So by the induction hypothesis we are reduced to show that
$\eta_\X $ is an equivalence if $\X$ carries the trivial right $\mO$-module structure.
By Proposition \ref{dfghcfghkk} the left adjoint of adjunction 2. lifts the functor 
$ (- )\circ_\mO \triv: \RMod_\mO(\sSeq(\mC)_{\geq 1}) \to \mC$ left adjoint to the trivial right $\mO$-module functor.
So by adjointness the right adjoint $$(-) \circ^{\mO^\vee} \triv : \mathrm{co}\RMod_{\mO^\vee}(\sSeq(\mC)_{\geq 1})  \to \RMod_\mO(\sSeq(\mC)_{\geq 1})$$ of the adjunction of 2. sends cofree right $\mO^\vee$-comodules to trivial right $\mO$-modules.

On the other hand the left adjoint of adjunction 2.
sends trivial right $\mO$-modules to cofree right $\mO^\vee$-comodules since the composition product on $\sSeq(\mC)_{\geq 1}$ preserves small sifted colimits in each component.


\end{proof}

\begin{corollary}\label{lies}
	
Let $\mC, \mD$ be stable symmetric monoidal $\infty$-categories compatible with small colimits that admit totalizations, $\phi: \mC \to \mD$ a 
symmetric monoidal functor that preserves small colimits and $\mQ$ a non-counital
$\infty$-cooperad in $\mC.$
There is a canonical equivalence of non-unital $\infty$-operads in $\mD:$
$$ \phi_*(\mQ)^\vee \simeq \phi_*(\mQ^\vee).$$

\end{corollary}

\begin{lemma}\label{djfghjdfh}

Let $\mC$ be a stable symmetric monoidal $\infty$-category compatible with small colimits that admits totalizations.
Let $\mO$ be a non-unital $\infty$-operad in $\mC$ such that $\mO_1 = \tu$.
For every $\n \in \bN$ there is a canonical equivalence of right $\mO$-modules:
$$\triv \circ_{\f_\n(\mO)} \mO \simeq  \tau_\n(\mO^\vee)\circ^{\mO^\vee} \triv.$$

\end{lemma}

\begin{proof}

Let $\psi: \f_\n(\mO) \to \mO $ be the canonical map of $\infty$-operads. 	

By Theorem \ref{ccojjkk} there is a commutative square of $\infty$-categories
\begin{equation*}
\begin{xy}
\xymatrix{
\co\RMod_{ \mO^\vee }(\sSeq(\mC)_{\geq 1}) \ar[d]^{(\psi^\vee)^\ast} 
\ar[r]
& \RMod_{ \mO }(\sSeq(\mC)_{\geq 1}) \ar[d]^{\psi^\ast}   
\\
\co\RMod_{\tau_\n(\mO^\vee) }(\sSeq(\mC)_{\geq 1}) \ar[r] \ar[d]^{(\psi^\vee)_\ast} &\RMod_{\f_\n(\mO)}(\sSeq(\mC)_{\geq 1}) \ar[d]^{\psi_\ast} 
\\
\co\RMod_{\mO^\vee }(\sSeq(\mC)_{\geq 1})  \ar[r]  & \RMod_{ \mO }(\sSeq(\mC)_{\geq 1}).
}
\end{xy} 
\end{equation*} 

The diagonal of this square sends $\mO^\vee$ to both objects we want to identify.

\end{proof}

\begin{lemma}\label{dghcvbhhh}

Let $\mC$ be a stable symmetric monoidal $\infty$-category that admits small colimits and totalizations.
Let $\mQ$ be a non-counital $\infty$-cooperad in $\mC$ such that $\mQ_1 = \tu$. 
For every $\A \in \Alg_{\mQ^\vee}(\mC)$ there is a canonical equivalence in $\mC:$
$$ \triv \circ_{\mQ^\vee} \A \simeq \colim_{\n \geq 1}((\tau_\n(\mQ) \circ^{\mQ} \triv) \circ_{\mQ^\vee} \A).$$

\end{lemma}

\begin{proof}We first assume that $\mC$ is a stable presentably symmetric monoidal $\infty$-category.
In this case by Remark \ref{remfs} for every $\n \geq 1$ the functor
$\tau_\n: \Op(\mC)^\nun \to \Op(\mC)^\nun_{\leq \n} \subset \Op(\mC)^\nun$ admits a left adjoint $\f_\n$.
By \cite[Remark 4.23.]{2018arXiv180306325H} the canonical map of $\infty$-operads $\colim_{\n \geq 1} \f_\n(\mQ^\vee) \to \mQ^\vee $ is an equivalence. 
By Lemma \ref{leuj} this implies that the canonical map 
$$ \colim_{\n \geq 1} \mQ^\vee \circ_{\f_\n(\mQ^\vee)} \A \to \A $$ of $\mQ^\vee$-algebras in $\mC$ is an equivalence. 
Hence the canonical map $$ \colim_{\n \geq 1}\triv \circ_{\f_\n(\mQ^\vee)} \A \simeq \triv \circ_{\mQ^\vee} (\colim_{\n \geq 1} \mQ^\vee \circ_{\f_\n(\mQ^\vee)} \A) \to \triv \circ_{\mQ^\vee} \A $$ is an equivalence.
By Lemma \ref{djfghjdfh} and Proposition \ref{eqqqqq} there is a canonical equivalence $\triv \circ_{\f_\n(\mQ^\vee)} \mQ^\vee \simeq \tau_\n(\mQ)\circ^{\mQ} \triv$ of right $\mQ^\vee$-modules.
So we obtain a canonical equivalence
$$\theta_\mC: (\tau_\n(\mQ)\circ^{\mQ} \triv) \circ_{\mQ^\vee} \A\simeq (\triv \circ_{\f_\n(\mQ^\vee)} \mQ^\vee) \circ_{\mQ^\vee} \A \to  \triv \circ_{\f_\n(\mQ^\vee)} \A$$ in $\mC$.
This proves the claim if $\mC$ is compatible with small colimits.

Now let $\mC$ be an arbitrary stable symmetric monoidal $\infty$-category
that admits small colimits and totalizations.
By Corollary \ref{cory} there are symmetric monoidal exact embeddings
$\iota: \mC \subset \mC', \bj: \mC' \subset \mC''$ into stable symmetric monoidal $\infty$-categories, where $\mC''$ is a presentably symmetric monoidal
$\infty$-category, $\mC'$ is compatible with small colimits, the first embedding $\iota: \mC \subset \mC'$ admits a left adjoint $\L$, the second embedding $\bj:\mC' \subset \mC''$ preserves small colimits and the whole embedding $\bj \circ \iota:\mC \subset \mC''$ preserves limits.
We obtain induced embeddings 
$$\iota: \coOp(\mC) \subset \coOp(\mC'), \bj: \coOp(\mC') \subset \coOp(\mC''),$$
$$\iota: \Alg_{\mQ^\vee}(\mC) \subset \Alg_{\mQ^\vee}(\mC'), \bj: \Alg_{\mQ^\vee}(\mC') \subset \subset \Alg_{\mQ^\vee}(\mC'')$$
$$\iota: \RMod_{\mQ^\vee}(\sSeq(\mC)_{\geq 1}) \subset \RMod_{\mQ^\vee}(\sSeq(\mC')_{\geq 1}),$$$$\bj: \RMod_{\mQ^\vee}(\sSeq(\mC')_{\geq 1}) \subset \RMod_{\mQ^\vee}(\sSeq(\mC'')_{\geq 1}),$$
$$\iota: \co\RMod_{\mQ}(\sSeq(\mC)_{\geq 1}) \subset \co\RMod_{\mQ}(\sSeq(\mC')_{\geq 1}), $$$$\bj: \co\RMod_{\mQ}(\sSeq(\mC')_{\geq 1}) \subset \co\RMod_{\mQ}(\sSeq(\mC'')_{\geq 1}).$$

Since the embedding $\bj\circ \iota: \mC \subset \mC''$ preserves limits,
the functor $$ (-) \circ^{\bj\iota(\mQ)}\bj\iota(\triv) \simeq (-) \ast^{\bj\iota(\mQ)}\bj\iota(\triv): \co\RMod_{\mQ}(\sSeq(\mC'')_{\geq 1}) \to \RMod_{\mQ^\vee}(\sSeq(\mC'')_{\geq 1}) $$
restricts to the functor $$ (-) \circ^{\mQ}\triv \simeq  (-) *^{\mQ}\triv: \co\RMod_{\mQ}(\sSeq(\mC)_{\geq 1}) \to \RMod_{\mQ^\vee}(\sSeq(\mC)_{\geq 1}).$$

Let $\T:=\tau_\n(\mQ) \circ^{\mQ} \triv$.
Consequently, there is a canonical equivalence: $$\bj\iota(\T) \simeq \bj\iota(\tau_\n(\mQ)) \circ^{\bj\iota(\mQ)} \bj\iota(\triv) \simeq \tau_\n(\bj\iota(\mQ)) \circ^{\bj\iota(\mQ)} \bj\iota(\triv).$$

By the first part of the proof there is a canonical equivalence in $\mC'':$
$$ \bj\iota(\triv) \circ_{\bj\iota(\mQ)^\vee} \bj\iota(\A) \simeq \colim_{\n \geq 1}((\tau_\n(\bj\iota(\mQ)) \circ^{\bj\iota(\mQ)} \bj\iota(\triv)) \circ_{\bj\iota(\mQ)^\vee} \bj\iota(\A))$$
$$ \simeq \colim_{\n \geq 1}(\bj\iota(\T) \circ_{\bj\iota(\mQ)^\vee} \bj\iota(\A)).$$

Since the embedding $\bj: \mC' \subset \mC''$ preserves small colimits, 
there are canonical equivalences
$$ \bj\iota(\triv) \circ_{\bj\iota(\mQ)^\vee} \bj\iota(\A) \simeq \bj\iota(\triv) \circ_{\bj\iota(\mQ^\vee)} \bj\iota(\A) \simeq \bj(\iota(\triv) \circ_{\iota(\mQ^\vee)} \iota(\A)), $$
$$\bj\iota(\T) \circ_{\bj\iota(\mQ)^\vee} \bj\iota(\A) \simeq \bj\iota(\T) \circ_{\bj\iota(\mQ^\vee)} \bj\iota(\A) \simeq \bj (\iota(\T) \circ_{\iota(\mQ^\vee)} \iota(\A)).$$
We obtain a canonical equivalence 
$$ \bj(\iota(\triv) \circ_{\iota(\mQ^\vee)} \iota(\A)) \simeq \colim_{\n \geq 1}\bj (\iota(\T) \circ_{\iota(\mQ^\vee)} \iota(\A))$$$$\simeq \bj( \colim_{\n \geq 1}(\iota(\T) \circ_{\iota(\mQ^\vee)} \iota(\A))) $$ in $\mC''$
and so a canonical equivalence 
$$ \iota(\triv) \circ_{\iota(\mQ^\vee)} \iota(\A) \simeq \colim_{\n \geq 1}(\iota(\T) \circ_{\iota(\mQ^\vee)} \iota(\A)) $$ in $\mC'.$
Applying the left adjoint $\L: \mC' \to \mC$ we get the desired equivalence in $\mC,$ where the first and last equivalence is by Proposition \ref{lemgf}:
$$\triv\circ_{\mQ^\vee} \A \simeq \L(\iota(\triv) \circ_{\iota(\mQ^\vee)} \iota(\A)) \simeq \L( \colim_{\n \geq 1}(\iota(\T) \circ_{\iota(\mQ^\vee)} \iota(\A)))\simeq $$$$\colim_{\n \geq 1}\L(\iota(\T) \circ_{\iota(\mQ^\vee)} \iota(\A)) \simeq \colim_{\n \geq 1}(\T \circ_{\mQ^\vee} \A). $$


\end{proof}

Dualising the last proposition we obtain the following corollary:

\begin{corollary}\label{dghcvbhhhh}
Let $\mC$ be a stable symmetric monoidal $\infty$-category that admits small limits, $\mO$ a non-unital $\infty$-operad in $\mC$ such that $\mO_1 = \tu$. 
For every $\C \in \co\Alg^\mathrm{dp}_{\mO^\vee}(\mC)$ in $\mC$ there is a canonical equivalence in $\mC$:
$$ \triv *^{\mO^\vee} \C \simeq \lim_{\n \geq 1}((\tau_\n(\mO) \circ_{\mO} \triv) *^{\mO^\vee}  \C).$$	

\end{corollary}

\begin{proof}
We apply Lemma \ref{dghcvbhhh} to the opposite $\infty$-category
while we use that $$ \tau_\n(\mO) *_{\mO} \triv \simeq \tau_\n(\mO) \circ_{\mO} \triv .$$	

\end{proof}

\begin{construction}\label{con}
Let $\mC$ be a stable symmetric monoidal $\infty$-category compatible with small colimits, $\mO$ an non-unital $\infty$-operad in $\mC$ such that $\mO_1 = \tu $ and $\X$ an $\mO$-algebra and $\Y$ a right $\mO$-module.
There is a canonical morphism $\rho:$
$$ \Y \circ_{\mO} \X \to (\Y \circ_{\mO} \triv) \circ^{\mO^\vee} \triv) \circ_{\mO} \X \to (\Y \circ_{\mO} \triv) \circ^{\mO^\vee} (\triv \circ_{\mO} \X) \to (\Y \circ_{\mO} \triv) *^{\mO^\vee} (\triv \circ_{\mO} \X)$$ in $\mC$,
where the first morphism is induced by the unit
$\Y \to (\Y \circ_{\mO} \triv) \circ^{\mO^\vee} \triv$
and the second morphism is the canonical morphism
$$ (\Z \circ^{\mO^\vee} \triv) \circ_{\mO} \X \to \Z \circ^{\mO^\vee} (\triv \circ_{\mO} \X) $$ in $\mC$ for any right $\mO^\vee$-comodule $\Z$ in $\sSeq(\mC).$
\end{construction}

Recall that by \cite[Example 6.1.6.22.]{lurie.higheralgebra}
for every object $\X$ of $\mC$ equipped with an action of a finite group $\G$
there is a norm map $\X_\G \to \X^\G$ from the coinvariants to the invariants,
which we will use in the following.

\begin{lemma}\label{gggg}
Let $\mC$ be a stable symmetric monoidal $\infty$-category
compatible with small colimits that admits small limits, $\mO$ a non-unital $\infty$-operad in $\mC$ such that $\mO_1 = \tu $ and $\X$ an $\mO$-algebra. 
If all norm maps associated to symmetric groups are equivalences, for every $\n \in \bN$ the following morphism is an equivalence:
$$\rho: \tau_\n(\mO) \circ_{\mO} \X \to (\tau_\n(\mO) \circ_{\mO} \triv) *^{\mO^\vee} (\triv \circ_{\mO} \X).$$

\end{lemma}

\begin{proof}

We proceed by induction on $\n \geq 1.$
For $\n=1$ the map $\rho$ identifies with the canonical equivalence
$$\rho: \tau_1(\mO) \circ_{\mO} \X =  \triv \circ_{\mO} \X \simeq \mO^\vee   {*}^{\mO^\vee} (\triv \circ_{\mO} \X) = (\tau_1(\mO) \circ_{\mO} \triv)  {*}^{\mO^\vee} (\triv \circ_{\mO} \X).$$

We continue with the induction step, where we set $\Y:= \triv \circ_\mO \X.$
Because $\mC$ is stable and for any $\Z \in \sSeq(\mC)$
the functor $(-)\circ \Z: \sSeq(\mC) \to \sSeq(\mC)$ is exact, the $\infty$-categories $\RMod_{\mO}(\sSeq(\mC)),  \mathrm{co} \RMod_{\mO^\vee}(\sSeq(\mC))$ are stable and the forgetful functors $$\RMod_{\mO}(\sSeq(\mC)) \to \sSeq(\mC), \mathrm{co} \RMod_{ \mO^\vee }(\sSeq(\mC)) \to \sSeq(\mC)$$ are exact.
Consequently, the following functors are exact: $$ (-) \circ_\mO \X: \RMod_{\mO}(\sSeq(\mC)) \to \sSeq(\mC),$$$$ (-) \circ_{\mO} \triv: \RMod_{\mO}(\sSeq(\mC)) \to \mathrm{co} \RMod_{\mO^\vee }(\sSeq(\mC)), $$$$ (-) {*}^{\mO^\vee} \Y : \mathrm{co} \RMod_{\mO^\vee }(\sSeq(\mC)) \to \sSeq(\mC)$$ 

\vspace{1mm}
The fiber $\T$ of the map $\tau_{\n+1}(\mO) \to \tau_\n(\mO)$
of right $\mO$-modules is the symmetric sequence concentrated in degree $\n+1$ with value $\mO_{\n+1}$ equipped with the trivial right $\mO$-module structure according to Lemma \ref{dfghfghjfghj}.
Applying the exact functor $(-) \circ_\mO \X: \RMod_\mO(\sSeq(\mC)) \to \sSeq(\mC)$ we obtain the fiber sequence 
$$\T \circ \Y= \T \circ (\triv \circ_\mO \X) \to \tau_{\n+1}(\mO) \circ_\mO \X \to \tau_\n(\mO) \circ_\mO \X. $$ 
Applying the exact functor $((-) \circ_\mO \triv) {*}^{\mO^\vee}\Y: \RMod_\mO(\sSeq(\mC)) \to \sSeq(\mC)$ we see that the fiber of the map 
$$(\tau_{\n+1}(\mO) \circ_\mO \triv) {*}^{\mO^\vee}\Y \to (\tau_\n(\mO) \circ_\mO \triv) {*}^{\mO^\vee}\Y $$ identifies with $$(\T \circ_\mO \triv) {*}^{\mO^\vee}\Y \simeq (\T \circ \mO^\vee) {*}^{\mO^\vee}\Y \simeq (\T  {*} \mO^\vee) {*}^{\mO^\vee}\Y \simeq \T  {*} \Y.$$

There is a commutative diagram
\begin{equation*}
\begin{xy}
\xymatrix{
\T \circ \Y \ar[d] \ar[r] 
& \tau_{\n+1}(\mO) \circ_{\mO} \X \ar[d] \ar[r] & \tau_\n(\mO) \circ_{\mO} \X \ar[d] 
\\
\T  {\ast} \Y \ar[r]  & (\tau_{\n+1}(\mO) \circ_{\mO} \triv)  {\ast}^{\mO^\vee}\Y \ar[r] & (\tau_\n(\mO) \circ_{\mO} \triv)  {\ast}^{\mO^\vee} \Y,
}
\end{xy} 
\end{equation*}
where both horizontal sequences are fiber sequences.
The left vertical map $\T \circ \Y \to \T  {\ast} \Y$ identifies with the
norm map $ (\mO_{\n+1} \ot \Y^{\ot \n+1})_{\Sigma_{\n+1}} \to  (\mO_{\n+1} \ot \Y^{\ot \n+1})^{\Sigma_{\n+1}}.$

\end{proof}


We obtain the following proposition:

\begin{theorem}\label{map}
	
Let $\mC$ be a stable presentably symmetric monoidal $\infty$-category, $\mO$ a non-unital $\infty$-operad in $\mC$ such that $\mO_1 = \tu $ and $\X$ an $\mO$-algebra in $\mC$.
There is a commutative square:
\begin{equation*}
\begin{xy}
\xymatrix{\X \ar[d] 
\ar[rr]
&& \triv *^{\mO^\vee}(\triv \circ_\mO \X) \ar[d] 
\\
\lim_{\n \geq 1}(\tau_\n(\mO) \circ_{\mO} \X) \ar[rr] && \lim_{\n \geq 1}((\tau_\n(\mO) \circ_{\mO} \triv) *^{\mO^\vee} (\triv \circ_\mO \X)).
}
\end{xy} 
\end{equation*} 
If all norm maps of objects in $\mC$ with an action of a symmetric group are equivalences,
the right vertical morphism and bottom horizontal morphism are equivalences.
	
\end{theorem}

\begin{proof}

There is a canonical commutative square:

\begin{equation*}
\begin{xy}
\xymatrix{\X \ar[d] 
\ar[rr]
&& \triv *^{\mO^\vee}(\triv \circ_\mO \X) \ar[d] 
\\
\lim_{\n \geq 1}(\tau_\n(\mO) \circ_{\mO} \X) \ar[rr] && \lim_{\n \geq 1}((\tau_\n(\mO) \circ_{\mO} \triv) *^{\mO^\vee} (\triv \circ_\mO \X)).
}
\end{xy} 
\end{equation*} 
The right vertical morphism is an equivalence by Corollary \ref{dghcvbhhhh} for $\C= \triv \circ_\mO \X $.
The bottom horizontal morphism is an equivalence by Lemma \ref{gggg}.	

\end{proof}

\begin{remark}
Recently, Heuts \cite{Heuts} generalizes Theorem \ref{map} to the situation that norm maps are not necessarily invertible.
	
\end{remark}

\section{A derived Milnor-Moore theorem}


\vspace{1mm}

In this section we apply Theorem \ref{map} to deduce a derived version of the
Milnor-Moore theorem (Theorem \ref{thg}).

\begin{definition}

The spectral Lie operad $$\Lie:= \coComm_\Sp^\vee(-1) \in \Op(\Sp) $$ 
is the negative shift of the Koszul-dual of the non-counital cocommutative spectral $\infty$-cooperad $\coComm_\Sp \in \coOp(\Sp).$

\end{definition}

\begin{remark}
Computing the Koszul dual one finds that for every $\n \in \Sigma$ there is a canonical equivalence of spectra
$$ \Lie_\n \simeq \bigoplus_{\n-1!} \tu_\Sp, $$
while the $\Sigma_\n$-action on $\Lie_\n$ is highly non-trivial.
In particular, $\Lie$ is an $\infty$-operad in connective spectra $\Sp_{\geq0}.$
	
\end{remark}


	
Corollary \ref{lies} implies the following remark:

\begin{remark}
	
Let $\mC$ be stable presentably symmetric monoidal $\infty$-category. 
There is a canonical equivalence of $\infty$-operads in $\mC:$
$$\Lie_\mC= \alpha^\mC_!(\Lie) \simeq (\coComm_\mC^\nun)^\vee(-1).$$

In particular, $\Lie_{\Mod_{\rH(\bZ)}(\Sp)}$ is the classical Lie operad.
\end{remark}

\vspace{1mm}

\begin{proposition}\label{caaaaart}
Let $\mC$ be a stable presentably symmetric monoidal $\infty$-category.

\begin{enumerate}
\item The forgetful functor $$\theta: \co\Alg^{\mathrm{dp}, \conil}_{\coComm_\mC^\nun}(\mC) \to \co\Alg_{\coComm_\mC^\nun}(\mC) \simeq \co\Alg_{\bE_\infty}(\mC)_{\tu/} $$ preserves finite products.

\item The functor $$\triv \circ_{\Lie(1)} (-): \Alg_{\Lie(1)}(\mC) \to \co\Alg^{\mathrm{dp}, \conil}_{\coComm_\mC^\nun}(\mC) $$ preserves finite products.

\end{enumerate}

\end{proposition}

\begin{proof}
	
1. We like to see that for any $\X, \Y \in  \co\Alg^{\mathrm{dp}, \conil}_{\coComm_\mC^\nun}(\mC) $ the canonical map
$\theta(\X \times \Y) \to \theta(\X) \times \theta(\Y)$ is an equivalence.
By Lemma \ref{comon} the forgetful functors $\co\Alg_{\coComm_\mC^\nun}(\mC) \to \mC$
and $\co\Alg^{\mathrm{dp}, \conil}_{\coComm_\mC^\nun}(\mC) \to \mC$ are comonadic.
Let $\Sym_{>0}$ be the right adjoint of the forgetful functor $\co\Alg^{\mathrm{dp}, \conil}_{\coComm_\mC^\nun}(\mC) \to \mC$.
By comonadicity any $\X \in \co\Alg^{\mathrm{dp}, \conil}_{\coComm_\mC^\nun}(\mC)$ is the totalization of a cosimplicial object $\X'$ in $\co\Alg^{\mathrm{dp}, \conil}_{\coComm_\mC^\nun}(\mC)$ taking values in the essential image of $\Sym_{>0}$, whose image in $\mC$ splits. Hence $\theta$ preserves the limit of $\X'$ and the canonical map $\theta(\X \times \Y) \to \theta(\X) \times \theta(\Y)$ 
identifies with the map $$\lim(\theta \circ (\X' \times \Y')) \to \lim((\theta \circ \X') \times (\theta \circ \Y')) \simeq \lim(\theta \circ \X') \times \lim(\theta \circ \Y').$$
Consequently, it is enough to check that the functor
$$\Sym: \mC \xrightarrow{\Sym_{>0}} \co\Alg^{\mathrm{dp}, \conil}_{\coComm_\mC^\nun}(\mC) \xrightarrow{\theta}\co\Alg_{\coComm_\mC^\nun}(\mC) \simeq \co\Alg_{\bE_\infty}(\mC)_{\tu/}$$ preserves finite products.
By \cite[Proposition 3.2.4.7.]{lurie.higheralgebra} the object-wise symmetric monoidal structure on $\co\Alg_{\bE_\infty}(\mC)_{\tu/}$ is cartesian.
For any $\X, \Y \in \mC$ the canonical map $\Sym(\X \oplus \Y) \to \Sym(\X) \otimes \Sym(\Y) $ identifies with the canonical equivalence in $\mC.$

2.: In view of (1) it is enough to see that the composition
$$\Alg_{\Lie(1)}(\mC) \xrightarrow{\triv \circ_{\Lie(1)} (-)} \co\Alg^{\mathrm{dp}, \conil}_{\coComm_\mC^\nun}(\mC)  \to \co\Alg_{\coComm_\mC^\nun}(\mC) \simeq \co\Alg_{\bE_\infty}(\mC)_{\tu/} $$ preserves finite products.

In the following we consider $\bN, \bZ $ as categories by viewing them as posets.	
Let $$\gamma: \Fun(\bZ, \co\Alg_{\bE_\infty}(\mC)_{\tu/}) \to \prod_{\bZ} \co\Alg_{\bE_\infty}(\mC)_{\tu/} $$ be the functor that sends a filtered object $\A$ to its associated graded object $(\A_{\bi}/{\A_{\bi-1}})_{\bi \in \bZ}  $.
The functor $\gamma$ restricts to a functor 
$$\gamma': \Fun(\bN, \co\Alg_{\bE_\infty}(\mC)_{\tu/}) \to \prod_{\bN} \co\Alg_{\bE_\infty}(\mC)_{\tu/}.$$

Let $\theta$ be the functor
$$ \Alg_{\Lie(1)}(\mC) \to \Fun(\bN, \Alg_{\Lie(1)}(\mC)), \ \A \mapsto (... \to \mO \circ_{\f_{\n}(\mO)} \A \to \mO \circ_{\f_{\n+1}(\mO)} \A \to ... ).$$

By Remark \ref{dghjjjkkhg} the functor $\theta$ is a section of the functor $$\colim: \Fun(\bN, \Alg_{\Lie(1)}(\mC)) \to \Alg_{\Lie(1)}(\mC).$$
Moreover by \cite[Remark 4.23.]{2018arXiv180306325H} the composition $$ \Alg_{\Lie(1)}(\mC) \xrightarrow{\theta} \Fun(\bN, \Alg_{\Lie(1)}(\mC)) \xrightarrow{\gamma'} \prod_{\bN} \Alg_{\Lie(1)}(\mC)  \xrightarrow{\oplus} \Alg_{\Lie(1)}(\mC) $$ 
factors as the forgetful functor $\Alg_{\Lie(1)}(\mC) \to \mC$ followed by the trivial Lie algebra functor $\mC \to \Alg_{\Lie(1)}(\mC).$
Thus the functor $ \triv \circ_{\Lie(1)} (-) : \Alg_{\Lie(1)}(\mC) \to \co\Alg_{\bE_\infty}(\mC)_{\tu/} $ factors as
$$ \Alg_{\Lie(1)}(\mC) \xrightarrow{\theta} \Fun(\bN, \Alg_{\Lie(1)}(\mC)) \xrightarrow{(\triv \circ_{\Lie(1)} (-))_*} \Fun(\bN, \co\Alg_{\bE_\infty}(\mC)_{\tu/}) $$$$ \xrightarrow{\colim}\co\Alg_{\bE_\infty}(\mC)_{\tu/}$$
and by Remark \ref{contin} the composition $$ \rho: \Alg_{\Lie(1)}(\mC) \xrightarrow{\theta} \Fun(\bN, \Alg_{\Lie(1)}(\mC)) \xrightarrow{(\triv \circ_{\Lie(1)} (-))_*} \Fun(\bN, \co\Alg_{\bE_\infty}(\mC)_{\tu/}) $$$$ \xrightarrow{\gamma' }\prod_{\bN}\co\Alg_{\bE_\infty}(\mC)_{\tu/}) \xrightarrow{\oplus} \co\Alg_{\bE_\infty}(\mC)_{\tu/} $$
factors as the forgetful functor
$ \Alg_{\Lie(1)}(\mC) \to \mC $ followed by the composition $$\Sym: \mC \xrightarrow{\coComm_\mC^\nun \circ (-)} \co\Alg^{\mathrm{dp}, \conil}_{\coComm_\mC^\nun}(\mC)\to \co\Alg_{\coComm_\mC^\nun}(\mC) \simeq \co\Alg_{\bE_\infty}(\mC)_{\tu/}$$
of the cofree functor and the forgetful functor.

By functoriality of Day-convolution with respect to left kan extensions
the functors $$\colim : \Fun(\bN,\co\Alg_{\bE_\infty}(\mC)_{\tu/}) \to\co\Alg_{\bE_\infty}(\mC)_{\tu/},$$
$$\oplus: \prod_{\bN}\co\Alg_{\bE_\infty}(\mC)_{\tu/} \to\co\Alg_{\bE_\infty}(\mC)_{\tu/} $$ 
are symmetric monoidal with respect to Day-convolution,
where $\co\Alg_{\bE_\infty}(\mC)_{\tu/}$ carries the cartesian structure
and $ \Fun(\bN,\co\Alg_{\bE_\infty}(\mC)_{\tu/}),  \prod_{\bN}\co\Alg_{\bE_\infty}(\mC)_{\tu/}$ carry the induced Day-convolution symmetric monoidal structures.
Hence it is enough to see that the functor
$$ \Alg_{\Lie(1)}(\mC) \xrightarrow{\theta} \Fun(\bN, \Alg_{\Lie(1)}(\mC)) \xrightarrow{(\triv \circ_{\Lie(1)} (-))_*} \Fun(\bN, \co\Alg_{\bE_\infty}(\mC)_{\tu/}) $$ is symmetric monoidal when $ \Alg_{\Lie(1)}(\mC)$ carries the cartesian structure and $$ \Fun(\bN,\co\Alg_{\bE_\infty}(\mC)_{\tu/}) $$ the Day-convolution.
By \cite[Theorem 1.13.]{GWILLIAM20183621} the functor $\gamma$ forming the associated graded, and so its restriction $\gamma'$ are symmetric monoidal with respect to Day-convolution.
Because $\oplus \circ \gamma'$ is conservative, it is enough to see that the
functor $\rho$ preserves finite products.
Since the forgetful functor $ \Alg_{\Lie(1)}(\mC) \to \mC$ preserves finite products, it suffices to see that $ \Sym: \mC \to\co\Alg_{\bE_\infty}(\mC)_{\tu/}$ preserves finite products.
By \cite[Proposition 3.2.4.7.]{lurie.higheralgebra} the object-wise symmetric monoidal structure on $\co\Alg_{\bE_\infty}(\mC)_{\tu/}$ is cartesian.

So we need to see that for any $\X, \Y \in \mC$ the canonical morphism
$$ \Sym(\X \oplus \Y) \to \Sym(\X \oplus \Y) \ot \Sym(\X \oplus \Y) \to \Sym(\X) \otimes \Sym(\Y)$$ is an equivalence, where the first morphism is the diagonal, and the second map is induced by the projections.
The latter map identifies with the canonical equivalence $ \Sym(\X \oplus \Y) \simeq \Sym(\X) \otimes \Sym(\Y)$ in $\mC.$

\end{proof}


Next we construct the enveloping Hopf algebra (Definition \ref{Eenv}).
For that we use the notions of monoid object and $\bE_\infty$-monoid object \cite[Definition 2.4.2.1., Definition 4.1.2.5.]{lurie.higheralgebra} as well as the notion of
group object in any $\infty$-category $\mD$ that admits finite products. A group object in $\mD$ is a monoid object having the property that the shearing morphism
$(\mu, \pr_1): \X \times \X \to \X \times \X$ is an equivalence, where
$\mu$ is the multiplication of $\X$ and $\pr_1$ is projection to the first factor.

\begin{notation}
Let $\mD$ be an $\infty$-category that admits finite products.
\begin{itemize}
\item Let $ \Mon(\mD) \subset \Fun(\Delta^\op,\mD)$
be the full subcategory spanned by the monoid objects in $\mD$.

\item Let $\Mon_{\bE_\infty}(\mD) \subset \Fun(\Fin_*,\mD)$
be the full subcategory spanned by the $\bE_\infty$-monoid objects in $\mD$.

\item Let $ \Grp(\mD) \subset \Mon(\mD)$
be the full subcategory spanned by the groups objects in $\mD.$

\item Let $ \Grp_{\bE_\infty}(\mD) \subset \Mon_{\bE_\infty}(\mD) $
be the full subcategory spanned by the $\bE_\infty$-monoid objects in $\mD$
whose underlying monoid object in $\mD$ is a group object.	
\end{itemize}

\end{notation}

\begin{remark}\label{remqp}
	
Let $\mD$ be an additive $\infty$-category. By definition of additivity $ \Grp(\mD)= \Mon(\mD)$ and so $\Grp_{\bE_\infty}(\mD)=\Mon_{\bE_\infty}(\mD).$
By \cite[Proposition 2.4.2.5., Proposition 2.4.3.9.]{lurie.higheralgebra} the forgetful functors $\Grp_{\bE_\infty}(\mD) =\Mon_{\bE_\infty}(\mD) \to \mD, \Grp(\mD)=\Mon(\mD)\to \mD$ are equivalences.

Let $\phi: \mC \to \mD$ be a finite products preserving functor, where $\mD$ is additive and $\mC$ admits finite products.
Then $ \Grp(\mC)= \Mon(\mC)$ because $\phi$ is conservative
and a monoid $\X$ of $\mC$ is a group if the shearing morphism
$(\mu, \pr_1): \X \times \X \to \X \times \X$ is an equivalence.

\end{remark}

\begin{lemma}\label{hfbbcll}
	
Let $\mC $ be a stable $\infty$-category and $\mD$ an $\infty$-category that admits geometric realizations, finite products and totalizations.
Let $ \phi: \mD \to \mC$ be a conservative functor that preserves geometric realizations, finite products and totalizations. 

\vspace{1mm}

The adjunction $ \Barc: \Grp(\mD) \rightleftarrows \mD_\ast: \Cobar $ is an equivalence.

\end{lemma}

\begin{proof}
	
As $\mC$ is additive, by Remark \ref{remqp} the forgetful functors
$\Mon(\mC)\to \mC, \Mon_{\bE_\infty}(\mC)\to \mC$ are equivalences.
We first assume that $\mC$ admits geometric realizations and totalizations.
In this case by \cite[Corollary 3.4.1.7.]{lurie.higheralgebra} there is a canonical equivalence
$\Alg_{\bE_\infty}(\Mod_\A(\mC)) \simeq \Mon_{\bE_\infty}(\mC)_{\A /} \simeq \mC_{\A /}.$	
So for any $\A \in \mC \simeq \Mon_{\bE_\infty}(\mC)$ the relative tensor product
$ \Barc(\A) = 0 \ot_\A 0 $ is the coproduct of $\A \to 0$ with itself
in the $\infty$-category $\Alg_{\bE_\infty}(\Mod_\A(\mC)) \simeq \Mon_{\bE_\infty}(\mC)_{\A /} \simeq \mC_{\A /}$ and so identifies with $\Sigma(\A).$
So the adjunction $ \Barc: \Mon(\mC) \simeq \mC \rightleftarrows \mC_\ast \simeq \mC: \Cobar$ identifies with the equivalence $\Sigma: \mC \simeq \mC: \Omega.$

For general $\mC$ we embed $\mC$ exact into the stable presentable
$\infty$-category $\Ind(\mC)$ and find that
for every $\A \in \Mon(\mC) \simeq \mC$ the simplicial object
$\Barc(0,\A,0) : \Delta^\op \to \mC, [\n]\mapsto \A^{\oplus \n}$
admits a colimit that is given by $\Sigma(\A)$
and for every $\B \in \mC_* \simeq \mC$ the cosimplicial object
$\Cobar(0,\B,0) : \Delta \to \mC, [\n]\mapsto \B^{\oplus \n}$
admits a limit that is given by $\Omega(\B)$.
Because $\mD$ admits geometric realizations and totalizations,
there is an adjunction $ \Barc: \Mon(\mD) \rightleftarrows \mD_\ast: \Cobar.$
Since $ \phi: \mD \to \mC$ preserves geometric realizations, finite products and totalizations, $\phi$ sends the unit and counit of the adjunction 
$ \Barc: \Mon(\mD) \rightleftarrows \mD_\ast: \Cobar $ to the unit and counit, respectively, of the adjunction $ \Sigma: \mC \rightleftarrows \mC: \Omega. $
As $\phi$ is conservative, both adjoints of the adjunction $ \Barc: \Mon(\mD) \rightleftarrows \mD_\ast: \Cobar$ are equivalences.

\end{proof}

\begin{corollary}\label{qqqp}
	
Let $\mC $ be a stable $\infty$-category that admits geometric realizations and totalizations and $\mO$ an $\infty$-operad in $\mC$ such that $\mO_0=0.$
The adjunction $$ \Barc: \Grp(\Alg_\mO(\mC)) \rightleftarrows \Alg_\mO(\mC): \Cobar $$ is an equivalence.
	
\end{corollary}

\begin{corollary}\label{qqqpp}
	
Let $\mC $ be a stable $\infty$-category that admits geometric realizations and totalizations and $\mO$ an $\infty$-operad in $\mC$ such that $\mO_0=0.$
The adjunction $$\Omega \circ \Barc: \Grp(\Alg_{\mO(1)}(\mC)) \rightleftarrows \Alg_{\mO(1)}(\mC) \simeq \Alg_\mO(\mC): \Cobar \circ \Sigma $$ is an equivalence over $\mC$ via the forgetful functors.
	
\end{corollary}

\begin{notation}Let $\mC$ be a symmetric monoidal $\infty$-category.
Let $$ \Bialg(\mC) = \Mon(\co\Alg_{\bE_\infty}(\mC)),$$$$
\Hopf(\mC) = \Grp(\co\Alg_{\bE_\infty}(\mC)).$$
	
\end{notation}

Remark \ref{remqp} gives the following:

\begin{corollary}
	
Let $\mC$ be an additive presentably symmetric monoidal $\infty$-category.
Then $$\Hopf(\mC) = \Bialg(\mC).$$
	
\end{corollary}

\begin{construction}	

Let $\mC$ be a stable presentably symmetric monoidal $\infty$-category.

By Proposition \ref{dfghcfghkk} there is an adjunction $$ \triv \circ_{\Lie(1)} (-) : \Alg_{\Lie(1)}(\mC) \rightleftarrows \co\Alg_{\coComm_\mC^\nun}(\mC): \triv *^{\coComm_\mC^\nun} (-).$$


By Corollary \ref{fghjok} and Corollary \ref{comparat} there is a canonical equivalence
$$\co\Alg_{\coComm_\mC^\nun}(\mC) \simeq \co\Alg_{\bE_\infty}(\mC)_{\tu/},\ \X \mapsto\X \oplus \tu.$$

\vspace{2mm}

By Proposition \ref{caaaaart} the latter adjunction followed by the equivalence
$$\co\Alg_{\coComm_\mC^\nun}(\mC) \simeq \co\Alg_{\bE_\infty}(\mC)_{\tu/},\ \X \mapsto\X \oplus \tu$$
preserves finite products and so gives rise to an adjunction on monoid objects
\begin{equation}\label{oi1}
\Grp(\Alg_{\Lie(1)}(\mC))= \Mon(\Alg_{\Lie(1)}(\mC)) \rightleftarrows \Bialg(\mC) \simeq \Mon(\co\Alg_{\bE_\infty}(\mC)_{\tu/}),\end{equation} 
where the left adjoint lands in $\Hopf(\mC) \simeq \Grp(\co\Alg_{\bE_\infty}(\mC)_{\tu/}).$

By Lemma \ref{hfbbcll} there is a canonical equivalence
\begin{equation}\label{oi2}
\Cobar: \Alg_{\Lie(1)}(\mC) \simeq \Grp(\Alg_{\Lie(1)}(\mC))= \Mon(\Alg_{\Lie(1)}(\mC)): \Barc\end{equation} 
covering $\Sigma: \mC \to \mC$.
Composing adjunctions (\ref{oi2}) and (\ref{oi1}) gives an adjunction 
$$\mU': \Alg_{\Lie(1)}(\mC) \simeq \Grp(\Alg_{\Lie(1)}(\mC)) \rightleftarrows  \Bialg(\mC):\Prim'. $$

\end{construction}

\begin{remark}\label{rema}
The left adjoint of the adjunction 
$$\mU': \Alg_{\Lie(1)}(\mC) \simeq \Grp(\Alg_{\Lie(1)}(\mC)) \rightleftarrows \Bialg(\mC):\Prim'$$ 
lifts the functor
$$ \Alg_{\Lie(1)}(\mC) \xrightarrow{\Omega} \Alg_{\Lie(1)}(\mC) \xrightarrow{\triv \circ_{\Lie(1)} (-)} \mC \xrightarrow{(-)\oplus \tu}\mC $$
and the right adjoint lifts the functor
$$\Bialg(\mC) \to \co\Alg_{\bE_\infty}(\mC)_{\tu/} \simeq \co\Alg_{\coComm_\mC^\nun}(\mC)  \xrightarrow{\triv *^{\coComm_\mC^\nun} (-)} \mC \xrightarrow{\Sigma} \mC.$$
	
\end{remark}

\begin{definition}\label{Eenv}Let $\mC$ be a stable presentably symmetric monoidal $\infty$-category.

The adjunction of enveloping bialgebra - primitive elements
$$ \mU: \Alg_{\Lie}(\mC) \rightleftarrows \Hopf(\mC)=\Bialg(\mC): \Prim$$ is the composite of the equivalence $$\Alg_{\Lie}(\mC) \simeq \Alg_{\Lie(1)}(\mC)$$
covering $\Sigma: \mC \to \mC$ and the adjunction $$\mU': \Alg_{\Lie(1)}(\mC) \rightleftarrows \Hopf(\mC):\Prim'.$$ 
\end{definition}


\begin{remark}
	
Since limits in group objects are formed underlying, Proposition \ref{caaaaart}
implies that the functor $ \mU: \Alg_{\Lie}(\mC) \to \Hopf(\mC)$ preserves finite products and so turns the direct sum to the tensor product.

\end{remark}

\begin{remark}\label{tens}
	
Remark \ref{rema} implies that the left adjoint of the adjunction 
$$\mU: \Alg_{\Lie}(\mC) \rightleftarrows \Hopf(\mC):\Prim$$ 
lifts the functor
$$ \Alg_\Lie(\mC) \xrightarrow{\Omega} \Alg_{\Lie}(\mC) \xrightarrow{\triv \circ_{\Lie} (-)} \mC \xrightarrow{\Sigma}\mC \xrightarrow{(-)\oplus \tu}\mC $$
and the right adjoint $\Prim$ lifts the functor
$$\Hopf(\mC) \xrightarrow{\text{forget}} \co\Alg_{\bE_\infty}(\mC)_{\tu/} \simeq \co\Alg_{\coComm_\mC^\nun}(\mC) \xrightarrow{\triv *^{\coComm_\mC^\nun} (-)} \mC.$$

By adjointness the composition $\mC \xrightarrow{\Lie} \Alg_{\Lie}(\mC) \xrightarrow{\mU} \Bialg(\mC)$ factors as
$$ \mC \xrightarrow{\triv_{\coComm_\mC^\nun}} \co\Alg_{\coComm_\mC^\nun}(\mC) \simeq \co\Alg_{\bE_\infty}(\mC)_{\tu/} \xrightarrow{\mathrm{free}} \Hopf(\mC)=\Bialg(\mC)=  \Mon(\co\Alg_{\bE_\infty}(\mC)).$$
The latter functor lifts the functor
$$ \mC \xrightarrow{(-)\oplus \tu} \mC_{\tu/} \xrightarrow{\mathrm{free}} \Alg(\mC),$$ which is the free associative algebra functor.
 
\end{remark} 



Theorem \ref{map} gives the following:

\begin{theorem}\label{map2}
	
Let $\mC $ be a stable presentably symmetric monoidal $\infty$-category
such that all norm maps of objects in $\mC$ with an action of a symmetric group are equivalences.
For every Lie algebra $X \in \mC$ the unit $\X \to \Prim \mU(X)$ identifies with the completion morphism
$$ X \to \lim_{\n \geq 1}\tau_\n(\Lie(1)) \circ_{\Lie(1) } X,$$
where we identify $X$ via Corollary \ref{qqqpp} with a group object in $\Lie(1)$-algebras.

\end{theorem}

\begin{proof}
	
Let $X' \in \Grp(\Alg_{\Lie(1)}(\mC))$ be the group object in $\Lie(1)$-algebras in $\mC$ corresponding to $X$ via the equivalence of Corollary \ref{qqqpp}.

The unit of Koszul duality \begin{equation}\label{ppluz}
X' \to \triv *^{\Lie(1)^\vee}(\triv \circ_{\Lie(1)} X') \end{equation}
is a morphism in $ \Grp(\Alg_{\Lie(1)}(\mC))$ whose corresponding morphism of Lie-algebras in $\mC$ is the unit $ X \to \Prim \mU(X)$ by construction of the adjunction $\mU: \Alg_{\Lie}(\mC) \rightleftarrows \Hopf(\mC):\Prim. $

By Theorem \ref{map} the morphism (\ref{ppluz})
is the completion morphism $$ X' \to \lim_{\n \geq 1}\tau_\n(\Lie(1)) \circ_{\Lie(1) }X'.$$

\end{proof}

The initial $\Comm$-algebra $\rH(\bQ)$ in $\Mod_{\rH(\bQ)}$
(Lemma \ref{lek}) has an underling $\Comm^\nun$-algebra used for the following lemma:

\begin{definition}
	
A stable $\infty$-category $\mC$ is $\bQ$-linear
if for every $\X, \Y \in \mC$ the abelian group $\pi_0\mC(\X,\Y)$ is a $\bQ$-vector space.	
	
\end{definition}

\begin{remark}
	
A stable $\infty$-category $\mC$ is $\bQ$-linear if for any $\X, \Y \in \mC$, $\n > 0$ the map
$$ \n \bullet (-): \pi_0\mC(\X,\Y) \to \pi_0\mC(\X,\Y), f \mapsto X \xrightarrow{ } \X^{\oplus\n} \xrightarrow{f^{\oplus\n}} \Y^{\oplus \n} \to \Y$$ is a bijection.	
	
\end{remark}

The next lemma is well-known to the experts and a proof is sketched in \cite[Proposition 1.7.2.]{MR3701353}.
We offer an alternative proof for the reader's convenience, which we find more conceptual and transparent.

\begin{lemma}\label{trivo}
	
Let $\mC$ be a $\bQ$-linear stable symmetric monoidal $\infty$-category
compatible with small colimits and $\mO$ a non-unital $\infty$-operad in $\mC$
such that $\mO_1= \tu.$
The loops functor $\Omega: \Alg_{\mO}(\mC) \to \Alg_{\mO}(\mC)$
factors as the forgetful functor $\Alg_{\mO}(\mC) \to \mC$ followed by the shifted trivial $\mO$-algebra functor $ \mC \xrightarrow{\Omega}\mC \xrightarrow{\triv_\mO } \Alg_{\mO}(\mC).$	
	
\end{lemma}

To prove Lemma \ref{trivo} we use the following lemma:

\begin{lemma}\label{lemon}

The $\Comm^\nun$-algebra $\rH(\bQ)[-1] $ in $\Mod_{\rH(\bQ)}$ is trivial.		
\end{lemma}

\begin{proof}

Let $\D$ be the rational chain complex concentrated in degree $0$ and $-1$ whose
non-trivial differential is the identity of $\bQ.$
There is map $\kappa: \D \to \bQ$ of rational chain complexes that is in degree zero the identity.
Consider the pullback square of non-unital $\bE_\infty$-algebras in rational chain complexes 
\begin{equation}\label{uia}
\begin{xy}
\xymatrix{
\P  \ar[d] 
\ar[rr]
&& \Sym^\nun(\D) \ar[d]^\phi
\\
0 \ar[rr]  && \bQ,
}
\end{xy} 
\end{equation}
where $\bQ$ carries the non-trivial commutative $\bQ$-algebra structure
underlying the unique commutative $\bQ$-algebra structure,
and the right vertical map is induced by $\kappa$ via the free non-unital $\bE_\infty$-algebra $\Sym^\nun.$
The map $\phi$ is degree-wise surjective since $\kappa$ is degree-wise surjective. Therefore $\phi$ is a fibration of non-unital $\bE_\infty$-algebras in rational chain complexes so that square (\ref{uia}) is a homotopy pullback square.
Since $\D$ is acyclic and $\Sym^\nun$ preserves quasi-isomorphisms, $\Sym^\nun(\D)$ is acyclic.
Hence $\P $ is the loops of $\bQ$ in non-unital $\bE_\infty$-algebras.

The unit $ \D \to \Sym^\nun(\D)$ is an quasi-isomorphism since source and target are acyclic.
Therefore the induced map $\lambda: \bQ[-1] = \fib(\kappa) \to \P=\fib(\phi)$ on kernels is a quasi-isomorphism.
Since $\P$ is a $\bE_\infty$-algebra in rational chain complexes,
the square of an element of $\P$ of degree -1 is zero.
Thus the image of the generator under $\lambda$ squares to zero,
so that $\lambda$ is a map of non-unital $\bE_\infty$-algebras.

Thereore $\bQ[-1]$ endowed with the trivial algebra structure is loops of $\bQ$ in non-unital $\bE_\infty$-algebras.
The canonical symmetric monoidal functor
$$\gamma: \Alg_{\Comm^\nun}(\Ch_{\bQ}) \to \Alg_{\Comm^\nun}(\Mod_{\rH(\bQ)})$$
sends square (\ref{uia}) to a pullback square since the forgetful functor $\Alg_{\Comm^\nun}(\Mod_{\rH(\bQ)}) \to \Mod_{\rH(\bQ)}$ detects loops
and the functor $\Ch_{\bQ} \to \Mod_{\rH(\bQ)}$ sends homotopy pullback squares to pullback squares.
Hence $\gamma$ preserves trivial algebras, $\rH(\bQ)[-1]$ carries the trivial $\Comm^\nun$-algebra structure.	

\end{proof}

\begin{notation}
	
Let $\mC$ be a symmetric monoidal $\infty$-category compatible with small colimits and $\mO, \mO'$ symmetric sequences in $\mC.$
We write $\mO \boxtimes \mO'$ for the symmetric sequence 
$$\Sigma \xrightarrow{(\mO, \mO')} \mC \times \mC \xrightarrow{\ot}\mC.$$
\end{notation}

The next lemma is \cite[Proposition 3.9.]{https://doi.org/10.48550/arxiv.2104.03870}:

\begin{lemma}

Let $\mC$ be a symmetric monoidal $\infty$-category
compatible with small colimits.
The functor $$\boxtimes: \sSeq(\mC) \times \sSeq(\mC) \to \sSeq(\mC)$$ is lax monoidal with respect to composition product.	
	
\end{lemma}

\begin{corollary}\label{cop}
	
Let $\mC$ be a symmetric monoidal $\infty$-category
compatible with small colimits and let $\mO, \mO'$ be $\infty$-operads in $\mC.$
There is a canonical functor $$\Alg_{\mO}(\mC) \times \Alg_{\mO'}(\mC) \to \Alg_{\mO \boxtimes \mO'}(\mC)$$ covering the functor $\ot: \mC \times \mC \to \mC$ and natural in $\mO, \mO'.$

\end{corollary}



\begin{proof}[Proof of Lemma \ref{trivo}]

By Corollary \ref{cop} there is a canonical functor $$\theta: \Alg_{\Comm^\nun}(\Mod_{\rH(\bQ)}) \times \Alg_{\mO}(\mC) \to \Alg_{\Comm^\nun \boxtimes \mO}(\mC) \simeq \Alg_{\mO}(\mC)$$
covering the functor $\ot: \mC \times \mC \to \mC.$
The induced functor 
\begin{equation*}\label{eqxk}
\{\rH(\bQ)[-1] \} \times \Alg_{\mO}(\mC) \to \Alg_{\Comm^\nun}(\Mod_{\rH(\bQ)}) \times \Alg_{\mO}(\mC) \xrightarrow{\theta} \Alg_{\mO}(\mC)\end{equation*}
identifies with the functor 
$$ \{\rH(\bQ) \} \times \Alg_{\mO}(\mC) \to \Alg_{\Comm^\nun}(\Mod_{\rH(\bQ)}) \times \Alg_{\mO}(\mC) \xrightarrow{\theta} \Alg_{\mO}(\mC) \xrightarrow{ \Omega }\Alg_{\mO}(\mC),$$
which is the loops functor $\Omega: \Alg_{\mO}(\mC) \to \Alg_{\mO}(\mC)$.

By Lemma \ref{lemon} the $\Comm^\nun$-algebra $\rH(\bQ)[-1] $ is trivial.
Hence the claim follows from the existence of the following commutative square
provided by Corollary \ref{cop}:
$$
\begin{xy}
\xymatrix{
\Alg_{\triv}(\Mod_{\rH(\bQ)}) \times \Alg_{\mO}(\mC)  \ar[d] 
\ar[rr]
&& \Alg_{\triv \boxtimes \mO}(\mC) \simeq \Alg_{\triv}(\mC) \ar[d]
\\
\Alg_{\Comm^\nun}(\Mod_{\rH(\bQ)}) \times \Alg_{\mO}(\mC) \ar[rr]  && \Alg_{\Comm^\nun \boxtimes \mO}(\mC) \simeq  \Alg_{\mO}(\mC).
}
\end{xy} 
$$ 
\end{proof}

We obtain the following corollary, which is \cite[Proposition 1.7.2.]{MR3701353}:

\begin{corollary}\label{trivosa}
Let $\mC$ be a $\bQ$-linear stable symmetric monoidal $\infty$-category
compatible with small colimits. Every group object in $\Alg_{\Lie(1)}(\mC)$ is trivial as $\Lie(1)$-algebra.
	
\end{corollary}

\begin{proof}
	
By Corollary \ref{qqqp} the functor $ \Cobar: \Alg_{\Lie(1)}(\mC) \to \Grp(\Alg_{\Lie(1)}(\mC)) $ is an equivalence.
Since the functor $\Cobar$ lifts the loops functor
$\Omega: \Alg_{\Lie(1)}(\mC) \to \Alg_{\Lie(1)}(\mC) $ along the forgetful functor,
the claim follows from Corollary \ref{trivo}.

\end{proof}

We are grateful to Gijs Heuts for communicating to us the following lemma, which also appears
as \cite[Lemma 13.8., 13.10.]{Heuts}.

\begin{lemma}\label{calc}

Let $\mC$ be a $\mathbb{Q}$-linear stable presentably symmetric monoidal $\infty$-category and $\mQ$ a non-counital $\infty$-cooperad in $\bQ$-vector spaces  such that $\mQ_1 = \tu$ and $\mQ^\vee(-1)$ is an $\infty$-operad in $\bQ$-vector spaces. Set $\mO:=\mQ^\vee.$ For every trivial $\mO$-algebra $\Y$ the morphism $$\alpha: \Y \to \lim_{\n \geq 1}\tau_\n(\mO) \circ_\mO \Y$$ is an equivalence.

\end{lemma}

\begin{proof}

We like to see that for every $\X \in \mC$ the morphism
$$\alpha: \X \to \lim_{\n \geq 1}\tau_\n(\mO) \circ_\mO \triv_\mO(\X)$$ in $\mC$ is an equivalence.
By stability it is enough to see that the cofiber of $\alpha$ vanishes.
The latter is the limit of the inverse system $...\to \cofib(\alpha_{\n+1}) \to \cofib(\alpha_\n) \to ... \to \cofib(\alpha_1)$ of cofibers of the maps $$\alpha_\n: \X \to \tau_\n(\mO) \circ_\mO \triv_\mO(\X)
\simeq (\tau_\n(\mO) \circ_\mO \triv) \circ \X.$$
Consequently, it is enough to see that for every $\n \geq 1$
the induced morphism $ \cofib(\alpha_{\n+1}) \to \cofib(\alpha_\n)$ 
is the zero morphism. 

Observe that $(\tau_\n(\mO) \circ_\mO \triv)_1 \simeq \tu$ and $\alpha_\n$ is the canonical morphism $$\X \simeq (\tau_\n(\mO) \circ_\mO \triv)_1 \ot \X \to
\coprod_{\bk \geq 1} (\tau_\n(\mO) \circ_\mO \triv)_\bk \ot_{\Sigma_\bk} \X^{\ot \bk}.$$ 

Hence the cofiber of $\alpha_\n$ is equivalent to $\prod_{\bk \geq 2} (\tau_\n(\mO) \circ_\mO \triv)_\bk \ot_{\Sigma_\bk} \X^{\ot \bk} $ and the canonical morphism
$\cofib(\alpha_\n) \to \cofib(\alpha_{\n-1})$ identifies with the morphism
$$\coprod_{\bk \geq 2} (\tau_\n(\mO) \circ_\mO \triv)_\bk \ot_{\Sigma_\bk} \X^{\ot \bk} \to \coprod_{\bk \geq 2} (\mO_{\leq \n-1} \circ_\mO \triv)_\bk \ot_{\Sigma_\bk} \X^{\ot \bk} $$
induced by the morphisms $ (\tau_\n(\mO) \circ_\mO \triv)_\bk \to (\mO_{\leq \n-1} \circ_\mO \triv)_\bk$ for $\bk \geq 2.$

We will show the following:

\begin{enumerate}

\item For every $\bk > 1 $ the rational homology of $(\tau_\n(\mO) \circ_\mO \triv)_\bk$ is concentrated in degrees $\geq 1-\n $.
\item For every $\bk > 1 $ the rational homology of $(\tau_\n(\mO) \circ_\mO \triv)_\bk$ is concentrated in degrees $\leq 1-\n $.

\end{enumerate}

1: By assumption the rational homology of $(\tau_1(\mO) \circ_\mO \triv)_\bk \simeq (\triv \circ_\mO \triv)_\bk \simeq \mQ_\bk $ is concentrated in degree zero. Therefore 1. follows by induction over $\n \geq 1$ from the fact that the rational homology of the fiber $\F$ of the map $ (\tau_\n(\mO) \circ_\mO \triv)_\bk \to (\mO_{\leq \n-1} \circ_\mO \triv)_\bk$ is concentrated in degree $1-\n: $ the fiber $\F$ is the $\bk$-th term of the cofree right
$\mQ \simeq \mO^\vee$-comodule on the symmetric sequence concentrated in degree $\n$ with value $\mO_\n.$
We set $\mO':= \mO(-1)= \mQ^\vee(-1)$ so that $\mO'$ is an $\infty$-operad in $\bQ$-vector spaces.
Then we have that $$\F \simeq (\mO'_\n[1-\n] \ot \mQ^{\ot \n})_\bk \simeq \coprod_{\bk_1+...+\bk_\n= \bk}  \mO'_\n[1-\n] \ot \Sigma_\bk \times_{( \Sigma_{\bk_1} \times ... \times \Sigma_{\bk_\n})} (\mQ_{\bk_1} \ot ... \ot \mQ_{\bk_\n}) $$ has rational homology concentrated in degree $1-\n.$

\vspace{1mm}
2: Let $\tau_{> \n}(\mO) $ be the fiber of the canonical map $ \mO \to \tau_\n(\mO) $ of $\infty$-operads considered as a map of right $\mO$-modules in $ \sSeq(\Sp).$
Applying the exact functor $ (-) \circ_{\mO} \triv : \RMod_{\mO}(\sSeq(\Sp)) \to  \sSeq(\Sp)$
we get a fiber sequence
$\W \to \triv \to \tau_\n(\mO) \circ_\mO \triv $ and so for every $\bk > 1 $ a fiber sequence
$\W_\bk \to 0 \to (\tau_\n(\mO) \circ_\mO \triv)_\bk $.
Hence $\W_{\bk} \simeq \Omega((\tau_\n(\mO) \circ_\mO \triv)_\bk) $.
Consequently, 2. is equivalent to the condition that for every $\bk > 1 $ the rational homology of $\W_\bk= (\tau_{> \n}(\mO) \circ_{\mO} \triv)_{\bk}$ is concentrated in degrees $\leq -\n $.

By Remark \ref{dghfgh} and Remark \ref{zuk} we find that $\W_\bk$ is the colimit of a filtered diagram 
$$ \tau_{> \n}(\mO) \simeq \D_0 \to ... \to \D_\ell \to ... $$ 
such that the cofiber $\C_\ell$ of the morphism $\D_{\ell-1} \to \D_\ell$
is equivalent to the $\ell$-th shift of 
$$\colim_{\f \in {(\Fin^{\ell+2}_\rd)}_\mathrm{ndeg} }(\tau_{> \n}(\mO)_{\I_1} \ot \bigotimes_{\bi \in \I_1} \tau_{> \n}(\mO)_{ \f_1^{-1}(\bi) } \ot \bigotimes_{\bi \in \I_{2}} \mO_{ \f_{2}^{-1}(\bi) } \ot ... \ot \bigotimes_{\bi \in \I_{\ell}} \mO_{ \f_{\ell}^{-1}(\bi)} \ot \bigotimes_{\bi \in \I_{\ell+1}} \triv_{ \f_{\ell+1}^{-1}(\bi)}),$$
where ${(\Fin^{\ell+2}_\rd)}_\mathrm{ndeg} \subset \Fin^{\ell+2}_\rd $
is the full subcategory of sequences of maps of finite sets $\f:  \rd \xrightarrow{\f_{\ell+1}} \I_{\ell+1}  \xrightarrow{\f_{\ell}} ... \xrightarrow{\f_1} \I_1$ of length $\ell+1$ such that no map in the sequence is a bijection.

So $\C_\ell$ is equivalent to the $\ell$-th shift of 
$$\colim_{ \f \in {(\Fin^{\ell+2}_\rd)}_\mathrm{ndeg}'}(\Lie_{\I_1}[1-|\I_1|] \ot \bigotimes_{\bi \in \I_1} \Lie_{ \f_1^{-1}(\bi) }[1- | \f_1^{-1}(\bi) | ] \ot  ... $$$$ \ot \bigotimes_{\bi \in \I_{\ell}} \Lie_{ \f_\ell^{-1}(\bi) }[1- | \f_\ell^{-1}(\bi) |] \ot \bigotimes_{\bi \in \I_{\ell+1}} \triv_{ \f_{\ell+1}^{-1}(\bi) } ), $$
where ${(\Fin^{\ell+2}_\rd)}_\mathrm{ndeg}' \subset \Fin^{\ell+2}_\rd $
is the full subcategory of sequences of maps of finite sets $\f:  \rd \xrightarrow{\f_{\ell+1}} \I_{\ell+1}  \xrightarrow{\f_{\ell}} ... \xrightarrow{\f_1} \I_1$ of length $\ell+1$ such that all maps in the sequence are surjections, no map in the sequence is a bijection and the cardinality $| \I_1 |$ of $\I_1$ is larger than $\n.$

For every $\f \in {(\Fin^{\ell+2}_\rd)}_\mathrm{ndeg}'$ and $ 1 \leq \bj \leq \ell$
the object $\bigotimes_{\bi \in \I_{\bj}} \Lie_{ \f_\bj^{-1}(\bi) }[1- | \f_\bj^{-1}(\bi) | ] $ has rational homology concentrated in negative degrees.
Therefore $$ \Lie_{\I_1}[1-|\I_1|] \ot \bigotimes_{\bi \in \I_1} \Lie_{ \f_1^{-1}(\bi) }[1- | \f_1^{-1}(\bi) | ] \ot ... \ot \bigotimes_{\bi \in \I_{\ell}} \Lie_{ \f_\ell^{-1}(\bi) }[1- | \f_\ell^{-1}(\bi) | ] \ot \bigotimes_{\bi \in \I_{\ell+1}} \triv_{ \f_{\ell+1}^{-1}(\bi) }$$
has rational homology concentrated in degrees $\leq -\n - \ell. $ 
Thus the cofiber $\C_\ell$ has rational homology concentrated in degrees $\leq -\n. $ 

Hence by induction on $\ell$ the rational homology of $\D_\ell$ is concentrated in degrees $\leq -\n, $ where the case $\ell=1$ follows from the fact that
$\mD_0 \simeq \tau_{> \n}(\mO) $ has rational homology concentrated in degrees $\leq -\n. $ 
So $(\tau_{> \n}(\mO) \circ_{\mO} \triv)_{\bk}  \simeq \colim_{\ell \geq 1 } \D_\ell $ has rational homology concentrated in degrees $\leq -\n $.

\end{proof}

\begin{remark}\label{zuk}
Let $\mC$ be a stable symmetric monoidal $\infty$-category compatible with small colimits.
Let $\mO$ be a non-unital $\infty$-operad in $\mC$ such that $\mO_1= \tu$.
For every left $\mO$-module $\X$ and right $\mO$-module $\Y$ in $\sSeq(\mC)_{\geq 1}$ 
let $\mathrm{Bar}(\X,\mO,\Y): \Delta^\op \to \mC$ be the bar construction. 

\vspace{1mm}
For every $[\n]\in \Delta$ let $ \Delta_{\leq \n} \subset \Delta$ be the full subcategory of objects $[\br]$ with $\br \leq \n$
and $\mathrm{Bar}(\X,\mO,\Y)^\n:= \mathrm{Bar}(\X,\mO,\Y)_{|\Delta_{\leq \n}^\op} $ the restriction.
There is a filtered diagram
$$\X \circ \Y \simeq \colim(\mathrm{Bar}(\X,\mO, \Y)^0) \to ... \to \colim(\mathrm{Bar}(\X,\mO, \Y)^\n) \xrightarrow{\alpha_\n}$$$$ \colim(\mathrm{Bar}(\X,\mO, \Y)^{\n+1}) \to... \to \colim(\mathrm{Bar}(\X,\mO, \Y)) \simeq \X \circ_\mO \Y,$$
whose colimit is $\X \circ_\mO \Y$. 

\vspace{1mm}
Let $\L_\n$ be the $\n$-th latching object of $\mathrm{Bar}(\X,\mO, \Y) $
defined as the colimit of the restriction of the functor 
$ (\Delta^\op)_{/[\n]} \to \Delta^\op \xrightarrow{\mathrm{Bar}(\X,\mO, \Y) } \mC $ to the full subcategory of surjective maps $ [\n] \to [\bk] $ in $\Delta$ with $\bk \neq \n.$
By \cite[Remark 1.2.4.3.]{lurie.higheralgebra} there is a canonical equivalence $$\X \circ \mO^{\circ \n} \circ \Y \simeq  \mathrm{Bar}(\X,\mO, \Y)_\n \simeq \L_\n \oplus \cofib(\alpha^{\n-1} )[-\n].$$ 

For any $\rd \in \Sigma$ let ${(\Fin^{\n+2}_\rd)}_\mathrm{deg}, {(\Fin^{\n+2}_\rd)}_\mathrm{ndeg} \subset {(\Fin^{\n+2}_\rd)} $
be the full subcategory of sequences of maps of finite sets 
$\f: \rd \xrightarrow{\f_{\n+1}} \I_{\n+1}  \xrightarrow{\f_{\n}} ... \xrightarrow{\f_1} \I_1$ of length $\n+1$ such that at least one of the maps is a bijection, no map in the sequence is a bijection, respectively.

Under the equivalence $$(\X \circ \mO^{\circ \n} \circ \Y)_\rd \simeq \colim_{ \f \in \Fin^{\n+2}_\rd }(  \X_{\I_1} \ot \bigotimes_{\bi \in \I_1} \X_{ \f_1^{-1}(\bi) }  \ot ... \ot \bigotimes_{\bi \in \I_{\n}} \mO_{ \f_{\n}^{-1}(\bi) } \ot \bigotimes_{\bi \in \I_{\n+1}} \Y_{ \f_{\n+1}^{-1}(\bi) } ) $$
of Remark \ref{dghfgh} the summand $(\L_\n)_\rd$ corresponds to the summand
$$\colim_{ \f \in {(\Fin^{\n+2}_\rd)}_\mathrm{deg} }(  \X_{\I_1} \ot \bigotimes_{\bi \in \I_1} \X_{ \f_1^{-1}(\bi) }  \ot ... \ot \bigotimes_{\bi \in \I_{\n}} \mO_{ \f_{\n}^{-1}(\bi) } \ot \bigotimes_{\bi \in \I_{\n+1}} \Y_{ \f_{\n+1}^{-1}(\bi) } ),$$
while the summand $ \cofib( \alpha^{\n-1}_\rd )[-\n]$ 
corresponds to the summand
$$\colim_{ \f \in {(\Fin^{\n+2}_\rd)}_\mathrm{ndeg} }(\X_{\I_1} \ot \bigotimes_{\bi \in \I_1} \X_{ \f_1^{-1}(\bi) }  \ot ... \ot \bigotimes_{\bi \in \I_{\n}} \mO_{ \f_{\n}^{-1}(\bi) } \ot \bigotimes_{\bi \in \I_{\n+1}} \Y_{ \f_{\n+1}^{-1}(\bi) }).$$
\end{remark}

\vspace{2mm}

Now we are ready to prove a derived version of the Milnor-Moore theorem.

\begin{theorem}\label{thg}

Let $\mC $ be a $\bQ$-linear stable presentably symmetric monoidal $\infty$-category. The enveloping Hopf algebra functor $$\mU: \Alg_{\Lie}(\mC) \to \Hopf(\mC)$$ is fully faithful.

\end{theorem}

\begin{proof}


The statement follows from Theorem \ref{map2}, Corollary \ref{trivosa} and Lemma \ref{calc}.

\end{proof}


\begin{remark}
In \cite[Theorem 4.4.6.]{MR3701353} Gaitsgory-Rozenblyum give another proof of Theorem \ref{thg} via a filtration of the primitive elements.

\end{remark}

Theorem \ref{thg} gives the following:

\begin{corollary}\label{thf} Let $\mC $ be a $\bQ$-linear stable presentably symmetric monoidal $\infty$-category. The functor of primitive elements $\Prim: \Hopf(\mC) \to \Alg_{\Lie}(\mC) $ exhibits $\Alg_{\Lie}(\mC) $ as the $\infty$-category of restricted $L_\infty$-algebras in $\mC.$
	
\end{corollary}

\begin{proof}By Theorem \ref{thg} the unit of the adjunction $$\mU: \Alg_{\Lie}(\mC) \to \Hopf(\mC):\overline{\Prim}$$ is an equivalence. Consequently, the canonical transformation $\mL \to \Prim\circ \T$ is an equivalence, where $\mL$ is the free $\Lie$-algebra. This implies that the functor of primitive elements $\Prim: \Hopf(\mC) \to \Alg_{\Lie}(\mC) $ exhibits $\Alg_{\Lie}(\mC) $ as the $\infty$-category $\LMod_{\Prim\circ \T}(\mC)$ of $\Prim\circ \T$-algebras in $\mC$.	
\end{proof}

In the following we study corollaries of Theorem \ref{thg}.

\begin{lemma}\label{hfbbcll}

Let $\mC $ be a stable $\infty$-category and $\mD$ an $\infty$-category that admits geometric realizations, finite limits and totalizations.
Let $ \phi: \mD \to \mC$ be a conservative functor that preserves geometric realizations, finite products and totalizations. 
The forgetful functor $\Sp(\mD) \to \Grp_{\bE_\infty}(\mD)$ is an equivalence.

\end{lemma}	

\begin{proof}

By Lemma \ref{hfbbcll} the adjunction $ \Barc: \Grp(\mD) \rightleftarrows \mD_\ast: \Cobar $ is an equivalence.	
The composition $\mD_* \xrightarrow{\Cobar} \Grp(\mD) \xrightarrow{\mathrm{\text{forget}}} \mD_*$
is the loops functor $\Omega.$
So the functor $\Grp_{\bE_\infty}(\Omega): \Grp_{\bE_\infty}(\mD) \to \Grp_{\bE_\infty}(\mD)$ factors as the equivalence
$ \Grp_{\bE_\infty}(\Cobar): \Grp_{\bE_\infty}(\mD) \simeq \Grp_{\bE_\infty}(\mD_*)
\to \Grp_{\bE_\infty}(\Grp(\mD)) $ followed by the functor
$\Grp_{\bE_\infty}(\Grp(\mD)) \to \Grp_{\bE_\infty}(\mD).$

By \cite[Proposition 2.8.]{2013arXiv1305.4550G} the $\infty$-category $\Grp_{\bE_\infty}(\mD)$ is additive.
By Remark \ref{remqp} this implies that the forgetful functor $\Grp(\Grp_{\bE_\infty}(\mD)) \to \Grp_{\bE_\infty}(\mD)$ is an equivalence. The latter functor identifies with the forgetful functor $\Grp_{\bE_\infty}(\Grp(\mD)) \to \Grp_{\bE_\infty}(\mD).$
So the functor $\Grp_{\bE_\infty}(\Omega): \Grp_{\bE_\infty}(\mD) \to \Grp_{\bE_\infty}(\mD)$ is an equivalence.
By Definition \cite[1.4.2.]{lurie.higheralgebra} the $\infty$-category $\Sp(\mD)$ is the limit of the tower
$$... \xrightarrow{\Omega} \mD_* \xrightarrow{\Omega} \mD_*$$
in the $\infty$-category of small $\infty$-categories
and so also in the $\infty$-category of small $\infty$-categories having finite products. Applying the small limits preserving endofunctor
$\mD \mapsto \Grp_{\bE_\infty}(\mD)$ of the $\infty$-category of small $\infty$-categories having finite products, we find that
$\Grp_{\bE_\infty}(\Sp(\mD))$ is the limit of the tower
$$... \xrightarrow{\Grp_{\bE_\infty}(\Omega)}\Grp_{\bE_\infty}(\mD) \xrightarrow{\Grp_{\bE_\infty}(\Omega)} \Grp_{\bE_\infty}(\mD),$$
where all transition functors are equivalences.
Consequently, we find that $$\Grp_{\bE_\infty}(\Sp(\mD)) \simeq \Sp(\Grp_{\bE_\infty}(\mD)) \simeq \Grp_{\bE_\infty}(\mD).$$

The $\infty$-category $\Sp(\mD)$ is stable and so additive, which implies that the forgetful functor $\Grp_{\bE_\infty}(\Sp(\mD)) \to \Sp(\mD)$ is an equivalence.

\end{proof}

\begin{proposition}\label{defo}
Let $\mC $ be an additive symmetric monoidal $\infty$-category and $\mO$ a non-unital $\infty$-operad in $\mC$.
The forgetful functor $\Grp_{\bE_\infty}(\Alg_\mO(\mC)) \to \Grp_{\bE_\infty}(\mC) \simeq \mC$ is an equivalence.	

\end{proposition}

\begin{proof}
Without loss of generality we can assume that $\mC$ is small.
By Corollary \ref{cory} there is a large additive symmetric monoidal $\infty$-category $\mD$ compatible with small colimits that has small limits,
and a symmetric monoidal additive embedding $\mC \to \mD$ that preserves small limits. 
By Remark \ref{imp} the $\infty$-category $\Alg_\mO(\mD)$ ha small limits and small sifted colimits, which are preserved by the forgetful functor to $\mD$.
Consider the commutative square
\begin{equation*}
\begin{xy}
\xymatrix{
\Sp(\Alg_\mO(\mD))  \ar[d] 
\ar[r]^{ } 
& \Grp_{\bE_\infty}(\Alg_\mO(\mD)) \ar[d]^\alpha
\\
\Sp(\mD)\simeq \mD  \ar[r]^{\id}  & \Grp_{\bE_\infty}(\mD) \simeq \mD.
}
\end{xy} 
\end{equation*} 

By Lemma \ref{hfbbcll} the top functor is an equivalence.
By \cite[Theorem 7.3.4.13.]{lurie.higheralgebra} the left vertical functor is an equivalence. Hence the right vertical functor $\alpha$ is an equivalence.
The forgetful functor $\Alg_\mO(\mC)\to \mC$ is the pullback of the 
forgetful functor $\Alg_\mO(\mD)\to \mD$.
Hence the forgetful functor $\Grp_{\bE_\infty}(\Alg_\mO(\mC)) \to \mC $ is the pullback of the forgetful functor $\Grp_{\bE_\infty}(\Alg_\mO(\mD)) \to \mD $ and so an equivalence.	

\end{proof}

\begin{remark}\label{remol}

Let $\mC $ be an additive symmetric monoidal $\infty$-category and $\mO$ a non-unital $\infty$-operad in $\mC$. By Proposition \ref{defo} the forgetful functor $ \Grp_{\bE_\infty}(\Alg_{\mO}(\mC)) \to \mC$ is an equivalence.
The trivial $\mO$-algebra functor
$\triv_\mO: \mC \to \Alg_{\mO}(\mC)$ is a section of the forgetful functor
$\Alg_{\mO}(\mC) \to \mC$ and so preserves finite products and so group objects.
Hence the induced functor
$\mC \simeq \Grp_{\bE_\infty}(\mC) \xrightarrow{\Grp_{\bE_\infty}(\triv_\mO)} \Grp_{\bE_\infty}(\Alg_{\mO}(\mC))$ is a section and so an inverse of the forgetful functor
$\Grp_{\bE_\infty}(\Alg_{\mO}(\mC)) \to \Grp_{\bE_\infty}(\mC) \simeq \mC$.

\end{remark}

\vspace{2mm}
Let $\mC $ be a symmetric monoidal $\infty$-category.
By \cite[Proposition 3.2.4.7.]{lurie.higheralgebra} the object-wise symmetric monoidal structure on $\co\Alg_{\bE_\infty}(\mC)$ 
is cartesian.
Therefore the following definition is reasonable:

\begin{notation}
Let $\mC $ be a symmetric monoidal $\infty$-category.
Let $$\Hopf_{\bE_\infty}(\mC):= \Grp_{\bE_\infty}(\co\Alg_{\bE_\infty}(\mC)).$$

\end{notation}

The symmetric monoidal forgetful functor $\co\Alg_{\bE_\infty}(\mC) \to \mC$
gives rise to a forgetful functor $$\Hopf_{\bE_\infty}(\mC) \subset \Mon_{\bE_\infty}(\co\Alg_{\bE_\infty}(\mC)) \simeq \Alg_{\bE_\infty}(\co\Alg_{\bE_\infty}(\mC)) \to \Alg_{\bE_\infty}(\mC).$$

\vspace{1mm}

\begin{corollary}\label{coray}

Let $\mC $ be a $\bQ$-linear stable presentably symmetric monoidal $\infty$-category. The free $\bE_\infty$-algebra functor $$\Sym: \mC \to \Alg_{\bE_\infty}(\mC)$$ lifts to an embedding $\mC \to \Hopf_{\bE_\infty}(\mC)$.

\end{corollary}

\begin{proof}

By Theorem \ref{thg} the enveloping Hopf algebra functor $$\mU: \Alg_{\Lie}(\mC) \simeq \Alg_{\Lie(1)}(\mC) \simeq \Grp(\Alg_{\Lie(1)}(\mC)) \xrightarrow{\Grp(\triv\circ_{\Lie(1)}(-))}  \Grp(\co\Alg_{\bE_\infty}(\mC))$$ is fully faithful and by Lemma \ref{caaaaart} preserves finite products.
We obtain an induced commutative square, where the top functor and so the bottom functor is fully faithful:
\begin{equation*}
\begin{xy}
\xymatrix{
\Grp_{\bE_\infty}(\Grp(\Alg_{\Lie(1)}(\mC))) \ar[d]^\simeq
\ar[rrr]^{\Grp_{\bE_\infty}(\Grp(\triv\circ_{\Lie(1)}(-)))}
&&& \Grp_{\bE_\infty}(\Grp(\co\Alg_{\bE_\infty}(\mC))) \ar[d]^\simeq
\\
\Grp_{\bE_\infty}(\Alg_{\Lie(1)}(\mC)) \ar[rrr]^{\Grp_{\bE_\infty}(\triv\circ_{\Lie(1)}(-))}  &&& \Grp_{\bE_\infty}(\co\Alg_{\bE_\infty}(\mC))= \Hopf_{\bE_\infty}(\mC).
}
\end{xy} 
\end{equation*}

By Proposition \ref{defo} and Remark \ref{remol} the forgetful functor
$ \Grp_{\bE_\infty}(\Alg_{\Lie(1)}(\mC)) \to \mC$ is an equivalence
inverse to the functor $$\Grp_{\bE_\infty}(\triv_{\Lie(1)}):\mC \simeq \Grp_{\bE_\infty}(\mC) \xrightarrow{ } \Grp_{\bE_\infty}(\Alg_{\Lie(1)}(\mC))$$
induced by the trivial $\Lie(1)$-algebra functor $\triv_{\Lie(1)}: \mC \to \Alg_{\Lie(1)}(\mC).$

Let $\phi$ be the composition $$\mC \xrightarrow{\Grp_{\bE_\infty}(\triv_{\Lie(1)})}	\Grp_{\bE_\infty}(\Alg_{\Lie(1)}(\mC)) \xrightarrow{\Grp_{\bE_\infty}(\triv\circ_{\Lie(1)}(-))} \Hopf_{\bE_\infty}(\mC).$$

The composition $\mC \xrightarrow{\phi} \Hopf_{\bE_\infty}(\mC) \xrightarrow{\text{forget}} \co\Alg_{\bE_\infty}(\mC)$
is the functor $$\mC \xrightarrow{\triv_{\Lie(1)}} \Alg_{\Lie(1)}(\mC) \xrightarrow{\triv\circ_{\Lie(1)}(-)} \co\Alg^{\mathrm{dp}, \conil}_{\coComm^\nun_\mC}(\mC) \to \co\Alg_{\bE_\infty}(\mC),$$ which by Remark \ref{contin} factors as $$ \mC \xrightarrow{\coComm^\nun \circ (-)}\co\Alg^{\mathrm{dp}, \conil}_{\coComm^\nun_\mC}(\mC) \to \co\Alg_{\coComm^\nun_\mC}(\mC) \xrightarrow{(-)\oplus \tu} \co\Alg_{\bE_\infty}(\mC).$$
Thus the composition $\phi': \mC \xrightarrow{\phi} \Hopf_{\bE_\infty}(\mC) \xrightarrow{\text{forget}} \Alg_{\bE_\infty}(\mC)$
lifts the functor $$ \mC \to \mC : \X \mapsto \coprod_{\n \geq 0} (\X^{\ot\n})_{\Sigma_\n}$$ along the forgetful functor
$\V: \Alg_{\bE_\infty}(\mC) \to \mC$ whose left adjoint we denote by $\Sym.$
Mapping to the first summand defines a natural transformation
$\id \to \V \circ \phi'$ of endofunctors of $\mC$ that gives rise to a map $\Sym = \Sym \circ \id \to \Sym \circ \V \circ \phi' \to \phi' $
of functors $\mC \to \Alg_{\bE_\infty}(\mC)$ that is component-wise the canonical equivalence.

\end{proof}

If $\mC$ is not $\bQ$-linear, we offer the following variant (Proposition \ref{coray2}) of Corollary \ref{coray}.
Let $\mC $ be a preadditive presentably symmetric monoidal $\infty$-category.
By Remark \ref{preo} the $\infty$-category $\co\Alg^{\mathrm{dp}, \conil}_{\coComm^\nun}(\mC)$ is presentable and thus admits finite products.
So we can make the following definition:

\begin{notation}
Let $\mC $ be a preadditive presentably symmetric monoidal $\infty$-category.
Let $$\Hopf^{\mathrm{dp}, \conil}_{\bE_\infty}(\mC):= \Grp_{\bE_\infty}(\co\Alg^{\mathrm{dp}, \conil}_{\coComm^\nun}(\mC)).$$
	
\end{notation}

The forgetful functor $\co\Alg^{\mathrm{dp}, \conil}_{\coComm^\nun}(\mC) \to \co\Alg_{\coComm^\nun}(\mC) \simeq \co\Alg_{\bE_\infty}(\mC)_{\tu/}$
gives rise to a forgetful functor $$\Hopf^{\mathrm{dp}, \conil}_{\bE_\infty}(\mC) \to \Hopf_{\bE_\infty}(\mC).$$

\begin{proposition}\label{coray2}
	
Let $\mC $ be a stable presentably symmetric monoidal $\infty$-category that admits totalizations. 
The free $\bE_\infty$-algebra functor $$\Sym: \mC \to \Alg_{\bE_\infty}(\mC)$$ lifts to an embedding $\mC \to \Hopf^{\mathrm{dp}, \conil}_{\bE_\infty}(\mC)$.
	
\end{proposition}

\begin{proof}
	
By Proposition \ref{caaaaart} the functor $\triv\circ_{\Lie(1)}(-): \Alg_{\Lie(1)}(\mC) \to \co\Alg^{\mathrm{dp}, \conil}_{\bE_\infty}(\mC)$ preserves finite products and so gives rise to a functor
$$\Grp_{\bE_\infty}(\triv\circ_{\Lie(1)}(-)): \Grp_{\bE_\infty}(\Alg_{\Lie(1)}(\mC)) \to \Grp_{\bE_\infty}(\co\Alg^{\mathrm{dp}, \conil}_{\bE_\infty}(\mC)).$$

By Remark \ref{remol} the trivial $\Lie(1)$-algebra functor
$\triv_{\Lie(1)}: \mC \to \Alg_{\Lie(1)}(\mC)$ lifts to a functor $\Grp_{\bE_\infty}(\triv_{\Lie(1)}): \mC \simeq \Grp_{\bE_\infty}(\mC) \to \Grp_{\bE_\infty}(\Alg_{\Lie(1)}(\mC))$ inverse to the forgetful functor
$\Grp_{\bE_\infty}(\Alg_{\Lie(1)}(\mC)) \to \mC$.
Let $\phi$ be the composition $$\mC \xrightarrow{\Grp_{\bE_\infty}(\triv_{\Lie(1)})} \Grp_{\bE_\infty}(\Alg_{\Lie(1)}(\mC)) \xrightarrow{\Grp_{\bE_\infty}(\triv\circ_{\Lie(1)}(-))}$$$$ \Hopf^{\mathrm{dp}, \conil}_{\bE_\infty}(\mC)= \Grp_{\bE_\infty}(\co\Alg^{\mathrm{dp}, \conil}_{\bE_\infty}(\mC)).$$ 

By Remark \ref{contin} the composition $\phi': \mC \xrightarrow{\phi} \Hopf^{\mathrm{dp}, \conil}_{\bE_\infty}(\mC) \xrightarrow{\text{forget}} \Hopf_{\bE_\infty}(\mC) \xrightarrow{\text{forget}} \Alg_{\bE_\infty}(\mC)$
lifts the functor $$ \mC \to \mC : \X \mapsto \coprod_{\n \geq 0} (\X^{\ot\n})_{\Sigma_\n}$$ along the forgetful functor
$\V: \Alg_{\bE_\infty}(\mC) \to \mC$ whose left adjoint we denote by $\Sym.$
Mapping to the first summand defines a natural transformation
$\id \to \V \circ \phi'$ of endofunctors of $\mC$ that gives rise to a map $\Sym = \Sym \circ \id \to \Sym \circ \V \circ \phi' \to \phi' $
of functors $\mC \to \Alg_{\bE_\infty}(\mC)$ that is component-wise the canonical equivalence.

\end{proof}

\subsection{The classical Milnor-Moore theorem}

Next we deduce the classical Milnor-Moore theorem from Theorem \ref{thg}.
This gives an independent proof of the classical Milnor-Moore theorem.

Let $\K$ be a field of characteristic zero.
Theorem \ref{thg} applied to $\mC=\Mod_{\rH(\K)}$ the derived $\infty$-category of $\K$
states that the enveloping Hopf algebra functor \begin{equation}\label{zzpp}
\mU: \Alg_\Lie(\Mod_{\rH(\K)}) \rightleftarrows \Hopf(\Mod_{\rH(\K)}):\Prim \end{equation} is fully faithful. 

\begin{proposition}\label{oyok} Let $\K$ be a field of characteristic zero.
	
\begin{enumerate}
\item The functor $$ \mU: \Alg_\Lie(\Mod_{\rH(\K)}) \to \Hopf(\Mod_{\rH(\K)})$$
sends Lie algebra structures on connective ${\rH(\K)}$-modules to Hopf algebra structures on connective ${\rH(\K)}$-modules.
		
\item The functor $$ \mU: \Alg_\Lie(\Mod_{\rH(\K)}) \to \Hopf(\Mod_{\rH(\K)})$$
sends Lie algebra structures on $\K$-vector spaces to Hopf algebra structures on $\K$-vector spaces. 
		
\item The functor $$ \Prim_{\geq0}: \Hopf(\Mod_{\rH(\K)})_{\geq0}\to \Alg_\Lie(\Mod_{\rH(\K)})_{\geq0} $$
sends Hopf algebra structures on $\K$-vector spaces to Lie algebra structures on $\K$-vector spaces, and agrees with the classical functor of primitive elements on such.
		
\end{enumerate}	
	
\end{proposition} 

\begin{proof}
	
(1): The forgetful functor $ \Alg_\Lie(\Mod_{\rH(\K)})_{\geq 0} \to (\Mod_{\rH(\K)})_{\geq 0}$ is the pullback of the monadic forgetful functor $ \Alg_\Lie(\Mod_{\rH(\K)}) \to \Mod_{\rH(\K)}$.
Since the spectral $\infty$-operad Lie is an $\infty$-operad in connective spectra, the free Lie algebra functor $$\Mod_{\rH(\K)} \to \Alg_\Lie(\Mod_{\rH(\K)}), \X \mapsto \bigoplus_{\n \geq 0}(\Lie_\n \ot \X^{\ot\n})_{\Sigma_\n}$$ sends connective ${\rH(\K)}$-modules to Lie structures on connective ${\rH(\K)}$-modules. This implies that the forgetful functor $\nu: \Alg_\Lie(\Mod_{\rH(\K)})_{\geq 0} \to (\Mod_{\rH(\K)})_{\geq 0}$ admits a left adjoint and so is monadic (as a pullback of a monadic functor). Consequently, $ \Alg_\Lie(\Mod_{\rH(\K)})_{\geq 0}$ is generated by free Lie algebras on connective ${\rH(\K)}$-modules under geometric realizations of $\nu$-split simplicial objects,
whose colimit is preserved by the embedding $\Alg_\Lie(\Mod_{\rH(\K)})_{\geq 0} \subset \Alg_\Lie(\Mod_{\rH(\K)}).$
Since $(\Mod_{\rH(\K)})_{\geq 0} \subset \Mod_{\rH(\K)}$ is closed under colimits, it is enough to see that $ \mU: \Alg_\Lie(\Mod_{\rH(\K)}) \to \Hopf(\Mod_{\rH(\K)})$
sends free Lie algebras generated by connective ${\rH(\K)}$-modules to Hopf algebras in connective ${\rH(\K)}$-modules. But by adjointess $\mU$ sends the free Lie algebra on any ${\rH(\K)}$-module $\X$ to the free associative algebra on $\X$, which is connective if $\X$ is connective.
	
(3): For (3) it is enough to prove that the functor \begin{equation}\label{wawa}
\Prim_{\geq0}: (\co\Coalg(\Mod_{\rH(\K)})_{\geq0})_{\tu/} \simeq (\co\Coalg(\Mod_{\rH(\K)})_{\tu}/)_{\geq0} \to (\Mod_{\rH(\K)})_{\geq0}\end{equation}
sends augmented cocommutative coalgebra structures on $\K$-vector spaces to Lie algebras on $\K$-vector spaces.
The functor (\ref{wawa}) is right adjoint to the trivial cocommutative coalgebra functor $(\Mod_{\rH(\K)})_{\geq0} \to (\co\Coalg(\Mod_{\rH(\K)})_{\tu/})_{\geq0}$.
The classical functor of primitive elements $ \co\Coalg(\Vect_\K)_{\tu/} \to \Vect_{\K}$ is right adjoint to the trivial cocommutative coalgebra functor $\Vect_\K \to \co\Coalg(\Vect_\K)_{\tu/}$.
By K\"unneth the embedding $\Vect_\K \hookrightarrow  (\Mod_{\rH(\K)})_{\geq0}$ is symmetric monoidal and admits a symmetric monoidal left adjoint $\pi_0: (\Mod_{\rH(\K)})_{\geq0} \to \Vect_\K$. Hence we obtain an induced adjunction $$\pi_0: \co\Coalg((\Mod_{\rH(\K)})_{\geq0})_{\tu/} \rightleftarrows \co\Coalg(\Vect_\K)_{\tu/}$$ whose right adjoint is fully faithful.
There is a commutative square
\begin{equation*}
\begin{xy}
\xymatrix{(\Mod_{\rH(\K)})_{\geq0}\ar[d]^{\pi_0}
\ar[rr]^{\triv}
&& \co\Coalg((\Mod_{\rH(\K)})_{\geq0})_{\tu/} \ar[d]^{\pi_0}
\\
\Vect_\K  \ar[rr]^{\triv} &&  \co\Coalg(\Vect_\K)_{\tu/}.
}
\end{xy} 
\end{equation*} 
So (3) follows from the latter commutative square by passing to right adjoints.
	
(2): Let $\Ass^{\mathrm{nu}}$ be the non-unital spectral $\Ass$-operad, which is the image of the non-unital $\Ass$-operad in spaces and so a spectral Hopf $\infty$-operad. Let $\mathrm{co}\Ass^{\mathrm{nu}}$ be the Spanier-Whitehead dual of $\Ass^{\mathrm{nu}} $, which is a non-co-unital spectral Hopf $\infty$-co-operad.
The unique map $\coComm^{\mathrm{nu}} \to \mathrm{co}\Ass^{\mathrm{nu}}$ of spectral Hopf $\infty$-cooperads induces on shifted Koszul duality a map $\theta: \Lie \to \Ass^{\mathrm{nu}}$ of spectral $\infty$-operads, where we use that $\Ass^{\mathrm{nu}}(1)$ is Koszul dual to $\mathrm{co}\Ass^{\mathrm{nu}}.$
Let $\mC$ be a stable presentably symmetric monoidal $\infty$-category. By Theorem \ref{thh}
and Proposition \ref{fghfghjok} there is a commutative square
$$
\begin{xy}
\xymatrix{
\Alg_{\Lie}(\mC) \ar[r]^\simeq \ar[d]^{\theta_!} & \Alg_{\Lie(1)}(\mC) \ar[d]^{\theta(1)_!} 
\ar[rr]^{\mathrm{TQ}_{\Lie_\mC(1)}}
&& \co\Alg_{\coComm^{\mathrm{nu}}}(\mC) \ar[d] \ar[r]^\simeq & \co\Coalg(\mC)_{\tu/}\ar[d]
\\
\Alg_{\Ass^{\mathrm{nu}}}(\mC) \ar[r]^\simeq & \Alg_{\Ass^{\mathrm{nu}}(1)}(\mC) \ar[rr]^{\mathrm{TQ}_{\Ass_\mC^{\mathrm{nu}}(1)  }} && \co\Alg_{\mathrm{co}\Ass^{\mathrm{nu}}}(\mC) \ar[r]^\simeq & \co\Alg(\mC)_{\tu/}.
}
\end{xy} 
$$
The bottom functor $$ \Alg(\mC)_{/\tu} \simeq \Alg_{\Ass^{\mathrm{nu}}}(\mC)\simeq \Alg_{\Ass^{\mathrm{nu}}(1)}(\mC) \xrightarrow{\mathrm{TQ}_{\Ass_\mC^{\mathrm{nu}}(1) } } \co\Alg_{\mathrm{co}\Ass^{\mathrm{nu}}}(\mC) \simeq \co\Alg(\mC)_{\tu/}$$
identifies with the Bar-construction.
Consequently, by construction of $\mU$ for $\mC=\Mod_{\rH(\K)} $ to prove (2) it is enough to see that the composition $$\rho: \Alg_{\Lie}(\mC) \xrightarrow{\theta_!} \Alg(\mC)_{/\tu} \xrightarrow{\mathrm{Bar}}\co\Alg(\mC)_{\tu/} \xrightarrow{\text{forget}} \mC$$
sends Lie structures on connective $\rH(\K)$-modules to connective $\rH(\K)$-modules. Let $\nu:\Alg(\mC)_{/\tu} \to \mC$ be the forgetful functor. 
We will prove that the functor $\theta_!: \Alg_{\Lie}(\mC) \to \Alg(\mC)_{/\tu} $
is symmetric monoidal when the source carries the cartesian structure and the target carries the object-wise symmetric monoidal structure. This will guarantee that $\nu \circ \theta_!:  \Alg_{\Lie}(\mC) \to\mC$ induces a functor
$(\nu \circ \theta_!)_*: \Alg_\Lie(\mC)\simeq \Grp(\Alg_{\Lie}(\mC)) \to \Alg(\mC)_{/\tu}$. By Eckmann-Hilton the functor 
$\rho$ will therefore factor as
$$\Alg_{\Lie}(\mC) \simeq \Grp(\Alg_{\Lie}(\mC)) \xrightarrow{(\nu \circ \theta_!)_*} \Alg(\mC)_{/\tu} \xrightarrow{\mathrm{Bar}}\co\Alg(\mC)_{\tu/}\xrightarrow{\text{forget}} \mC,$$
which is the functor
$\nu \circ \theta_!: \Alg_{\Lie}(\mC) \xrightarrow{}\mC$
since $\nu \circ \theta_!$ is symmetric monoidal and preserves sifted colimits and therefore preserves the Bar-construction.
So it will be enough to see that the functor $\theta_!: \Alg_{\Lie}(\mC) \to \Alg(\mC)_{/\tu} $ is symmetric monoidal and sends Lie structures on $\K$-vector spaces to associative algebra structures on $\K$-vector spaces.
The canonical map $\sigma: \Lie_\K \to \Ass^{\mathrm{nu}}_\K$ of operads in chain complexes over $\K$
from the classical Lie operad over $\K$ to the non-unital associative operad over $\K$ is sent by the universal functor $\Ch_\K \to \Mod_{\rH(\K)} $ to $\rH(\K) \smash \theta.$
The map $\sigma$ gives rise to a right adjoint forgetful functor
$$\G:  \Alg(\Ch_\K)_{/\tu} \simeq \Alg_{\Ass_\K^{\mathrm{nu}}}(\Ch_\K) \to \Alg_{\Lie_\K}(\Ch_\K).$$
We write $\mU': \Alg_{\Lie_\K}(\Ch_\K)\to \Alg(\Ch_\K)_{/\tu}$ for the left adjoint of $\G$.
The functor $\mU'$ is the classical enveloping algebra that sends $\L$ to the quotient
of $ \bigoplus_{\n \geq0}\L^{\ot\n}$ by the left and right ideal generated by the elements $[\A,\B]-\A\B-\B\A$ for $\A,\B \in \L.$
This description of $\mU'$ guarantees that $\mU'$ preserves quasi-isomorphisms since $\K$ is a field.
The functor $\mU'$ is known to send the product to the tensor product.
Thus $\mU'$ induces a symmetric monoidal functor \begin{equation}\label{ahal}
\Alg_{\Lie_\K}(\Ch_\K)[\text{quasi-isos}]\to \Alg(\Ch_\K)_{/\tu}[\text{quasi-isos}]\end{equation} on localizations left adjoint to the functor
\begin{equation}\label{uuup}\Alg(\Ch_\K)_{/\tu}[\text{quasi-isos}] \to \Alg_{\Lie_\K}(\Ch_\K)[\text{quasi-isos}]\end{equation} induced by the forgetful functor $\Alg(\Ch_\K)_{/\tu} \to \Alg_{\Lie_\K}(\Ch_\K)$.
By rectification the functor (\ref{uuup}) is equivalent to the forgetful functor
$$\Alg(\Mod_{\rH(\K)})_{/\tu} \to \Alg_{\Lie}(\Mod_{\rH(\K)})$$
so that by adjointness the functor (\ref{ahal}) identifies with the functor $\theta_!.$
The functor (\ref{ahal}) fits into a commutative square:
\begin{equation*}
\begin{xy}
\xymatrix{\Alg_\Lie(\Vect_\K) \ar[d]
\ar[rr]
&& \Alg(\Vect_\K)_{/\tu} \ar[d]
\\
\Alg_\Lie(\Ch_\K)[\text{quasi-isos}] \ar[rr] && \Alg(\Ch_\K)_{/\tu}[\text{quasi-isos}].
			}
\end{xy} 
\end{equation*} 
So the claim follows.
	
\end{proof}

We obtain the classical Milnor-Moore theorem as a corollary:

\begin{theorem}\label{Mil} Let $\K$ be a field of characteristic zero.
The left adjoint of the adjunction $$ \mU: \Alg_\Lie(\Vect_\K) \rightleftarrows \Hopf(\Vect_\K): \Prim_{\geq 0}$$ is fully faithful and the essential image of $\mU$ precisely consists of the primitively generated Hopf algebras.
\end{theorem}

\begin{proof}
By Proposition \ref{oyok} (1) the adjunction $$\mU: \Alg_\Lie(\Mod_{\rH(\K)}) \rightleftarrows \Hopf(\Mod_{\rH(\K)}):\Prim$$
induces an adjunction \begin{equation}\label{zzppp2}
\mU: \Alg_\Lie(\Mod_{\rH(\K)})_{\geq 0} \rightleftarrows \Hopf(\Mod_{\rH(\K)})_{\geq 0}: \Prim_{\geq 0}, \end{equation} where $\Prim_{\geq 0} $ takes the connective cover of the primitives.
By Proposition \ref{oyok} (2) and (3) adjunction (\ref{zzppp2}) restricts to an adjunction $$ \mU: \Alg_\Lie(\Vect_\K) \rightleftarrows \Hopf(\Vect_\K): \Prim_{\geq 0}$$	
whose right adjoint agrees with the classical functor of primitive elements.
By uniqueness of adjoints the left adjoint agrees with the classical enveloping Hopf algebra and the counit $\mU \Prim_{\geq 0} \to \id $ identifies with the canonical map.
In particular, a Hopf algebra $\rH$ over $\K$ 
belongs to the essential image of the left adjoint if and only if
the map $\mU\Prim(\rH) \to \rH$ is an isomorphism, i.e. if and only if 
$\rH$ is generated by its primitives as an algebra.

\end{proof}

\subsection{Higher chromatic heights}\label{scomp}

Next we consider an application to chromatic homotopy theory.
We fix the following terminology:

Let $p$ be a prime and $n \geq 1.$ We consider spaces and spectra localized at $p$, i.e.
local for the homology theory associated to the $p$-local sphere spectrum.
Let $ \Sp_{(p)} \subset \Sp$ the full reflexive subcategory of $p$-local spectra
and $ \mS_{(p)} \subset \mS_*$ the full reflexive subcategory of $p$-local pointed spaces.

By \cite[Theorem 2.2.]{2018arXiv180306325H} there are respective reflexive full subcategories $$\Sp_{T_n} \subset \Sp_{(p)}, \  \mS_{v_n} \subset \mS_{(p)} $$ of $T_n$-local spectra and $v_n$-periodic homotopy types.
By \cite{Bousfield1}, \cite{Kuhntelescopic} there is a functor $\Phi: \mS_{(p)} \to \Sp_{T_n}$, the Bousfield-Kuhn functor, such that 
\begin{enumerate}
\item The localization functor $\Sp_{(p)} \to \Sp_{T_n}$ factors as $\Phi \circ \Omega^\infty : \Sp_{(p)} \to \Sp_{T_n}$.

\item The localization functor
$\mS_{(p)} \to \mS_{v_n} $ universally inverts the $\Phi$-equivalences, which are known as $v_n$-periodic equivalences.
	
\end{enumerate}

%
%
%

Heuts \cite{Heuts} proves that the Bousfield-Kuhn functor $\Phi: \mS_{v_n} \to \Sp_{T_n}$ lifts to an equivalence $ \mS_{v_n} \to \Alg_{\Lie(1)}(\Sp_{T_n})$.
Composing with the equivalence $\Alg_{\Lie(1)}(\Sp_{T_n}) \simeq \Alg_{\Lie}(\Sp_{T_n})$ that loops the underlying spectrum, one obtains an equivalence
$\Phi[-1] : \mS_{v_n} \to \Alg_{\Lie}(\Sp_{T_n})$.

We deduce the following chromatic variant of the Milnor-Moore theorem: 

\begin{theorem}\label{20}
	
Let $n \geq 1$ be a natural and $X$ a $v_n$-periodic homotopy type.
The unit $$\Phi(X)[-1] \to \Prim \mU(\Phi(X)[-1])$$ identifies with the Goodwillie completion
$ \Phi \to \lim_{n \geq 0} P_n(\Phi)$ evaluated at the loops $\Omega(X)$.
	
\end{theorem}

\begin{proof}
	
By \cite[Theorem 2.6.]{2018arXiv180306325H} the Bousfield-Kuhn functor $\Phi: \mS_{v_n} \to \Sp_{T_n}$ lifts to an equivalence $\bar{\Phi} : \mS_{v_n} \to \Alg_{\Lie(1)}(\Sp_{T_n})$.
In particular, $\Phi$ preserves limits and small filtered colimits. By \cite[Remark 6.1.1.32.]{lurie.higheralgebra} this implies that the induced functor $$\Phi_*: \Fun( \mS_{v_n}, \mS_{v_n}) \to \Fun( \mS_{v_n}, \Sp_{T_n}) $$
sends the Goodwillie completion 
$ \id_{\mS_{v_n}} \to \lim_{n \geq 0} P_n(\id_{\mS_{v_n}}) $ of functors $ \mS_{v_n} \to \mS_{v_n}$
to the Goodwillie completion
$\Phi \to \lim_{n \geq 0} P_n(\Phi)$ of functors $ \mS_{v_n} \to \Sp_{T_n} $.


By  \cite[Theorem 11.3.]{pereira2013goodwillie} the  Goodwillie completion of the identity of $\Alg_{\Lie(1)}(\Sp)$ is
$$ \id \to \lim_{n \geq 0} \tau_n(\Lie(1)) \circ_{\Lie(1)}(-). $$
See also \cite[Theorem 2.20.]{kuhn2017operad}.
By  stability the localization $L: \Sp \to \Sp_{T_n}$ preserves finite limits.
By  \cite[Remark 6.1.1.32.]{lurie.higheralgebra} this implies that the  Goodwillie completion of the identity of $\Alg_{\Lie(1)}(\Sp_{T_n})$ is $$ \id \to \lim_{n \geq 0} \tau_n(\Lie(1)) \circ_{\Lie(1)}(-),$$
where the composition product and spectral Lie $\infty$-operad are in $\Sp_{T_n}.$ 

Since $\bar{\Phi}$ is an equivalence, by \cite[Remark 6.1.1.30.]{lurie.higheralgebra} the induced functor $$ \bar{\Phi}^*: \Fun(\Alg_{\Lie(1)}(\Sp_{T_n}), \Alg_{\Lie(1)}(\Sp_{T_n})) \to \Fun( \mS_{v_n}, \Alg_{\Lie(1)}(\Sp_{T_n})) $$ sends the Goodwillie completion 
$$ \id \to \lim_{n \geq 0} \tau_n(\Lie(1)) \circ_{\Lie(1)}(-)$$ to the Goodwillie completion 
$$ \bar{\Phi} \to \lim_{n \geq 0} \tau_n(\Lie(1)) \circ_{\Lie(1)}  (-) \circ \bar{\Phi} $$ of functors $ \mS_{v_n} \to \Alg_{\Lie(1)}(\Sp_{T_n})$.

Since the forgetful functor $\nu: \Alg_{\Lie(1)}(\Sp_{T_n}) \to \Sp_{T_n}$ preserves limits and small filtered colimits, by \cite[Remark 6.1.1.32.]{lurie.higheralgebra} it sends the Goodwillie completion 
$$ \bar{\Phi} \to \lim_{n \geq 0} \tau_n(\Lie(1)) \circ_{\Lie(1)}  (-) \circ \bar{\Phi}$$ of functors $ \mS_{v_n} \to \Alg_{\Lie(1)}(\Sp_{T_n})$ to the Goodwillie completion 
$$ \Phi \to \lim_{n \geq 0} \nu \circ \tau_n(\Lie(1)) \circ_{\Lie(1)}  (-) \circ \bar{\Phi}$$ of functors $ \mS_{v_n} \to \Sp_{T_n}$.
Consequently, by uniqueness of Goodwillie completions the Goodwillie completion 
$$ \Phi \to \lim_{n \geq 0} \nu \circ \tau_n(\Lie(1)) \circ_{\Lie(1)}  (-) \circ \bar{\Phi}$$ of functors $ \mS_{v_n} \to \Sp_{T_n}$ identifies with $\Phi \to \lim_{n \geq 0} P_n(\Phi)$.

Hence for every $X \in  \mS_{v_n} $ the canonical map $\Phi(\Omega(X)) \to \lim_{n \geq 0} P_n(\Phi)(\Omega(X))$  in $\Sp_{T_n}$ 
identifies with the map 
$$ \Phi(\Omega(X)) \simeq \Omega(\Phi(X)) \to \lim_{n \geq 0} \tau_n(\Lie(1)) \circ_{\Lie(1)} \bar{\Phi}(\Omega(X)) \simeq \lim_{n \geq 0} \tau_n(\Lie(1)) \circ_{\Lie(1)} \Omega(\bar{\Phi}(X)).$$

By Theorem \ref{map2} the latter is the unit $ \Omega(\Phi(X)) \to \Prim \mU(\Omega(\Phi(X))).$ 


\end{proof}

\begin{remark}Let $n$ be a natural .
A $v_n$-periodic homotopy type $X$ is called $\Phi$-good if the Goodwillie completion map $ \Phi \to \lim_{n \geq 0} P_n(\Phi)$ evaluated at $X$ is an equivalence.

Let $X$ be a $v_n$-periodic homotopy type. By Theorem \ref{2} 	the unit $$\Phi(X)[-1] \to \Prim \mU(\Phi(X)[-1])$$ is an equivalence if and only if $\Omega(X)$ is $\Phi$-good.

It is expected that any $v_n$-periodic homotopy type $X$ is $\Phi$-good (see \cite{behrens2024unstable} last paragraph).
This implies that the unit $$\Phi(X)[-1] \to \Prim \mU(\Phi(X)[-1])$$ is an equivalence for every
$v_n$-periodic homotopy type $X$.
By \cite[Theorem 2.6.]{2018arXiv180306325H} the Bousfield-Kuhn functor $\Phi: \mS_{v_n} \to \Sp_{T_n}$ lifts to an equivalence $\bar{\Phi} : \mS_{v_n} \to \Alg_{\Lie(1)}(\Sp_{T_n})$.
Thus also the functor $\bar{\Phi}[-1] : \mS_{v_n} \to \Alg_{\Lie(1)}(\Sp_{T_n}) \simeq \Alg_{\Lie(\Sp_{T_n}}) $ is an equivalence.
This implies that the unit $\id \to \Prim \mU$ is an equivalence if the expectation holds.

\end{remark}

\begin{remark}
Let $n$ be a natural and $X$ a $v_n$-periodic homotopy type.
The infinite suspension $ L_{T_n} \Sigma^\infty: \mS_* \to \Sp_{T_n}$ is symmetric monoidal for the cartesian product and $T_n$-local smash product and so induces a functor
$\Grp(\mS) \simeq \Hopf(\mS) \simeq \Hopf(\mS_*) \to \Hopf(\mS_{T_n}) .$ 
In particular, it sends the loop space $\Omega(X)$ to a Hopf algebra 
$L_{T_n} \Sigma^\infty(\Omega(X))$ in $T_n$-local spectra.

By \cite[Theorem 2.11., Lemma 5.15.]{2018arXiv180306325H} there is a canonical equivalence of Lie algebras in $T_n$-local spectra:
$$ \lim_{n \geq 0} P_n(\Phi)(\Omega(X)) \simeq \Prim L_{T_n} \Sigma^\infty(\Omega(X)). $$
Together with Theorem \ref{map2} we obtain a canonical equivalence of Lie algebras in $T_n$-local spectra:
$$ \Prim \mU(\Phi(X)[-1]) \simeq \Prim L_{T_n} \Sigma^\infty(\Omega(X)). $$

\end{remark}

\section{Appendix}

\subsection{Rational stable $\infty$-categories}

\begin{definition}
A stable presentable $\infty$-category is rational (or $\bQ$-linear) if the hom sets of the homotopy category are $\bQ$-vector spaces.
	
\end{definition}

\begin{notation}

Let $ \Pr^\L_{{\mathrm{st}}, \bQ} \subset \Pr^\L$ be the full subcategory spanned by the stable presentable $\infty$-categories such that the hom sets of the homotopy category are $\bQ$-vector spaces.
\end{notation}

\begin{lemma}\label{Prese}
The forgetful functor
$ \Mod_{\Mod_{\rH(\bQ)}(\Sp)}(\Pr^\L) \to \Pr^\L $
induces an equivalence
$$ \Mod_{\Mod_{\rH(\bQ)}(\Sp)}(\Pr^\L) \to \Pr^\L_{{\mathrm{st}}, \bQ}.$$

\end{lemma}

\begin{proof}
Since $\Mod_{\rH(\bQ)}(\Sp)$ is stable, by \cite[Theorem 4.6.]{2013arXiv1305.4550G} the canonical map
$$ \Mod_{\rH(\bQ)}(\Sp) \simeq \mS \ot \Mod_{\rH(\bQ)}(\Sp) \to \Sp \ot\Mod_{\rH(\bQ)}(\Sp)$$ is an equivalence.
The resulting map $$\Mod_{\rH(\bQ)}(\Sp) \simeq \mS \ot \Mod_{\rH(\bQ)}(\Sp) \to \Mod_{\rH(\bQ)}(\Sp)\ot \Mod_{\rH(\bQ)}(\Sp) \simeq $$$$
\Mod_{\rH(\bQ)}(\Sp)\ot_{\Sp} (\Sp \ot \Mod_{\rH(\bQ)}(\Sp)) \simeq \Mod_{\rH(\bQ)}(\Sp)\ot_{\Sp} \Mod_{\rH(\bQ)}(\Sp)$$$$
\simeq \Mod_{\rH(\bQ) \ot \rH(\bQ)}(\Sp) $$
identifies with induction along the canonical map of spectra
$ \rH(\bQ) \simeq \tu \ot \rH(\bQ) \to \rH(\bQ) \ot \rH(\bQ),$
which is an equivalence.
By \cite[Proposition 4.8.2.10.]{lurie.higheralgebra} this guarantees that the forgetful functor $ \Mod_{\Mod_{\rH(\bQ)}(\Sp)}(\Pr^\L) \to \Pr^\L$
is fully faithful. We characterize the essential image.

Every left module in $\Pr^\L$ over $\Mod_{\rH(\bQ)}(\Sp)$ has an underlying left module over $\Sp$ forgetting along the symmetric monoidal left adjoint functor $\Sp \to \Mod_{\rH(\bQ)}(\Sp)$.
By \cite[Corollary 9.4.]{Heineenrichednew} every left module over $\Sp$ in $\Pr^\L$ is a stable $\infty$-category, and for every stable presentable $\infty$-category $\mB$ the canonical map $\mB \simeq \mS \ot \mB \to \Sp \ot \mB$ is an equivalence.
Consequently, it is enough to see that for any stable presentable
$\infty$-category $\mC$ the following holds: the hom sets of the homotopy category of $\mC$ are $\bQ$-vector spaces if and only if the canonical map
$ \mC \simeq \mS \ot \mC \to \Mod_{\rH(\bQ)}(\Sp) \ot \mC$
is an equivalence. The latter map identifies with the functor
$$ \mC \simeq \mS \ot \mC \to \Mod_{\rH(\bQ)}(\Sp) \ot \mC \simeq\Mod_{\rH(\bQ)}(\Sp) \ot_{\Sp} (\Sp \ot \mC) \simeq $$$$\Mod_{\rH(\bQ)}(\Sp) \ot_{\Sp} \mC \simeq \Mod_{\rH(\bQ)}(\mC),$$
where the last equivalence is by \cite[Theorem 4.8.4.6.]{lurie.higheralgebra}.
The latter functor identifies with the free $\rH(\bQ)$-module functor.
Since the canonical map of spectra
$ \rH(\bQ) \simeq \tu \ot \rH(\bQ) \to \rH(\bQ) \ot \rH(\bQ)$
is an equivalence, by \cite[Proposition 4.8.2.10.]{lurie.higheralgebra} the forgetful functor $\Mod_{\rH(\bQ)}(\mC) \to \mC$ is fully faithful and the essential image consists of those $\X \in \mC$ such that one of the following equivalent conditions hold:

\begin{enumerate}
\item The map $\X \simeq \tu \ot \X \to \rH(\bQ) \ot \X$ is an equivalence.

\item For every $\Z \in \mC$ the map $\tu \ot \Z \to \rH(\bQ) \ot \Z$ induces an equivalence $$\mC(\rH(\bQ) \ot \Z, \X) \to \mC(\tu \ot \Z,\X).$$
\end{enumerate}

We will prove that every $\X \in \mC$ satisfies one of the equivalent conditions (1) or (2) if and only if for all $\Y \in \mC$
the group $\pi_0(\mC(\X,\Y)) $ is a $\bQ$-vector space.
The map in (1) is an equivalence if and only if for all $\Z \in \Mod_{\rH(\bQ)}(\Sp)$
the map $\Z \ot \tu \ot \X \to \Z \ot \rH(\bQ) \ot \X$ in $\mC$ is an equivalence.
By Yoneda this is equivalent to ask that for all $\Z \in \Mod_{\rH(\bQ)}(\Sp), \Y \in \mC$ the canonical map $$\Sp(\rH(\bQ) \ot \Z, \map_\mC(\X, \Y)) \to \Sp(\tu \ot \Z, \map_\mC(\X, \Y))$$ is an equivalence, where $\map_\mC(-,-)$ denotes the $\Sp$-enrichment.
By the equivalence of conditions (1) and (2) applied to $\mC=\Sp $ the latter condition is equivalent to say that for all $\Y \in \mC$ the canonical map  $$\map_\mC(\X, \Y) \simeq \tu \ot \map_\mC(\X, \Y) \to \rH(\bQ) \ot \map_\mC(\X, \Y)$$ is an equivalence.
This can be reprased by saying that 
the map $$\pi_*(\map_\mC(\X, \Y)) \to \pi_*(\rH(\bQ) \ot \map_\mC(\X, \Y)) \cong \bQ \ot \pi_*(\map_\mC(\X, \Y)) $$ is an isomorphism
or that for all $\ell \geq 0$ and $\Y \in \mC$ the group $ \pi_\ell(\map_\mC(\X, \Y))$ is a $\bQ$-vector space.
By stability this is equivalent to say that for all $\Y \in \mC$ the group $ \pi_0(\mC(\X, \Y))$ is a $\bQ$-vector space.

\end{proof}

\subsection{$\mO$-monoidal closures}

For the next notation we use that for every small $\mO$-monoidal $\infty$-category $\mC^\ot \to \mO^\ot$ there is an $\mO$-monoidal Yoneda-embedding $\mC^\ot \to \mP(\mC)^\ot$
\cite[Proposition 4.8.1.10. ]{lurie.higheralgebra}.

\begin{notation}\label{notorf}
Let $\mK, \mK' \subset \Cat_\infty$ be full subcategories and 
$\mC^\ot \to \mO^\ot$ an $\mO$-monoidal $\infty$-category.
Let $$\widehat{\mP}_\mK(\mC)^\ot \subset \widehat{\mP}(\mC)^\ot$$ be the full suboperad spanned by the functors $\mC_\X^\op \to \widehat{\mS} $ preserving $\mK$-indexed limits for some $\X \in \mO$.
Let $$\mP^{\mK'}_\mK(\mC)^\ot \subset \widehat{\mP}_\mK(\mC)^\ot $$ be the full suboperad such that for every $\X \in \mO$ the full subcategory 
$\mP^{\mK'}_\mK(\mC)_\X \subset \widehat{\mP}_\mK(\mC)_\X$ is generated by $\mC_\X$ under $\mK'$-indexed colimits.
\end{notation}

\begin{example}
If $\mK= \Cat_\infty$, we have $\mP^{\mK'}_\mK(\mC)^\ot = \mC^\ot.$	
\end{example}

\begin{remark}\label{reno}
The $\mO$-monoidal Yoneda-embedding $\mC^\ot \subset \widehat{\mP}(\mC)^\ot$
induces $\mO$-monoidal embeddings $\mC^\ot \subset \widehat{\mP}_\mK(\mC)^\ot, \mC^\ot \subset \mP^{\mK'}_\mK(\mC)^\ot, $
where the latter two embeddings preserve fiberwise $\mK$-indexed colimits.
\end{remark}

\begin{lemma}\label{lomos}

Let $\mK, \mK' \subset \Cat_\infty$ be full subcategories and 
$\mC^\ot \to \mO^\ot$ an $\mO$-monoidal $\infty$-category.

\begin{enumerate}
\item Let $\X \in \mO$. If the fiber $\mC_\X$ has $\mK$-indexed colimits,
$ \widehat{\mP}_\mK(\mC)_\X^\ot \subset \widehat{\mP}(\mC)^\ot_\X$ is a localization.
In particular, $ \widehat{\mP}_\mK(\mC)^\ot_\X$ has large colimits and so $ \mP^{\mK'}_\mK(\mC)^\ot_\X $ has $\mK'$-indexed colimits.	

\item If the $\mO$-monoidal $\infty$-category $\mC^\ot \to \mO^\ot$ is compatible with $\mK$-indexed colimits, $ \widehat{\mP}_\mK(\mC)^\ot \subset \widehat{\mP}(\mC)^\ot$ is a localization relative to $\mO^\ot$.
In particular, $ \widehat{\mP}_\mK(\mC)^\ot \to \mO^\ot$ is an $\mO$-monoidal $\infty$-category compatible with large colimits and so $ \mP^{\mK'}_\mK(\mC)^\ot \to \mO^\ot $ is an $\mO$-monoidal $\infty$-category compatible with $\mK'$-indexed colimits.	
\end{enumerate}	
\end{lemma}

\begin{lemma}\label{lomosay}

Let $\mK, \mK' \subset \Cat_\infty$ be full subcategories and 
$\mC^\ot \to \mO^\ot$ an $\mO$-monoidal $\infty$-category.
Let $\X \in \mO$. If the fiber $\mC_\X$ has $\mK'$-indexed colimits, $\mC_\X^\ot \subset \mP^{\mK'}_\mK(\mC)_\X^\ot$ is a localization.

\end{lemma}

\begin{proof}

First assume that $\mC_\X$ has small colimits. Then by Lemma \ref{lomos} (1) the embeddings $ \widehat{\mP}_{\Cat_\infty}(\mC)_\X \subset \widehat{\mP}(\mC)_\X, \widehat{\mP}_\mK(\mC)_\X \subset \widehat{\mP}(\mC)_\X$ 
admit a left adjoint.
Hence the embedding $ \widehat{\mP}_{\Cat_\infty}(\mC)_\X \subset \widehat{\mP}_\mK(\mC)_\X$ admits a left adjoint.
Thus the embedding $ \mC_\X = \mP^{\mK'}_{\Cat_\infty}(\mC)_\X \subset \mP^{\mK'}_\mK(\mC)_\X$ admits a left adjoint.

In the general case we set $\mD^\ot:=\mP^{\Cat_\infty}_{\mK'}(\mC)^\ot$.
By Lemma \ref{lomos} (1) the $\infty$-category $\mD_\X$ admits small colimits.
Moreover the embedding $\mC_\X \subset \mD_\X$ preserves $\mK'$-indexed colimits.
The first part of the proof shows that the embedding $ \mD_\X \subset \mP^{\mK'}_\mK(\mD)_\X$ admits a left adjoint.
The left adjoint restricts to a functor $\mP^{\mK'}_\mK(\mC)_\X \to \mC_\X$
since $\mC_\X$ is closed in $\mD_\X$ under $\mK'$-indexed colimits.
Hence the embedding $ \mC_\X \subset \mP^{\mK'}_\mK(\mC)_\X$ admits a left adjoint.

\end{proof}

\begin{proposition}\label{embe}
	
Let $\mK, \mK' \subset \Cat_\infty$ be full subcategories and 
$\mC^\ot \to \mO^\ot$ an $\mO$-monoidal $\infty$-category compatible with $\mK$-indexed colimits that admits fiberwise $\mK'$-indexed colimits.
	
There are $\mO$-monoidal embeddings $\mC^\ot \subset \mD^\ot \subset \mE^\ot$ such that 
\begin{itemize}
\item $\mE^\ot \to \mO^\ot$ is compatible with large colimits,
\item $\mD^\ot \to \mO^\ot$ is compatible with $\mK'$-indexed colimits, 
\item for every $\X \in \mO$ the fiber $\mD_\X $ is generated in $\mE_\X$ by $\mC_\X$ under $\mK'$-indexed colimits,
\item the embedding $\mC_\X \subset \mD_\X $ admits a left adjoint and preserves $\mK$-indexed colimits,
\item the embedding $\mC_\X \subset \mE_\X $ preserves limits.
\end{itemize}
	
\end{proposition}

\begin{proof}[Proof of Proposition \ref{embe}]

We set $\mE^\ot:=\widehat{\mP}_\mK(\mC)^\ot$ and $\mD^\ot:= \mP^{\mK'}_\mK(\mC)^\ot$.

\end{proof}

\begin{notation}
	
Let $ \Cat_\infty^\mathrm{fin} \subset  \Cat_\infty$ be the full subcategory generated by the final $\infty$-category under finite colimits.
\end{notation}

\begin{corollary}\label{cory}

Let $\mC^\ot \to \mO^\ot$ be a pointed, preadditive, additive, stable $\mO$-monoidal $\infty$-category, respectively, that admits fiberwise $\mK'$-indexed colimits, where $\mK' \subset \Cat_\infty$ is a full subcategory containing $\{\emptyset\}, \Fin, \Cat_\infty^\mathrm{fin}$, respectively,
depending on the case.

There are pointed, preadditive, additive, stable $\mO$-monoidal $\infty$-categories $\mD^\ot \to \mO^\ot, \mE^\ot \to \mO^\ot$, respectively,
and pointed, preadditive, additive, exact $\mO$-monoidal embeddings $\mC^\ot \subset \mD^\ot, \mD^\ot \subset \mE^\ot$, respectively, such that 
\begin{itemize}
\item $\mE^\ot \to \mO^\ot$ is compatible with large colimits,

\item $\mD^\ot \to \mO^\ot$ is compatible with $\mK'$-indexed colimits, 

\item for every $\X \in \mO$ the fiber $\mD_\X $ is generated in $\mE_\X$ by $\mC_\X$ under $\mK'$-indexed colimits,

\item the embedding $\mC_\X \subset \mD_\X $ admits a left adjoint,

\item the embedding $\mC_\X \subset \mE_\X $ preserves limits.
\end{itemize}

\end{corollary}

\begin{proof}Let $\X \in \mO.$
For $\mK= \{\emptyset\} $ the $\infty$-category $ \widehat{\mP}_\mK(\mC)_\X = \Fun_*(\mC_\X^\op, \widehat{\mS})$ of final object preserving functors admits a zero object if $\mC$ admits a final object.
The zero object belongs to the full subcategory $ \mP^{\mK'}_\mK(\mC) $.
For $\mK= \Fin $ the $\infty$-category $ \widehat{\mP}_\mK(\mC)_\X = \Fun^\Pi(\mC_\X^\op, \widehat{\mS})$ of finite products preserving functors is (pre)additive if $\mC$ is (pre)additive \cite[Corollary 2.10.]{2013arXiv1305.4550G}. So the full subcategory $ \mP^{\mK'}_\mK(\mC) $ closed under finite coproducts is (pre)additive.
For $\mK= \Cat_\infty^\mathrm{fin}$ the $\infty$-category $ \widehat{\mP}_\mK(\mC) = \Fun^\mathrm{lex}(\mC^\op, \widehat{\mS})$ of finite limits preserving functors is stable if $\mC$ is stable. So the full subcategory $ \mP^{\mK'}_\mK(\mC) $ closed under finite colimits and arbitrary shifts is stable.	

\end{proof}

\end{document}